\documentclass[11pt]{amsart}

\usepackage[T1]{fontenc}
\usepackage{lmodern}
\usepackage{amsmath,amssymb,mathtools}
\usepackage{booktabs}
\usepackage{array}
\usepackage{longtable}
\usepackage{graphicx}
\usepackage[table]{xcolor}
\usepackage{float}
\usepackage{enumitem}
\usepackage[hidelinks]{hyperref}
\usepackage{microtype}
\usepackage{placeins}
\usepackage{tikz}
\usetikzlibrary{calc}

\makeatletter
\def\section{\@startsection{section}{1}%
  \z@{1.0\linespacing\@plus\linespacing}{.65\linespacing}%
  {\normalfont\Large\bfseries}}
\makeatother

\numberwithin{equation}{section}

\newtheorem{theorem}{Theorem}[section]
\newtheorem{proposition}[theorem]{Proposition}
\newtheorem{lemma}[theorem]{Lemma}
\newtheorem{corollary}[theorem]{Corollary}
\theoremstyle{definition}
\newtheorem{definition}[theorem]{Definition}
\newtheorem{example}[theorem]{Example}
\newtheorem{question}[theorem]{Question}
\theoremstyle{remark}
\newtheorem{remark}[theorem]{Remark}

\DeclareMathOperator{\supp}{supp}
\newcommand{\sel}[1]{\textcolor{blue!65!black}{\mathbf{#1}}}
\newcommand{\oeislink}[2]{%
  \begingroup
  \setlength{\fboxsep}{1.5pt}%
  \colorbox{orange!18}{\href{#1}{\textcolor{black}{\strut\textbf{#2}}}}%
  \endgroup
}

\title[Equilateral completion in floretion coordinates]{Equilateral completion in floretion triangular coordinates:\\
locality, product points, and reflection symmetry}
\author{Creighton Dement}
\address{Independent researcher}
\email{floretionguru@gmail.com}
\date{29 October 2026}
\hypersetup{
  pdftitle={Equilateral Completion in Floretion Triangular Coordinates: Locality, Product Points, and Reflection Symmetry},
  pdfauthor={Creighton Dement}
}

\makeatletter
\renewenvironment{abstract}{%
  \global\setbox\abstractbox=\vtop\bgroup
    \normalfont\Small
    \list{}{\labelwidth\z@
      \leftmargin1.5pc \rightmargin\leftmargin
      \listparindent\normalparindent \itemindent\z@
      \parsep\z@ \@plus\p@
      
    }%
    \item[\hskip\labelsep\scshape\abstractname.]%
}{%
  \endlist\egroup
  \ifx\@setabstract\relax \@setabstracta \fi
}
\makeatother
\begin{document}
\begin{abstract}
We study nondegenerate equilateral triangles whose vertices are centroids
of the \(4^n\) order-\(n\) floretion tiles.  Scaled integer triangular
coordinates give exact equilateral completions and, by an obstruction
modulo \(3\), show that all three vertices have the same tile orientation.
In particular, the complete triangle family and its synchronized local
cyclic subfamily have respective sizes
\[
  |E_n|=\frac{4^n(4^n-1)}{12},\qquad
  |L_n|=\frac{7^n-4^n}{3},
\]
so the local proportion tends to zero.

Our principal new result answers the radial-homogeneity question
left open in the preceding revision: for every order, the number of
triangles containing a given tile is given by an exact affine formula
in its squared scaled centroid radius.  Although the two tile
orientations have different formulas, their attainable squared radii
belong to disjoint residue classes modulo \(3\), yielding a single
radial-incidence law on the full centroid set.  We also determine
local and nonlocal vertex incidences and obtain circular-shell and
squared-radius moment identities.

The unsigned vertex product defines an algebraic product point.
By a rigidity theorem of Ma and Zhang, the multiplication-generated
triangles are precisely the nontrivial global cyclic orbits; combining
this with our geometric shell classification, we show that the product
point is the origin exactly when the Euclidean triangle center is the
origin.  Further results enumerate central-vertex, no-\(e\), and
main-axis families, derive Fibonacci and finite-state recurrences for
product-centered and reflection-restricted classes, and characterize
pairwise anticommutation.
\end{abstract}

\maketitle

\clearpage
\section{Introduction}

The floretion coordinate model encodes the positive basis of
\(\mathbb H^{\otimes n}\) by words of length \(n\) in
\[
  \{i,j,k,e\}.
\]
The same \(4^n\) words label the tiles obtained after \(n\) stages of the
recursive midpoint subdivision of an equilateral triangle.  The associated
centroid map \(P\), its compatibility with that subdivision, and its
digitwise \(S_3\)-equivariance were developed in the preceding paper
\cite{Dement2026}.

The present paper asks a different, finite-geometric question.  For three
distinct order-\(n\) base words \(b_1,b_2,b_3\), when do
\[
  P(b_1),\quad P(b_2),\quad P(b_3)
\]
form a nondegenerate equilateral triangle?  Let \(E_n\) denote the set of
all such unordered triples.  The first main step is to rewrite \(P\) in
scaled integer triangular coordinates.  The two Euclidean completions of a
pair then become explicit integer formulas, while a residue calculation
modulo \(3\) separates the two tile orientations.  This gives an intrinsic
proof that an equilateral centroid triangle cannot mix orientations.  The
same conclusion is independently explained by the three-color
refinement-lattice theorem of Ivrissimtzis, Dodgson, and Sabin
\cite{IDS2004}.  The two orientation classes are triangular point sets with
\(2^n\) and \(2^n-1\) rows, so the finite completion formulas of Brouwer,
Joe, Noble, and Noble \cite{Brouwer2024} give
\[
  |E_n|=\frac{4^n(4^n-1)}{12}.
\]
A further count isolates those triangles containing the all-\(e\) word
\(e_n\) as an actual vertex (rather than as their center); its
closed form is also a cyclops Gray-code sequence
(Proposition~\ref{prop:central-vertex}).

The local cyclic construction from \cite[Theorem~4.8]{Dement2026}
supplies a distinguished subclass.  Following the canonical convention of
that theorem, local selections satisfy
\[
  \varnothing\ne S\subseteq\supp(b).
\]
The selected set is visible from the triangle itself, the local classes are
disjoint, and
\[
  |L_n|=\frac{7^n-4^n}{3},
  \qquad
  |E_n\setminus L_n|
   =\frac{16^n+3\cdot4^n-4\cdot7^n}{12}.
\]
The local class also has a natural group-action description: its triangles
are exactly the orbits of a free synchronized \(C_3\)-action on local-cycle
representatives.  This gives a second route to the local count and packages
several later divisions by \(3\) into one orbit construction.
Thus synchronized local cycles, although intrinsic to the coordinate model,
form an exponentially vanishing fraction of all equilateral centroid
triangles.

A second theme compares multiplication with Euclidean geometry.  Any two
positive basis words commute or anticommute, so permuting three factors
changes at most the global sign.  Hence the following unsigned product is
well defined.

\begin{definition}[Product point and product-centered triangle]
For \(T=\{b_1,b_2,b_3\}\in E_n\), let
\[
  p(T)=u(b_1b_2b_3),\qquad
  C_T=P(p(T)),\qquad
  Q_T=\frac{P(b_1)+P(b_2)+P(b_3)}{3},
\]
where \(u(\pm b)=b\) removes the global sign.  We call \(C_T\) the
\emph{product point} of \(T\).  The triangle is \emph{product-centered} if
\(C_T=Q_T\), and we write
\[
  P_n=\{T\in E_n:C_T=Q_T\}.
\]
\end{definition}

Theorem~4.8 of the preceding paper is the central structural result
governing synchronized local cycles, their products, and their Euclidean
centers, and it remains a principal input in the present paper.  In
particular, we shall see that the rule
\medskip
\begingroup
\setlength{\fboxsep}{4pt}
\begin{center}
\fbox{\parbox{0.88\linewidth}{\centering\small\itshape
Every unselected noncentral coordinate must have an even number of preceding selected coordinates.}}
\end{center}
\endgroup
\medskip
established in \cite[Theorem~4.8(4)]{Dement2026}, can be viewed as a
two-state machine that brings Fibonacci numbers naturally into the
enumeration.  Together with the unsigned product identity coming from
\cite[Theorem~4.8(2)]{Dement2026}, it yields a two-state recurrence for
all product-centered local triangles.  On the no-\(e\) support
\(S_n=\{i,j,k\}^n\), the absence of the zero residue modulo \(2\) gives a
second rigidity phenomenon: the last digit of an oriented equilateral
completion is forced to be constant or to follow the cyclic direction
determined by the global orientation.  Iterating the last-digit
decomposition proves
\[
  E_n^S=L_n^S,
  \qquad
  |E_n^S|=(2^n-1)3^{n-1}.
\]
Within this support, admissible product-centered selection patterns are
binary words assembled from \(0\) and \(11\), with a possible final
unpaired \(1\), producing the Fibonacci count
\[
  (F_{n+2}-1)3^{n-1}.
\]
Reflection symmetry gives a second Fibonacci family.  For one fixed main
axis, the local triangles whose product point lies on the axis number
\((5^n-2^n)/3\), while the product-centered local triangles symmetric about
the whole axis number \(F_{2n+2}-2^n\).

The same three axes also support a closed equilateral subfamily.  Every
axis-supported equilateral triangle has axis coordinates satisfying either
\(x=y=z\) or \(x+y+z=0\), and Section~\ref{sec:reflection} proves
\[
  |E_n^{\mathrm{ax}}|=4^{n-1}+2^n-2.
\]

We also ask when
\[
  \{b,c,u(bc)\}
\]
is itself an equilateral centroid triple.  Multiplication-generation turns
out to have a simple product-point characterization:
\[
  T\text{ is multiplication-generated}
  \quad\Longleftrightarrow\quad
  p(T)=e_n
  \quad\Longleftrightarrow\quad
  C_T=0.
\]
Within the local class this forces the selected set to be the full support,
so the multiplication-generated local triangles are exactly the nontrivial
global \(\gamma\)-orbits.  The remaining question in the earlier version of
this paper was whether nonlocal multiplication-generated triangles could
occur.  Haobo Ma and Wenlin Zhang have now answered this in the negative
\cite{MaZhang2026}: by adapting the digit-stripping induction of
Section~\ref{sec:sierpinski} and incorporating the unsigned-product condition,
they prove that \emph{all} multiplication-generated equilateral triangles
are global \(\gamma\)-orbits.  Consequently, at every order,
\[
\begin{aligned}
  \mathcal M_n
  &=\bigl\{\{b,\gamma(b),\gamma^2(b)\}:b\in\Delta_n\setminus\{e_n\}\bigr\},\\
  |\mathcal M_n|&=\frac{4^n-1}{3},
  \qquad \mathcal M_n\cap X_n=\varnothing.
\end{aligned}
\]
Their proof, presented in their 21 October 2026 manuscript, is cited rather
than reproduced here; its publication venue has not yet been specified.
Moreover, Proposition~\ref{prop:shell-triples} shows geometrically that
a nondegenerate equilateral centroid triangle is centered at the origin
if and only if it is a global \(\gamma\)-orbit.  Together with
Ma--Zhang rigidity this gives the further equivalence
\[
  C_T=0\quad\Longleftrightarrow\quad Q_T=0
  \quad\Longleftrightarrow\quad
  T=\{b,\gamma(b),\gamma^2(b)\}\quad(b\ne e_n),
\]
where \(C_T\) is the product point and \(Q_T\) the Euclidean
triangle center.

Section~10 then retains the central sign discarded by \(u\).  For
\(T\in E_n\), the square \(\Theta(T)^2\) of the vertex sum is scalar
exactly when the three vertices pairwise anticommute.  For a local cycle
\(T(b,S)\), this occurs exactly when \(|S|\) is odd, giving
\[
  |\mathrm{AC}_n\cap L_n|=\frac{7^n-1}{6};
\]
nonlocal examples already occur in order \(3\).

The radial viewpoint is developed further in Sections~\ref{sec:radial-incidence}
and~\ref{sec:alternative-enumerations}.  An elementary fact
about triangular lattice patches is central to that analysis and is independent
of the floretion model.  If
\(\Omega_M=\{(a,b,c)\in\mathbb Z_{\geq0}^3:a+b+c=M\}\), then the
number of nondegenerate equilateral triangles in \(\Omega_M\) with a specified
point \( (a,b,c) \) as a vertex is
\[
  \boxed{D_M(a,b,c)=M+ab+bc+ca.}
\]
A self-contained proof appears in Lemma~\ref{lem:grid-incidence}.

A further question, left open on computational evidence in the preceding
revision of this paper, asks how the complete collection of triangles is
distributed among its individual vertices.  Write
\(\eta_n(b)=|\{T\in E_n:b\in T\}|\), let
\(\Lambda_n(b)\) denote the integer centroid coordinates, and for
the following display write \(R=q_{120}(\Lambda_n(b))\), where
\(q_{120}(x,y)=x^2-xy+y^2\).
The principal new result, Theorem~\ref{thm:radial-incidence}, gives the
\emph{exact radial-incidence law}
\[
 \eta_n(b)=
 \begin{cases}
  \bigl((2^n-1)(2^n+2)-R\bigr)/3,
     &\text{upward tile},\\
  \bigl((2^n-2)(2^n+1)-R\bigr)/3,
     &\text{downward tile}.
 \end{cases}
\]
In particular, \(q_{120}(\Lambda_n(b))=q_{120}(\Lambda_n(c))\)
implies \(\eta_n(b)=\eta_n(c)\)
for \emph{every} order \(n\), not merely for the orders accessible to
enumeration.  The two orientation grids have different affine incidence
formulas, but a congruence modulo \(3\) shows that their realized
squared radii never coincide.  The finite triangular-grid incidence
count is elementary.  The floretion-specific step is to identify its
two grids with the two tile-orientation classes, express incidence in terms of the squared
centroid radius \(q_{120}(\Lambda_n(b))\), and use their modulo-\(3\)
separation to deduce a single global radial law rather than two
separate orientation-specific formulas.  If
\(b^2=\varepsilon_b e_n\), the equivalent algebraic identity of
Corollary~\ref{cor:radial-algebraic} makes the role of multiplication explicit:
\[
  \boxed{3\eta_n(b)+q_{120}(\Lambda_n(b))
    =4^n-2+(-2)^n\varepsilon_b.}
\]
The local and nonlocal parts of this incidence admit further exact
pointwise formulas:
\(\eta_n^{\mathrm{local}}(b)=2^{\nu(b)}-1\) and
\(\eta_n^{\mathrm{nonlocal}}(b)=\eta_n(b)-2^{\nu(b)}+1\), where
\(\nu(b)\) counts the non-\(e\) digits of \(b\).
Consequently, total radial homogeneity does not imply radial homogeneity
of either part.

On every nonzero complete
circular shell \(K\), the equilateral triangles with all three vertices
in \(K\) are exactly its global \(\gamma\)-orbits, hence number
\(|K|/3\); the counts meeting a shell in one or two vertices reduce
to an explicit finite pair-incidence sum.  The distinguished midpoint
shell---the incircle of the downward-centroid grid---has incidence
\(4^{n-1}-1\) at each of its vertices.  Summing the algebraic radial-incidence identity over all basis
words also yields the moment identity
\[
 \sum_{b\in\Delta_n}q(\Lambda_n(b))
   =\frac{4^n(4^n-1)}4=3|E_n|.
\]
The squared-radius identity connects the centroid geometry with the total
triangle census, while a short double count explains the balance between
zero and two pair completions.  These results concern the \emph{full}
class \(E_n\). Although the local and nonlocal vertex-incidence
functions are now determined exactly, their finer circular-shell
intersection distributions, together with those of the
pairwise-anticommuting subclass, remain separate questions.

\medskip
\noindent\textbf{Relation to the preceding paper.}
The base-vector notation, multiplication, recursive centroid map,
orientation rule, digitwise \(S_3\)-equivariance, reflection
anti-automorphisms, centralizer decomposition, and local cyclic construction
originate in \cite{Dement2026}.  The present paper restates the coordinate
data needed for Sections~2--4.  Theorem~4.8 of the preceding paper uses the
same canonical requirement \(S\subseteq\supp(b)\).  Its parts (2) and (4)
give, respectively, the signed product identity and the center-coincidence
criterion that drive the product-point enumerations in Sections~6--8.

For finite-order experimentation, an online floretion multiplier and
visualization tool is available at \cite{Calculator}.  Readers who prefer to
begin with concrete cases may consult Appendix~B first: it collects small
examples, mainly in order three, of the principal triangle classes and gives
the corresponding calculator input.  The order-two signed floretion group
was previously studied in the original notation by Mathar \cite{Mathar2010}.

\medskip
\noindent\textbf{Two coordinate descriptions.}
Our geometric arguments use two complementary descriptions of the same
triangular-lattice geometry.  The scaled floretion centroid coordinates
\(\Lambda_n\) use a pair of basis directions meeting at \(120^\circ\),
whereas finite triangular patches are most conveniently described by
\emph{barycentric} lattice triples
\[
  \Omega_M=\{(a,b,c)\in\mathbb Z_{\geq0}^3:a+b+c=M\}.
\]
No vertex of the patch is singled out as a geometric origin in this
three-coordinate notation.  Since \(c=M-a-b\), we may immediately pass to
just two coordinates \((a,b)\) (or to either other pair).  The triple
makes side symmetries and reflections transparent, while the pair makes
integer-lattice calculations and boundary counts economical.

With a unit length assigned to each basis vector, the squared-distance
forms are
\[
 q_{120}(x,y)=x^2-xy+y^2,\qquad
 q_{60}(r,s)=r^2+rs+s^2=q_{120}(r,-s).
\]
The latter is obtained by reversing one basis direction; the corresponding
barycentric-coordinate pair uses lattice directions meeting at
\(60^\circ\).  If two barycentric triples differ by
\((\delta a,\delta b,\delta c)\), where
\(\delta a+\delta b+\delta c=0\), then their Euclidean distance \(d\),
measured relative to the nearest-neighbor lattice spacing \(h\), is
\[
  \frac{d^2}{h^2}
  =\frac{(\delta a)^2+(\delta b)^2+(\delta c)^2}{2}
  =q_{60}(\delta a,\delta b).
\]
\emph{The two sign conventions and their physical scales must not be
mixed.}  In particular, the form \(q\) in (2.5) below is \(q_{120}\),
whereas the barycentric counting proofs of Lemmas~7.6 and~11.1 use
\(q_{60}\).  In Section~11, for corresponding centroid and barycentric points,
\(\|P(b)\|^2/h_n^2=q_{60}(\alpha,\beta)
=q_{120}(\Lambda_n(b))/3\).
The factor of three is essential.

\section{Base vectors, centroids, and orientation}

\begin{definition}[Base vectors]
For \(n\ge1\), let
\[
  \Delta_n=\{i,j,k,e\}^n.
\]
The signed basis group is
\[
  F_n=\{\pm b:b\in\Delta_n\},
\]
with identity \(e_n=ee\cdots e\).  Multiplication is coordinatewise
quaternion multiplication, with the local signs collected into one global
sign.
\end{definition}

The symbolic alphabet corresponds to the octal notation of
\cite{Dement2026} by
\[
  1\leftrightarrow i,\qquad
  2\leftrightarrow j,\qquad
  4\leftrightarrow k,\qquad
  7\leftrightarrow e.
\]
The octal alphabet is the natural encoding of floretion base vectors in
the broader framework.  It gives the positive base vectors a canonical
order and supports efficient bitwise multiplication.  It also makes
available arithmetic notions attached to that ordering, such as ordinary
and wrapped distances between base vectors.  These features are not used
in the arguments developed in the present paper, so the symbolic notation
\(i,j,k,e\) will be used for readability.

Retain the geometric convention
\[
\begin{aligned}
  v(i)&=(\cos210^\circ,\sin210^\circ) &&\text{(lower left)},\\
  v(j)&=(\cos 90^\circ,\sin 90^\circ) &&\text{(top)},\\
  v(k)&=(\cos330^\circ,\sin330^\circ) &&\text{(lower right)},\\
  v(e)&=(0,0).
\end{aligned}
\]
Thus \(v(i)+v(j)+v(k)=0\).  Fix \(d_1>0\) and put
\(d_r=d_1/2^{r-1}\).

\begin{definition}[Centroid map]
For \(b=b_1\cdots b_n\in\Delta_n\), define
\[
  P(b)=\sum_{r=1}^n \sigma_r(b)d_rv(b_r),
  \qquad
  \sigma_r(b)=(-1)^{\#\{s<r:b_s=e\}}.
\]
\end{definition}

This is the centroid of the tile labelled by \(b\) in the recursive
triangular tiling \cite{Dement2026}.  Let
\[
  \nu(b)=\#\{r:b_r\ne e\}.
\]
The tile is upward if and only if
\(\nu(b)\equiv n\pmod2\).  Since each noncentral quaternionic basis symbol
squares to \(-e\),
\begin{equation}\label{eq:square-orientation}
  b^2=(-1)^{\nu(b)}e_n.
\end{equation}
Thus, at fixed order, two tiles have the same orientation exactly when their
basis squares have the same sign.

Put \(N=2^n\).  The order-\(n\) tiling has
\begin{equation}\label{eq:updown-count}
  U_n=\frac{N(N+1)}2
  \quad\text{upward tiles},\qquad
  D_n=\frac{N(N-1)}2
  \quad\text{downward tiles}.
\end{equation}
The upward centroids form a triangular point set with \(N\) rows and the
downward centroids a translated triangular point set with \(N-1\) rows.

\subsection{Scaled integer centroid coordinates}

Put \(u=v(i)\) and \(w=v(j)\).  The angle between \(u\) and \(w\) is
\(120^\circ\), and \(v(k)=-u-w\).  Define
\[
  \lambda(i)=(1,0),\qquad
  \lambda(j)=(0,1),\qquad
  \lambda(k)=(-1,-1),\qquad
  \lambda(e)=(0,0).
\]

\begin{definition}[Scaled integer centroid-coordinate map]
For \(b=b_1\cdots b_n\in\Delta_n\), define
\begin{equation}\label{eq:Lambda}
  \Lambda_n(b)
   =\sum_{r=1}^n \sigma_r(b)\,2^{\,n-r}\lambda(b_r)\in\mathbb Z^2.
\end{equation}
If \(\Lambda_n(b)=(X_n(b),Y_n(b))\), these are the scaled integer centroid
coordinates of \(b\).
\end{definition}

\begin{lemma}[Integer realization of the centroid map]
For every \(b\in\Delta_n\),
\begin{equation}\label{eq:integer-realization}
  \frac{2^{n-1}}{d_1}P(b)=X_n(b)u+Y_n(b)w.
\end{equation}
In particular, \(\Lambda_n:\Delta_n\to\mathbb Z^2\) is injective.
\end{lemma}

\begin{proof}
Multiplying the centroid formula by \(2^{n-1}/d_1\) gives
\[
  \frac{2^{n-1}}{d_1}P(b)
   =\sum_{r=1}^n\sigma_r(b)2^{\,n-r}v(b_r).
\]
Writing \(\lambda(b_r)=(\lambda_1(b_r),\lambda_2(b_r))\), we have
\[
  v(b_r)=\lambda_1(b_r)u+\lambda_2(b_r)w.
\]
Hence the right-hand side becomes
\[
  \sum_{r=1}^n\sigma_r(b)2^{\,n-r}\lambda_1(b_r)\,u
  +\sum_{r=1}^n\sigma_r(b)2^{\,n-r}\lambda_2(b_r)\,w
  =X_n(b)u+Y_n(b)w,
\]
which is \eqref{eq:integer-realization}; this is exactly the coordinate pair
recorded by \eqref{eq:Lambda}.  Distinct words label distinct tiles of the
order-\(n\) triangulation, whose interior centroids are distinct; hence
\(\Lambda_n\) is injective.
\end{proof}

\begin{remark}[The triangular-lattice metric]
The coordinate basis \((u,w)\) is oblique.  Its induced quadratic form is
\begin{equation}\label{eq:qform}
  q(x,y)=x^2-xy+y^2.
\end{equation}
Thus \(q(1,0)=q(0,1)=q(-1,-1)=1\), preserving the symmetry among the
three noncentral digits.  More generally,
\[
  \|P(b)-P(c)\|^2
   =\left(\frac{d_1}{2^{n-1}}\right)^2
      q\!\left(\Lambda_n(b)-\Lambda_n(c)\right).
\]
\end{remark}

\begin{lemma}[Orientation residue classes]\label{lem:orientation-residues}
Let \(\Lambda_n(b)=(X,Y)\) and \(N=2^n\).  Then
\begin{equation}\label{eq:orientation-residues}
  X+Y\equiv
  \begin{cases}
    N-1\pmod3,&\text{if the tile of \(b\) is upward},\\
    N-2\pmod3,&\text{if the tile of \(b\) is downward}.
  \end{cases}
\end{equation}
Consequently the third residue class \(X+Y\equiv N\pmod3\) contains no
order-\(n\) tile centroids.
\end{lemma}

\begin{proof}
The tile is upward exactly when the number of \(e\)-digits is even.  For
\(n=1\), the three noncentral digits have coordinate sums
\(1,1,-2\equiv1=N-1\pmod3\), whereas \(e\) has sum \(0=N-2\).

Assume the assertion at order \(n\), write \(N=2^n\), and append a digit
\(a\) to a word \(b\).  From \eqref{eq:Lambda},
\[
  \Lambda_{n+1}(ba)
   =2\Lambda_n(b)+(-1)^{\#\{r:b_r=e\}}\lambda(a).
\]
Modulo \(3\), the coordinate sum of \(\lambda(a)\) is \(1\) for
\(a\in\{i,j,k\}\) and \(0\) for \(a=e\).  If \(a\) is noncentral the parity
of the number of \(e\)'s is unchanged; if \(a=e\) it is reversed.  The four
cases are
\[
\begin{array}{c@{\qquad}c@{\qquad}c@{\qquad}c}
\toprule
b & a & ba & X_{n+1}(ba)+Y_{n+1}(ba)\pmod3\\
\midrule
\text{up}   & a\ne e & \text{up}   & 2(N-1)+1=2N-1\\
\text{down} & a\ne e & \text{down} & 2(N-2)-1=2N-5\equiv2N-2\\
\text{up}   & a=e    & \text{down} & 2(N-1)=2N-2\\
\text{down} & a=e    & \text{up}   & 2(N-2)=2N-4\equiv2N-1\\
\bottomrule
\end{array}
\]
which is precisely \eqref{eq:orientation-residues} at order \(n+1\).
\end{proof}

\begin{remark}[Main axes in scaled coordinates]\label{rem:main-axes-coordinates}
In the scaled integer coordinates \(\Lambda_n\), the three main reflection axes
have the forms
\[
  (x,0),\qquad (0,y),\qquad (-z,-z).
\]
The order-\(n\) centroid coordinates that actually occur on these axes will be
determined in Subsection~\ref{subsec:axis-supported}.
\end{remark}

\begin{figure}[t]
  \centering
  \includegraphics[width=.90\textwidth]{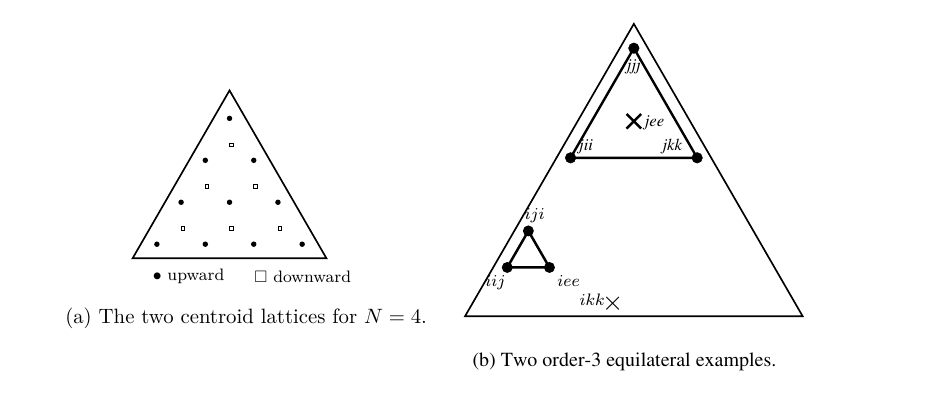}
  \caption{Centroid orientation classes and two order-\(3\) examples.
  In (b), the lower triangle \(\{iee,iij,iji\}\) has product point
  \(P(ikk)\), marked by a cross, which is not its Euclidean center.
  The upper triangle \(\{jii,jjj,jkk\}\) is product-centered, with
  \(C_T=Q_T=P(jee)\).}
  \label{fig:orientation}
\end{figure}

\section{Equilateral completion}

\begin{definition}[Completion number]
For distinct \(b_1,b_2\in\Delta_n\), define
\[
  \kappa_n(b_1,b_2)
  =\#\left\{
    b_3\in\Delta_n:
    \begin{array}{c}
      P(b_1),P(b_2),P(b_3)\text{ are the vertices}\\[-1mm]
      \text{of a nondegenerate equilateral triangle}
    \end{array}
  \right\}.
\]
We say that \(\{b_1,b_2\}\) can be completed if
\(\kappa_n(b_1,b_2)>0\).  Since a segment has exactly two Euclidean
equilateral third vertices,
\(\kappa_n(b_1,b_2)\in\{0,1,2\}\).
\end{definition}

Relative to the oblique basis \((u,w)\), rotations through \(+60^\circ\)
and \(-60^\circ\) are
\begin{equation}\label{eq:rot60}
  R_+(r,s)=(s,s-r),\qquad
  R_-(r,s)=(r-s,r).
\end{equation}
Both preserve \(q\) from \eqref{eq:qform}.

\begin{proposition}[The two candidate completions]\label{prop:two-completions}
Let
\[
  \Lambda_n(b_1)=(r_1,s_1),\qquad
  \Lambda_n(b_2)=(r_2,s_2),
  \qquad b_1\ne b_2.
\]
The two possible Euclidean third vertices have scaled integer coordinates
\begin{align}
  Z_+(b_1,b_2)
    &=(r_1+s_2-s_1,\;r_1-r_2+s_2),\label{eq:zplus}\\
  Z_-(b_1,b_2)
    &=(r_2-s_2+s_1,\;s_1+r_2-r_1).\label{eq:zminus}
\end{align}
Moreover,
\begin{equation}\label{eq:kappa-lookup}
  \kappa_n(b_1,b_2)
   =\#\bigl(
      \{Z_+(b_1,b_2),Z_-(b_1,b_2)\}
      \cap\Lambda_n(\Delta_n)
     \bigr).
\end{equation}
\end{proposition}

\begin{proof}
Put \(A=\Lambda_n(b_1)\) and \(B=\Lambda_n(b_2)\).  The two Euclidean
completions of the segment \(AB\) are
\(A+R_+(B-A)\) and \(A+R_-(B-A)\).  For example,
\[
\begin{aligned}
  A+R_+(B-A)
   &=(r_1,s_1)+(s_2-s_1,\;s_2-s_1-r_2+r_1)\\
   &=(r_1+s_2-s_1,\;r_1-r_2+s_2),
\end{aligned}
\]
giving \eqref{eq:zplus}; \eqref{eq:zminus} is analogous.
Equation \eqref{eq:integer-realization} changes the centroid plane to
\(\Lambda_n\)-coordinates by a fixed similarity, and \(\Lambda_n\) is
injective.  Hence each candidate in \(\Lambda_n(\Delta_n)\) corresponds to
exactly one completing base vector.
\end{proof}

\subsection{Embedding the finite tiling in the full triangular lattice}

The order-\(n\) picture is finite, while the lattice theorem used below is
naturally stated for the regular triangular tessellation of the plane.

\begin{lemma}[Reflection extension]
Let \(T_n^{(0)}\) be the geometric order-\(n\) triangular tiling inside its
outer equilateral triangle \(T^{(0)}\).  Reflect \(T_n^{(0)}\) across each of
the three lines containing the sides of \(T^{(0)}\), and adjoin the three
reflected copies.  The result tiles an equilateral triangle \(T^{(1)}\) with
twice the side length and opposite orientation.  Repeating gives nested
tiled triangles
\[
  T^{(0)}\subset T^{(1)}\subset T^{(2)}\subset\cdots,
  \qquad
  \operatorname{side}(T^{(m)})
   =2^m\operatorname{side}(T^{(0)}),
\]
whose union is the full regular triangular tiling of \(\mathbb R^2\) at the
same mesh scale.
\end{lemma}

\begin{proof}
Reflection across the three sides of one equilateral triangle produces its
three adjacent congruent copies.  The four triangles fit into a new
equilateral triangle of doubled side length and opposite orientation.
Reflection preserves the triangular mesh, so the same holds for the full
subdivision.  Iteration gives the displayed side lengths.  The triangles
remain concentric and their inradii tend to infinity, so every bounded
subset of the plane eventually lies in some \(T^{(m)}\).
\end{proof}

\subsection{The same-orientation theorem and its three-color interpretation}

In the full triangular mesh, consider the critical-point lattice consisting
of mesh vertices together with all face centroids.  Color mesh vertices with
one color, centroids of upward faces with a second, and centroids of
downward faces with a third.  Up to Euclidean similarity this is the
three-colored lattice of Ivrissimtzis, Dodgson, and Sabin.

\begin{proposition}[Three-color theorem; Ivrissimtzis--Dodgson--Sabin]
Let three points of this critical-point lattice form a nondegenerate
equilateral triangle.  Then the three points either all have the same color
or have three distinct colors.
\end{proposition}

\begin{proof}
This is Proposition~4.1 of Ivrissimtzis, Dodgson, and Sabin
\cite{IDS2004}.
\end{proof}

\begin{theorem}[Same-orientation theorem]\label{thm:same-orientation}
If \(P(b_1),P(b_2),P(b_3)\) form a nondegenerate equilateral triangle, then
the three associated tiles have the same orientation.  Equivalently,
\begin{equation}\label{eq:same-squares}
  b_1^2=b_2^2=b_3^2.
\end{equation}
\end{theorem}

\begin{proof}[First proof: integer centroid coordinates]
For \(A=(r,s)\in\mathbb Z^2\), put
\[
  \Sigma(A)=r+s,\qquad
  \rho(A)\equiv\Sigma(A)\pmod3.
\]
Let \(A=(r_1,s_1)\) and \(B=(r_2,s_2)\).  Direct addition in
\eqref{eq:zplus}--\eqref{eq:zminus} gives
\[
  \Sigma(Z_+)+\Sigma(A)+\Sigma(B)=3(r_1+s_2),
\]
and
\[
  \Sigma(Z_-)+\Sigma(A)+\Sigma(B)=3(r_2+s_1).
\]
Hence either candidate completion satisfies
\begin{equation}\label{eq:completion-residue}
  \rho(Z_\pm)\equiv-\rho(A)-\rho(B)\pmod3.
\end{equation}

If one tile is upward and the other downward, Lemma
\ref{lem:orientation-residues} gives residues \(N-1\) and \(N-2\).
Equation \eqref{eq:completion-residue} then gives residue
\[
  -(N-1)-(N-2)\equiv N\pmod3.
\]
But Lemma~\ref{lem:orientation-residues} says that no order-\(n\) centroid
lies in this third residue class.  Thus an opposite-orientation pair cannot
be completed.  If three centroids already form an equilateral triangle,
each pair is completed by the third, so all three orientations agree.
Equation \eqref{eq:square-orientation} gives \eqref{eq:same-squares}.
\end{proof}

\begin{proof}[Second proof: three-color theorem]
By the reflection-extension lemma, the finite order-\(n\) tiling lies in the
full regular triangular mesh.  Each point under consideration is a face
centroid, so none has the mesh-vertex color.  The three-distinct-colors
alternative would require one point of each color, including the
mesh-vertex color, and is impossible.  Hence the three centroids share one
color.  The two centroid colors distinguish upward from downward faces, so
the tiles have one orientation.  Equation \eqref{eq:square-orientation}
again gives \eqref{eq:same-squares}.
\end{proof}

\begin{corollary}[Orientation obstruction]
If \(b_1^2\ne b_2^2\), then
\[
  \kappa_n(b_1,b_2)=0.
\]
\end{corollary}

\begin{example}
In order \(2\),
\[
  (ei)^2=-e_2,\qquad (ik)^2=e_2.
\]
Consequently \(\{ei,ik\}\) cannot be completed, even though the
\mbox{definition uses} only centroid positions.  The converse is false: common orientation is
necessary but not sufficient.
\end{example}

\section{Counting pairs with zero, one, and two completions}

Let \(T_k\) denote the triangular point set with \(k\) rows, so
\[
  |T_k|=\frac{k(k+1)}2.
\]
For \(j\in\{0,1,2\}\), let \(a_j(k)\) be the number of unordered pairs of
points of \(T_k\) having exactly \(j\) equilateral completions in \(T_k\).
Brouwer, Joe, Noble, and Noble proved \cite[Theorems~3.3--3.5]{Brouwer2024}
that
\begin{equation}\label{eq:aj02}
  a_2(k)=a_0(k)=
  \begin{cases}
    \dfrac{(k-1)^2(k+1)(k+3)}{32},&k\text{ odd},\\[6pt]
    \dfrac{k(k-2)(k+2)^2}{32},&k\text{ even},
  \end{cases}
\end{equation}
and
\begin{equation}\label{eq:aj1}
  a_1(k)=
  \begin{cases}
    \dfrac{(k^2-1)(k^2+2k+3)}{16},&k\text{ odd},\\[6pt]
    \dfrac{k(k+2)(k^2+2)}{16},&k\text{ even}.
  \end{cases}
\end{equation}

\begin{theorem}[Same-orientation completion counts]
Let \(A_j(n)\) be the number of unordered pairs of order-\(n\) base vectors
having the same orientation and completion number \(j\).  With \(N=2^n\),
\begin{equation}\label{eq:A02}
  A_0(n)=A_2(n)=\frac{N^2(N^2-4)}{16},
\end{equation}
and
\begin{equation}\label{eq:A1}
  A_1(n)=\frac{N^2(N^2+2)}8.
\end{equation}
\end{theorem}

\begin{proof}
The upward centroid set is a copy of \(T_N\) and the downward centroid set
a copy of \(T_{N-1}\).  By Theorem~\ref{thm:same-orientation}, a completion
of a same-orientation pair remains in the same centroid set.  Hence
\[
  A_j(n)=a_j(N)+a_j(N-1).
\]
Since \(N=2^n\) is even and \(N-1\) is odd, \eqref{eq:aj02} gives
\[
\begin{aligned}
  A_0(n)=A_2(n)
   &=\frac{N(N-2)(N+2)^2}{32}
     +\frac{(N-2)^2N(N+2)}{32}\\
   &=\frac{N^2(N^2-4)}{16}.
\end{aligned}
\]
Likewise \eqref{eq:aj1} gives
\begin{align*}
  A_1(n)
   &=\frac{N(N+2)(N^2+2)}{16}
     +\frac{N(N-2)(N^2+2)}{16}\\
   &=\frac{N^2(N^2+2)}8.\qedhere
\end{align*}
\end{proof}

\begin{corollary}[All-pair completion counts]
Let \(B_j(n)\) be the number of all unordered pairs
\(\{b_1,b_2\}\subset\Delta_n\) with \(\kappa_n(b_1,b_2)=j\).  Then
\begin{align}
  B_0(n)&=\frac{N^2(5N^2-8)}{16},\label{eq:B0}\\
  B_1(n)&=\frac{N^2(N^2+2)}8,\qquad
  B_2(n)=\frac{N^2(N^2-4)}{16}.\label{eq:B12}
\end{align}
\end{corollary}

\begin{proof}
There are
\[
  U_nD_n=\frac{N^2(N^2-1)}4
\]
pairs of opposite orientation, and every such pair has completion number
zero.  Thus \(B_0(n)=U_nD_n+A_0(n)\), giving \eqref{eq:B0}; the other two
counts are \(A_1(n)\) and \(A_2(n)\).
\end{proof}

\begin{remark}
Exactly one half of all unordered pairs have opposite orientation and are
immediately noncompletable.  Among same-orientation pairs, the proportions
with zero, one, and two completions tend respectively to
\[
  \frac14,\qquad\frac12,\qquad\frac14.
\]
Thus same-orientation pairs that cannot be completed form a substantial
class rather than a boundary anomaly.
\end{remark}

\begin{corollary}[Total number of equilateral centroid triangles]
Let \(E_n\) be the set of nondegenerate equilateral triangles formed by
order-\(n\) tile centroids.  Then
\begin{equation}\label{eq:En-total}
  |E_n|=\frac{N^2(N^2-1)}{12}
       =\frac{4^n(4^n-1)}{12}.
\end{equation}
\end{corollary}

\begin{proof}
Each equilateral triangle contributes its three sides to the pair count,
while a pair of completion number \(j\) lies in exactly \(j\) triangles.
Therefore
\[
  3|E_n|=B_1(n)+2B_2(n),
\]
and substitution of \eqref{eq:B12} gives \eqref{eq:En-total}.
\end{proof}

\begin{remark}[A second derivation and OEIS notes]
Kagey gives a bijective proof \cite{Kagey2022} of the classical
triangular-grid enumeration that a point set \(T_k\) contains
\[
  \binom{k+2}{4}
\]
equilateral triangles, with no restriction on orientation.  The same count
is recorded as \oeislink{https://oeis.org/A000332}{OEIS A000332}.
Since the two centroid orientation classes are \(T_N\) and \(T_{N-1}\),
Theorem~\ref{thm:same-orientation} gives directly
\[
  |E_n|
   =\binom{N+2}{4}+\binom{N+1}{4}
   =\frac{N^2(N^2-1)}{12},
\]
which is a second derivation of \eqref{eq:En-total}.  The pair-completion
formulas above retain the finer information encoded by
\(B_0(n),B_1(n),B_2(n)\).

The values
\[
  1,20,336,5440,87296,\ldots
\]
of \(|E_n|\) appear, after an index shift, as
\oeislink{https://oeis.org/A166984}{OEIS A166984}.  Likewise
\[
  3,36,528,8256,131328,\ldots
\]
of \(B_1(n)\) appear to reproduce
\oeislink{https://oeis.org/A233196}{OEIS A233196}, again after an index shift.
\end{remark}

\begin{proposition}[Equilateral triangles with the central word as a vertex]
\label{prop:central-vertex}
Let \(e_n=e^{\otimes n}\) be the all-\(e\) basis word and set
\[
  \mathcal V_n
   =\{T\in E_n:e_n\in T\},\qquad
  C_n=|\mathcal V_n|.
\]
Then, for every \(n\ge1\),
\begin{equation}\label{eq:central-vertex-count}
  C_n=\frac{4^n+(-2)^n-2}{3}.
\end{equation}
Every triangle in \(\mathcal V_n\) is nonlocal.
\end{proposition}

\begin{proof}
The centroid \(P(e_n)=0\) is fixed by the digitwise cyclic symmetry
\(\gamma\).  Of the two orientation classes, the one containing \(e_n\)
is the smaller \(T_{2^n-1}\) when \(n\) is odd, and the larger
\(T_{2^n}\) when \(n\) is even: the central \(e\)-digit reverses the
orientation at each refinement.  Write its number of rows as
\[
  k=\begin{cases}
       2^n-1,&n\text{ odd},\\
       2^n,&n\text{ even}.
     \end{cases}
\]
In both cases \(k=3m+1\) for some \(m\ge0\).  The triangular centroid
patch in this orientation class is centered at \(P(e_n)=0\) and has
\(k-1=3m\) lattice intervals along each side.

By Proposition~\ref{prop:two-completions} and the same-orientation
Theorem~\ref{thm:same-orientation}, every equilateral triangle having
\(P(e_n)=0\) as a vertex is represented uniquely by
\(\{0,p,R_+p\}\), where \(p\ne0\) is a centroid in this patch
and \(R_+\) is the \(60^\circ\) rotation in \eqref{eq:rot60}.
Thus admissible points \(p\) are the noncentral lattice points common
to the original patch and its copy rotated through \(60^\circ\)
(the choice of the inverse rotation gives the same intersection).
Figure~\ref{fig:central-vertex-overlap} illustrates this overlap
at order \(n=3\), where the two noncentral shells have \(6\) and \(12\)
points.

For \(m\ge1\), the two centered equilateral triangular patches
intersect in a regular hexagon.  Its six sides are portions of the
six supporting lines of the two triangles, each at the same distance
from the center.  If the original triangle has side length \(3m\)
in lattice units, its inradius is \(r=(3m)\sqrt3/6\);
the resulting regular hexagon has side length
\(2r/\sqrt3=m\) lattice units.
The triangular lattice is preserved by \(60^\circ\) rotation about
its central lattice point.  Hence the overlap contains the center
and \(m\) successive hexagonal shells, with \(6j\) lattice points
in shell \(j\).  Removing the center gives
\[
 C_n=6\sum_{j=1}^{m}j
    =3m(m+1)
    =\frac{(k-1)(k+2)}{3}.
\]
For \(m=0\) the patch consists only of the center and the same
formula gives \(C_n=0\).  Substituting the two values of \(k\)
proves \eqref{eq:central-vertex-count}.

Finally, a nondegenerate local triangle has, at every selected
position, the three distinct digits \(i,j,k\).  Since \(e_n\)
has only \(e\)-digits, it cannot be a vertex of any such
triangle.  Hence \(\mathcal V_n\subseteq E_n\setminus L_n\).
\end{proof}

\begin{remark}[The cyclops Gray-code sequence]\label{rem:cyclops}
The numbers \(C_n\), beginning at \(n=1\), are
\[
  0,\ 6,\ 18,\ 90,\ 330,\ 1386,\ 5418,\ 21930,\ldots.
\]
They agree exactly with
\oeislink{https://oeis.org/A279260}{OEIS A279260},
with \(C_n=\mathrm{A279260}(n-1)\) in its zero-based indexing.
That sequence describes integers whose binary reflected Gray codes
are cyclops palindromes.  In fact, the Gray code of \(C_n\) is
\(1^{n-1}0\,1^{n-1}\), with the all-zero case at \(n=1\).
This is an equality of enumerations; no natural geometric--Gray-code
bijection is asserted here.  An order-two floretion coefficient
realization of \(C_n\) appears in Appendix~A,
equation~\eqref{eq:cyclops-generator}, and a small example appears
in Appendix~B.
\end{remark}

\begin{figure}[!htbp]
  \centering
  \includegraphics[width=.94\textwidth]{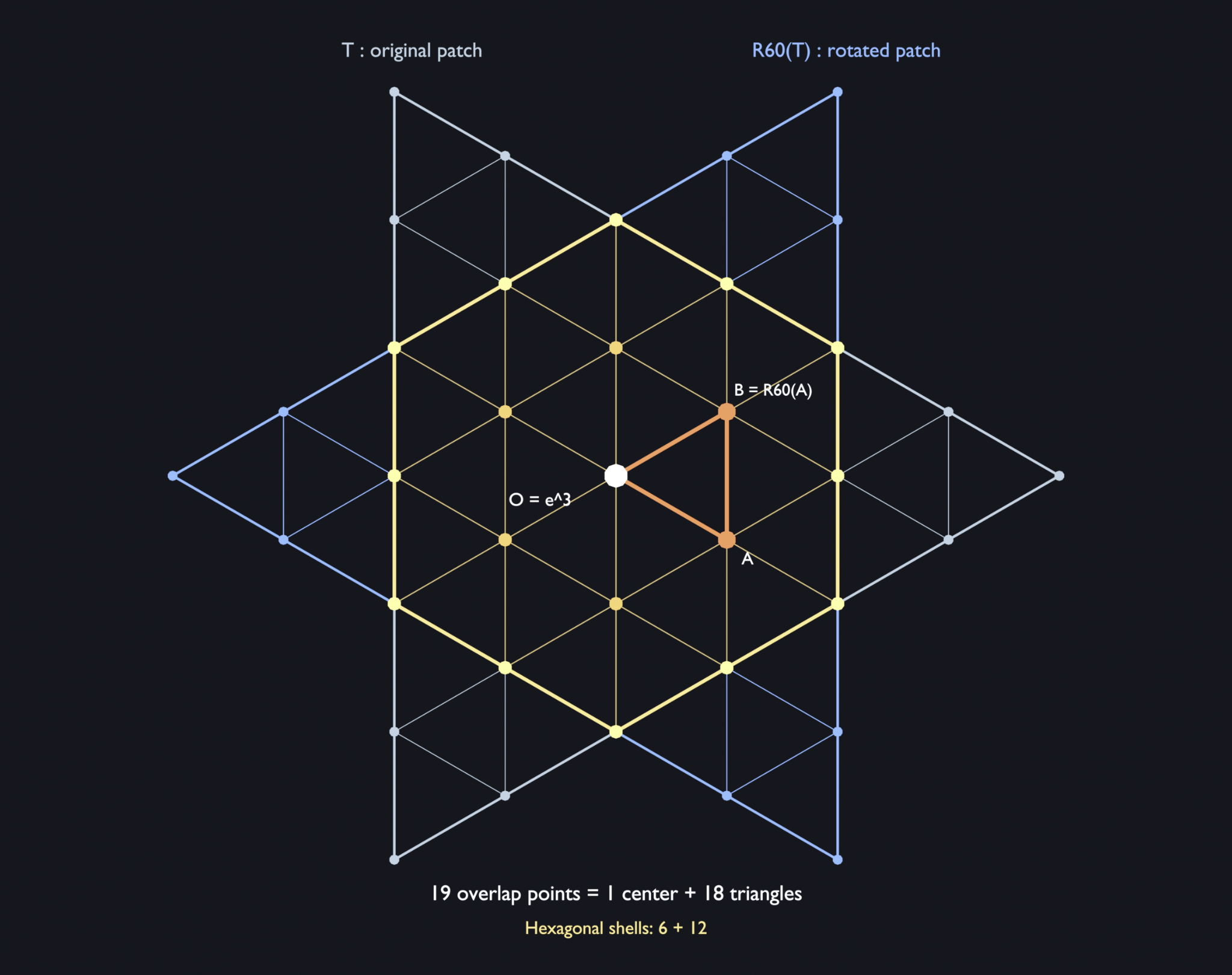}
  \caption{The central-vertex count at order \(n=3\).  The seven-row
  triangular centroid patch \(T\) containing \(P(e_3)=0\) and its
  \(60^\circ\)-rotated copy overlap in a regular hexagonal lattice patch.
  Its \(19=1+6+12\) points comprise the center and two concentric
  hexagonal shells; the \(18\) noncentral points each determine one
  equilateral triangle with \(e_3\) as a vertex.  The orange triangle
  illustrates one such completion \(\{0,A,R_+(A)\}\).}
  \label{fig:central-vertex-overlap}
\end{figure}
\FloatBarrier

\section{Local cycles and locality}\label{sec:locality}

Let
\[
  \gamma:i\mapsto j\mapsto k\mapsto i,\qquad \gamma(e)=e,
\]
and for \(b=b_1\cdots b_n\) put
\[
  \supp(b)=\{r\in\{1,\ldots,n\}:b_r\ne e\}.
\]

\begin{definition}[Local cycle]
For fixed \(b\in\Delta_n\), an \emph{admissible selected set} is a nonempty
set
\[
  S\subseteq\supp(b).
\]
For such \(S\), let \(\gamma_S(b)\) be obtained by applying \(\gamma\) in
the coordinates of \(S\) and fixing the others.  The associated local orbit
is
\[
  T(b,S)=\{b,\gamma_S(b),\gamma_S^2(b)\}.
\]
\end{definition}

This is the canonical local-cycle convention of
\cite[Theorem~4.8]{Dement2026}.  By that theorem, every \(T(b,S)\) above is a
nondegenerate equilateral centroid triangle.

\begin{definition}[Local and nonlocal classes]
For nonempty \(S\subseteq\{1,\ldots,n\}\), put
\[
  L_n(S)=\{T(b,S):b\in\Delta_n,\ S\subseteq\supp(b)\}.
\]
Define
\[
  L_n=\bigcup_{\varnothing\ne S\subseteq\{1,\ldots,n\}}L_n(S).
\]
A triangle \(T\in E_n\) is \emph{local} if \(T\in L_n\), and
\emph{nonlocal} otherwise.  Thus
\[
  X_n=E_n\setminus L_n
\]
is the nonlocal class.
\end{definition}

\begin{definition}[Common cyclic direction]
Let \((b_0,b_1,b_2)\) be an ordered triple and let \(S\) be a set of
coordinates.  The coordinates in \(S\) have a \emph{common cyclic
direction} if there is one \(\varepsilon\in\{+1,-1\}\) such that for every
\(r\in S\),
\[
  (b_{0r},b_{1r},b_{2r})
   =\bigl(a_r,\gamma^\varepsilon(a_r),\gamma^{2\varepsilon}(a_r)\bigr)
\]
for some \(a_r\in\{i,j,k\}\), with exponents read modulo \(3\).
\end{definition}

This is stronger than requiring each varying coordinate to contain
\(i,j,k\).  For example, in
\[
  (ij,ji,kk)
\]
the first coordinate is \((i,j,k)\), whereas the second is \((j,i,k)\);
the two coordinates have opposite cyclic directions.  Reordering the words
rotates or reverses both coordinate orders simultaneously, so the mismatch
cannot be removed.

\begin{proposition}[Digitwise recognition of local cycles]\label{prop:local-recognition}
An unordered equilateral triple
\[
  T=\{b_0,b_1,b_2\}\in E_n
\]
belongs to \(L_n\) if and only if its vertices can be ordered so that in
every coordinate \(r\) either
\[
  (b_{0r},b_{1r},b_{2r})=(a,a,a)
  \qquad(a\in\{i,j,k,e\}),
\]
or
\[
  \{b_{0r},b_{1r},b_{2r}\}=\{i,j,k\},
\]
and all coordinates of the second type have a common cyclic direction.
In that case \(S\) is exactly the set of coordinates of the second type.
\end{proposition}

\begin{proof}
A local cycle has the stated fixed patterns outside \(S\), while in \(S\)
its three digits are
\(a,\gamma(a),\gamma^2(a)\) with one common direction.  Conversely, suppose
an ordering with the stated patterns exists and let \(S\) be the varying
coordinates.  The first word has no \(e\)-digit in \(S\), so
\(S\subseteq\supp(b_0)\).  Applying either \(\gamma_S\) or
\(\gamma_S^{-1}\), according to the common direction, gives the other two
words.  Thus \(T\) is a local orbit.  The set \(S\) is intrinsic to the
triple as its set of varying coordinates.
\end{proof}

\begin{proposition}[\(C_3\)-orbits of local-cycle representatives]
\label{prop:C3-local-representatives}
Let
\[
  \widetilde L_n
   :=\{(b,S):b\in\Delta_n,\ 
       \varnothing\ne S\subseteq\supp(b)\},
\]
and let
\[
  \pi:\widetilde L_n\longrightarrow L_n,
  \qquad
  \pi(b,S)=T(b,S).
\]
The cyclic group \(C_3=\langle\tau\rangle\) acts on \(\widetilde L_n\) by
\[
  \tau\cdot(b,S)=(\gamma_S(b),S).
\]
This action is free, meaning that every stabilizer is trivial.  For each
\(T\in L_n\), the three representatives sent to \(T\) by \(\pi\) form
exactly one \(C_3\)-orbit.  Consequently, the \(C_3\)-orbits of
\(\widetilde L_n\) are in bijection with \(L_n\).  More generally, if
\(A\subseteq\widetilde L_n\) is a union of \(C_3\)-orbits (equivalently,
\(C_3\)-invariant), then
\begin{equation}\label{eq:C3-invariant-subset}
  |\pi(A)|=\frac{|A|}{3}.
\end{equation}
\end{proposition}

\begin{proof}
If \((b,S)\in\widetilde L_n\), then \(S\ne\varnothing\) and every selected
digit is noncentral.  Hence neither \(\gamma_S(b)\) nor
\(\gamma_S^2(b)\) equals \(b\), so the stabilizer of \((b,S)\) in \(C_3\)
is trivial.  Orbit--Stabilizer therefore gives an orbit of size \(3\).

For a local triangle \(T=T(b,S)\), Proposition~\ref{prop:local-recognition}
shows that \(S\) is intrinsic: it is exactly the set of varying coordinates.
The three possible starting vertices are
\[
  b,\qquad \gamma_S(b),\qquad \gamma_S^2(b),
\]
so the three representatives of \(T\) are precisely one \(C_3\)-orbit.
Thus the orbits are in bijection with \(L_n\).  If \(A\) is
\(C_3\)-invariant, it is a union of such three-element orbits, which gives
\eqref{eq:C3-invariant-subset}.
\end{proof}

This is complementary to the use of Orbit--Stabilizer in the preceding
paper: there the conjugation stabilizer is the centralizer and determines
its cardinality \cite[Theorem~6.3]{Dement2026}, whereas here the stabilizers
are trivial and the orbits themselves are the geometric objects being
counted.

\begin{corollary}[The \(e\)-support test]
All three vertices of a local triangle have the same \(e\)-support.
Consequently, unequal \(e\)-supports imply nonlocality.
\end{corollary}

\begin{proposition}[Fixed selected-set count]\label{prop:fixed-S-count}
For every nonempty \(S\subseteq\{1,\ldots,n\}\),
\[
  |L_n(S)|=3^{|S|-1}4^{n-|S|}.
\]
Moreover, the classes \(L_n(S)\) are pairwise disjoint.
\end{proposition}

\begin{proof}
For fixed \(S\), each selected coordinate has three possible starting digits
and each unselected coordinate four, giving
\(3^{|S|}4^{n-|S|}\) starting words.  Every nondegenerate orbit has three
cyclic representatives.  Proposition~\ref{prop:local-recognition} shows
that \(S\) is exactly the varying-coordinate set, so a local triangle cannot
belong to two different \(L_n(S)\).
\end{proof}

\begin{theorem}[Number of local-cycle triangles]
For \(n\ge1\),
\begin{equation}\label{eq:Ln}
  |L_n|=\frac{7^n-4^n}{3}.
\end{equation}
\end{theorem}

\begin{proof}
By the preceding proposition,
\[
  |L_n|
    =\sum_{\varnothing\ne S\subseteq\{1,\ldots,n\}}
       3^{|S|-1}4^{n-|S|}
    =\frac13\sum_{k=1}^n\binom nk3^k4^{n-k}
    =\frac{7^n-4^n}{3}.
\]
\end{proof}

Proposition~\ref{prop:C3-local-representatives} gives a second derivation.
Indeed, in each coordinate a representative has four possibilities when the
coordinate is unselected and three when it is selected, so
\[
  |\widetilde L_n|=7^n-4^n,
\]
after removing the empty selection.  Since every \(C_3\)-orbit has size
\(3\), the same proposition gives \eqref{eq:Ln}.  Thus the division by
\(3\) records the orbit size, rather than merely correcting for three
choices of starting vertex.

\FloatBarrier
\begin{table}[t]
\centering
\caption{Completion, local-cycle, and nonlocal counts for small orders.}
\label{tab:completion-local}
\small
\begin{tabular}{crrrrrrr}
\toprule
\(n\)&\(N\)&\(B_0(n)\)&\(B_1(n)\)&\(B_2(n)\)&\(|E_n|\)&\(|L_n|\)&\(|X_n|\)\\
\midrule
1&2&3&3&0&1&1&0\\
2&4&72&36&12&20&11&9\\
3&8&1248&528&240&336&93&243\\
4&16&20352&8256&4032&5440&715&4725\\
5&32&327168&131328&65280&87296&5261&82035\\
6&64&5240832&2098176&1047552&1397760&37851&1359909\\
\bottomrule
\end{tabular}
\end{table}

\begin{corollary}[Number of nonlocal equilateral triangles]
For \(n\ge1\),
\begin{equation}\label{eq:Xn}
  |X_n|
   =|E_n|-|L_n|
   =\frac{16^n+3\cdot4^n-4\cdot7^n}{12}.
\end{equation}
The first values are
\[
  0,9,243,4725,82035,1359909,\ldots,
\]
and
\[
  \frac{|L_n|}{|E_n|}
  \sim4\left(\frac7{16}\right)^n.
\]
Thus locality is not asymptotically typical.
\end{corollary}

\begin{remark}[Global cycles]
For \(b\ne e_n\), the global orbit
\(\{b,\gamma(b),\gamma^2(b)\}\) is the local cycle represented by the
canonical selected set
\[
  S=\supp(b).
\]
This is the global special case of
Proposition~\ref{prop:C3-local-representatives}.  Since \(e_n\) is the
unique word fixed by \(\gamma\), the \(C_3\)-action on the remaining
\(4^n-1\) words is free, giving
\[
  \frac{4^n-1}{3}
\]
nontrivial global cycles.  The same family reappears in
Section~\ref{sec:multiplication} as the multiplication-generated
global-\(\gamma\) family.
\end{remark}

\begin{corollary}[Central-vertex triangles are nonlocal]
\label{cor:central-vertex-nonlocal}
For every $n\geq1$, the central-vertex family $V_n$ of
Proposition~\ref{prop:central-vertex} satisfies
\[
 V_n\cap L_n=\varnothing,\qquad V_n\subseteq X_n.
\]
\end{corollary}
\begin{proof}
Every local triangle has a nonempty selected set $S$.  At each
position in $S$, its three vertices have the distinct digits
$i,j,k$.  Therefore none of its vertices can equal the all-$e$
word $e_n$, and no triangle in $V_n$ is local.
\end{proof}

\section{Product points and product-centered triangles}\label{sec:product-points}

Definition~1.1 applies to both local and nonlocal equilateral triangles.
For local cycles, the product and center statements needed here are direct
consequences of \cite[Theorem~4.8]{Dement2026}.

As an immediate corollary of
\cite[Theorem~4.8(2)]{Dement2026}, we obtain the unsigned product identity
in the notation used here.

\begin{corollary}[Unsigned product identity for local cycles]
\label{cor:local-product}
Let \(T=T(b,S)\) be a local triangle, with
\(\varnothing\ne S\subseteq\supp(b)\), and let \(c=c(b,S)\) be obtained
from \(b\) by replacing the digits in \(S\) by \(e\).  Then
\[
  p(T)=c,\qquad C_T=P(c).
\]
\end{corollary}

\begin{proof}
By \cite[Theorem~4.8(2)]{Dement2026}, the cyclically ordered product of the
three vertices is \((-1)^{\nu(b)}c\).  Removing the global sign gives
\(p(T)=c\), and therefore \(C_T=P(c)\).
\end{proof}

Likewise, the center criterion \cite[Theorem~4.8(4)]{Dement2026} gives
directly, in the present notation, that \(T(b,S)\) is product-centered if
and only if
\begin{equation}\label{eq:pc-parity}
  \#\{s\in S:s<r\}\equiv0\pmod2
  \qquad
  \text{for every \(r\notin S\) with \(b_r\ne e\)}.
\end{equation}

The parity condition gives an exact count over all local cycles.

\begin{theorem}[Number of product-centered local triangles]\label{thm:pc-local-count}
Let \(H_0=1\), \(H_1=7\), and for \(n\ge2\) let
\begin{equation}\label{eq:Hrec}
  H_n=5H_{n-1}+5H_{n-2}.
\end{equation}
Then
\begin{equation}\label{eq:pc-local-count}
  |P_n\cap L_n|=\frac{H_n-4^n}{3}.
\end{equation}
The first values are
\[
  1,8,57,373,2342,14343,\ldots.
\]
\end{theorem}

\begin{proof}
Count admissible pairs \((b,S)\), temporarily allowing \(S=\varnothing\),
and track the parity of the number of selected coordinates already read.
Let \(E_m\) and \(O_m\) be the weighted counts after \(m\) coordinates with
even and odd selected parity.  At even parity, an unselected coordinate may
contain any of \(i,j,k,e\), while a selected coordinate has three choices
and changes parity.  At odd parity, \eqref{eq:pc-parity} forces an
unselected coordinate to be \(e\), while a selected coordinate again has
three choices and changes parity.  Thus
\[
  \binom{E_{m+1}}{O_{m+1}}
   =
   \begin{pmatrix}4&3\\3&1\end{pmatrix}
   \binom{E_m}{O_m},
  \qquad
  \binom{E_0}{O_0}=\binom10.
\]
Writing \(H_m=E_m+O_m\), the characteristic polynomial is
\(x^2-5x-5\), so \eqref{eq:Hrec} holds with \(H_0=1\), \(H_1=7\).
The empty selection contributes \(4^n\) pairs and must be removed.
The remaining admissible pairs form a \(C_3\)-invariant subset of
\(\widetilde L_n\): the action changes only selected digits, while
\eqref{eq:pc-parity} depends on \(S\) and on the unchanged digits outside
\(S\).  Thus the recurrence counts \(H_n-4^n\) admissible representatives,
and Proposition~\ref{prop:C3-local-representatives} passes from those
representatives to triangles by division by \(3\).  This gives
\eqref{eq:pc-local-count}.
\end{proof}

For comparison, exhaustive exact enumeration by
\eqref{eq:zplus}--\eqref{eq:kappa-lookup} gives the split in
Table~\ref{tab:pc} through order \(6\).  No closed formula for the nonlocal
column is asserted here.

\begin{table}[htbp]
\centering
\caption{Product-centered equilateral triangles for small orders.}
\label{tab:pc}
\begin{tabular}{crrr}
\toprule
\(n\)&\(|P_n|\)&\(|P_n\cap L_n|\)&\(|P_n\cap X_n|\)\\
\midrule
1&1&1&0\\
2&8&8&0\\
3&72&57&15\\
4&598&373&225\\
5&4622&2342&2280\\
6&34638&14343&20295\\
\bottomrule
\end{tabular}
\end{table}

\section{Main-axis geometry and reflection symmetry}\label{sec:reflection}

The centroid map intertwines the digitwise \(S_3\)-action with the
dihedral symmetries of the triangular tiling \cite{Dement2026}.  In
particular, each transposition of \(i,j,k\) acts as reflection in one of the
three main triangular axes.

Fix one such transposition \(\tau\), and let
\[
  E_n^{\mathrm{sym}}(\tau)=\{T\in E_n:\tau(T)=T\}.
\]
Let \(E_n^{\mathrm{sym}}\) be the union over the three transpositions, so \(E_n^{\mathrm{sym}}\) is the
set of equilateral centroid triangles symmetric about at least one main
axis.

There is also a weaker axis condition involving only the product point.
Let \(f_\tau\in\{i,j,k\}\) be the unique noncentral digit fixed by \(\tau\),
and put
\[
  L_n^{\mathrm{pp\text{-}ax}}(\tau)=\{T\in L_n:C_T
     \text{ lies on the main axis fixed by }\tau\}.
\]
The triangle itself need not be \(\tau\)-invariant.

\begin{proposition}[Local triangles with product point on a fixed axis]\label{prop:Y-axis}
For any one of the three main axes,
\begin{equation}\label{eq:Y-axis}
  |L_n^{\mathrm{pp\text{-}ax}}(\tau)|=\frac{5^n-2^n}{3}.
\end{equation}
\end{proposition}

\begin{proof}
Let \(T=T(b,S)\) be a local triangle.  By
Corollary~\ref{cor:local-product},
\(C_T=P(c(b,S))\), where \(c(b,S)\) is obtained from \(b\) by replacing the
selected digits by \(e\).  The injective centroid map intertwines \(\tau\)
with reflection in the corresponding main axis.  Therefore \(C_T\) lies on
that axis if and only if
\[
  \tau(c(b,S))=c(b,S).
\]
Since \(\tau\) fixes exactly \(e\) and \(f_\tau\) among the four symbolic
digits, this is equivalent to
\[
  b_r\in\{e,f_\tau\}
  \qquad\text{for every }r\notin S.
\]
The set of local-cycle representatives \((b,S)\in\widetilde L_n\) of
Proposition~\ref{prop:C3-local-representatives} satisfying this condition is
\(C_3\)-invariant, since \(\gamma_S\) changes only selected coordinates.
For fixed nonempty \(S\), it has
\(3^{|S|}2^{n-|S|}\) representatives, and therefore its total cardinality is
\[
  \sum_{\varnothing\ne S\subseteq\{1,\ldots,n\}}
    3^{|S|}2^{n-|S|}
   =(2+3)^n-2^n
   =5^n-2^n.
\]
Proposition~\ref{prop:C3-local-representatives} now gives
\[
  |L_n^{\mathrm{pp\text{-}ax}}(\tau)|=\frac{5^n-2^n}{3}.
\]
\end{proof}

\begin{remark}[A reverse recurrence-to-geometry step]
The count \eqref{eq:Y-axis} is also the \(r=3\) member of the translated
Jacobsthal family in Appendix~A.  Once the neighboring cases \(r=2,4,5\)
had been identified with existing triangle classes, the missing \(r=3\)
term suggested looking for a local class with exactly two allowed
unselected digits.  The fixed alphabet \(\{e,f_\tau\}\) supplies precisely
such a class.  The proof above is independent of that heuristic.
\end{remark}

\subsection{Axis-supported equilateral triangles}\label{subsec:axis-supported}

We now impose a condition on the vertices themselves rather than only on a
product point or a reflection symmetry.  Let
\[
  E_n^{\mathrm{ax}}
  =
  \{\,T\in E_n:\text{ each vertex lies on a main axis}\,\}.
\]
The identity word \(e_n\), whose centroid is the common axis intersection, is
included in the three axes.

We now make use of Remark~\ref{rem:main-axes-coordinates} and determine the
possible integer values on these axes.  Put \(N=2^n\) and \(m=2^{n-1}\).
Let \(\tau_I\) be the transposition of \(j\) and \(k\), whose geometric action
is reflection in the \(I\)-axis.  By equivariance and injectivity of the
centroid map,
\[
  P(b)\text{ lies on the \(I\)-axis}
  \quad\Longleftrightarrow\quad
  \tau_I(b)=b.
\]
Since \(\tau_I\) fixes exactly the digits \(i\) and \(e\), the \(I\)-axis
words are exactly the words in \(\{i,e\}^n\), and their scaled integer
centroid coordinates have the form \(\Lambda_n(b)=(x,0)\).  By digitwise
\(S_3\)-equivariance, the corresponding coordinates on the other two axes
have the forms \((0,y)\) and \((-z,-z)\).

\begin{lemma}[Scaled integer coordinates on a main axis]\label{lem:axis-scalars}
Let \(R_n\) denote the set of parameters \(x\in\mathbb Z\) for which
\((x,0)\) occurs as an \(I\)-axis coordinate.  Then
\begin{equation}\label{eq:axis-scalars}
  R_n
  =
  \{\,x\in\mathbb Z:
      1-m\le x\le2m-1,\quad
      x\equiv0\text{ or }m\pmod3\,\}.
\end{equation}
The corresponding \(J\)- and \(K\)-axis coordinates are respectively
\((0,x)\) and \((-x,-x)\), with the same parameter \(x\in R_n\).
In particular, \(|R_n|=2m=2^n\).
\end{lemma}

\begin{proof}
Consider the \(I\)-axis.  Its words are exactly the \(2^n=2m\) words in
\(\{i,e\}^n\).  Since \(\Lambda_n\) is injective, their integer coordinates
are \(2m\) distinct values.

For every such coordinate \(x\), \eqref{eq:Lambda} shows that each positive
contribution is bounded by its corresponding dyadic weight, so
\[
  x\le 2^{n-1}+2^{n-2}+\cdots+1=2^n-1=2m-1.
\]
A negative contribution can occur only after at least one earlier
\(e\)-digit, so the total magnitude of all possible negative contributions
is at most
\[
  2^{n-2}+2^{n-3}+\cdots+1=m-1.
\]
Hence every \(I\)-axis coordinate \(x\) satisfies
\[
  1-m\le x\le2m-1.
\]
Since \(Y=0\) on this axis, Lemma~\ref{lem:orientation-residues} says
\[
  x\equiv N-1\quad\text{or}\quad N-2\pmod3.
\]
Because \(N=2m\) and \(m\not\equiv0\pmod3\), these two residues are precisely
\(0\) and \(m\pmod3\).  The interval \([1-m,2m-1]\) contains exactly
\(2m\) integers in these two residue classes.  Since the \(I\)-axis already
supplies \(2m\) distinct coordinate values, they are precisely the integers
in \eqref{eq:axis-scalars}.

Applying the digitwise cycle
\(\gamma:i\mapsto j\mapsto k\mapsto i\) sends an \(I\)-axis word with
coordinate \((x,0)\) successively to \(J\)- and \(K\)-axis words with
coordinates \((0,x)\) and \((-x,-x)\).  Hence the same parameter set \(R_n\)
describes all three axes in the stated forms.
\end{proof}

\begin{theorem}[Axis-supported equilateral triangles]\label{thm:axis-supported}
For \(n\ge1\),
\begin{equation}\label{eq:axis-total}
  |E_n^{\mathrm{ax}}|
  =4^{n-1}+2^n-2.
\end{equation}
More precisely, after labeling the vertices by the three main axes in scaled
integer centroid coordinates as
\[
  P=(x,0),\qquad Q=(0,y),\qquad R=(-z,-z),
  \qquad x,y,z\in R_n,
\]
every nondegenerate triangle in \(E_n^{\mathrm{ax}}\) belongs to exactly one
of the two branches
\[
  E_n^{\mathrm{ax},=}
   :=\{\,T\in E_n^{\mathrm{ax}}:x=y=z\,\},
  \qquad
  E_n^{\mathrm{ax},0}
   :=\{\,T\in E_n^{\mathrm{ax}}:x+y+z=0\,\},
\]
and
\begin{align}
  |E_n^{\mathrm{ax},=}|&=2^n-1,\label{eq:axis-equal-branch}\\
  |E_n^{\mathrm{ax},0}|&=4^{n-1}-1.\label{eq:axis-zero-branch}
\end{align}
\end{theorem}

\begin{proof}
Using \(q(u,v)=u^2-uv+v^2\) for squared lengths, up to the common scale
factor, we first justify labeling one vertex on each main axis.  If two
distinct vertices lie on the same axis, by symmetry write them as
\[
  P=(a,0),\qquad Q=(b,0),
\]
and put the third vertex on a second axis, say \(R=(0,c)\).  Then
\[
  q(R-P)=a^2+ac+c^2,
  \qquad
  q(R-Q)=b^2+bc+c^2.
\]
Equality of these squared lengths gives
\[
  (a-b)(a+b+c)=0.
\]
Since \(P\ne Q\), we have \(a\ne b\), so \(a+b+c=0\).  Substituting
\(c=-a-b\) and comparing with \(q(Q-P)=(a-b)^2\) then gives \(3ab=0\).
Thus one of the two same-axis vertices is the origin.  Since the origin lies
on all three main axes, it may be assigned to the remaining axis.  The case
in which all three vertices lie on one axis is degenerate.  Therefore every
nondegenerate member of \(E_n^{\mathrm{ax}}\) admits a labeling
\[
  P=(x,0),\qquad Q=(0,y),\qquad R=(-z,-z),
  \qquad x,y,z\in R_n.
\]
For a nondegenerate triangle this labeling is unique once the three axis
names are fixed.  If no vertex is the origin, each vertex lies on exactly
one main axis.  If one vertex is the origin, the other two nonzero vertices
lie on distinct axes and determine those two axis labels, while the origin
is assigned to the third.

For this labeling, the three squared side lengths, up to the same common
scale factor, are
\[
  q(Q-P)=x^2+xy+y^2,
\]
\[
  q(R-P)=x^2+xz+z^2,
  \qquad
  q(R-Q)=y^2+yz+z^2.
\]
Thus equality of the three side lengths gives
\[
  (y-z)(x+y+z)=0,
  \qquad
  (x-z)(x+y+z)=0.
\]
Hence either
\[
  x=y=z
  \qquad\text{or}\qquad
  x+y+z=0.
\]
The two branches meet only at \(x=y=z=0\), which is the degenerate origin
triple.

For the equal-coordinate branch, every nonzero \(x\in R_n\) gives the
global \(\gamma\)-orbit obtained from the unique \(I\)-axis word with axis
coordinate \(x\).  Therefore
\[
  |E_n^{\mathrm{ax},=}|=|R_n|-1=2^n-1.
\]

It remains to count the zero-sum branch.  We first include the degenerate
solution \((0,0,0)\).  Since \(m=2^{n-1}\), either
\(m=3h+1\) or \(m=3h+2\).  In all four shifted counting problems below,
the original interval bounds become \(0\le A,B,C\le m-1\), while the
prescribed sum is either \(m-1\) or \(m-2\).  Hence the upper bounds are
automatic and ordinary stars and bars applies.

If \(m=3h+1\), then \eqref{eq:axis-scalars} splits as
\[
  R_n
  =
  \{\,3a:-h\le a\le2h\,\}
  \ \cup\
  \{\,3a+1:-h\le a\le2h\,\}.
\]
For the residue-\(1\) class, for example,
\(-3h\le3a+1\le6h+1\) is equivalent to
\(-h-\tfrac13\le a\le2h\), hence to \(-h\le a\le2h\) because \(a\) is an
integer.
A zero-sum triple must use three entries from the same residue class.
For three residue-\(0\) entries, \(a+b+c=0\).  After shifting
\[
  A=a+h,\qquad B=b+h,\qquad C=c+h,
\]
we obtain
\[
  A+B+C=3h=m-1,
  \qquad A,B,C\ge0.
\]
Thus stars and bars gives
\[
  \binom{m+1}{2}
\]
solutions.  For three residue-\(1\) entries,
\(a+b+c=-1\), and the same shift gives
\[
  A+B+C=3h-1=m-2,
\]
hence
\[
  \binom m2
\]
solutions.

If \(m=3h+2\), then
\[
  R_n
  =
  \{\,3a:-h\le a\le2h+1\,\}
  \ \cup\
  \{\,3a+2:-h-1\le a\le2h\,\}.
\]
Again a zero-sum triple must lie wholly in one residue class.  The
residue-\(0\) class gives, after shifting by \(h\),
\[
  A+B+C=3h=m-2,
\]
and therefore \(\binom m2\) solutions.  For the residue-\(2\) class,
\(a+b+c=-2\); shifting each variable by \(h+1\) gives
\[
  A+B+C=3h+1=m-1,
\]
and therefore \(\binom{m+1}{2}\) solutions.

In either case the number of zero-sum triples including the origin triple is
\[
  \binom{m+1}{2}+\binom m2=m^2=4^{n-1}.
\]
Removing \((0,0,0)\) gives \eqref{eq:axis-zero-branch}.  Adding the two
nondegenerate branches proves \eqref{eq:axis-total}.
\end{proof}

\medskip
\noindent\textbf{Order-two example.}
For \(n=2\), Lemma~\ref{lem:axis-scalars} gives
\(R_2=\{-1,0,2,3\}\).  The equal-coordinate choice \(x=y=z=2\)
gives the global \(\gamma\)-orbit
\[
  \{ie,je,ke\},
  \qquad
  (2,0),\ (0,2),\ (-2,-2),
\]
where the displayed pairs are the scaled centroid coordinates.  The
zero-sum choice \((x,y,z)=(-1,-1,2)\) gives instead
\[
  \{ei,ej,ke\},
  \qquad
  (-1,0),\ (0,-1),\ (-2,-2),
\]
and the other permutations of \((-1,-1,2)\) give the remaining two
zero-sum triangles.  Thus already in order two the two branches contribute
three triangles each, for \(|E_2^{\mathrm{ax}}|=6\).  Order-three examples
of the equal-coordinate and zero-sum branches are given in Appendix~B.

\begin{remark}[A square hidden in the zero-sum count]
The identity
\[
  \binom{m+1}{2}+\binom m2=m^2
\]
shows that the two residue classes combine into a square count.  Thus the
zero-sum axis solutions, with the degenerate origin triple restored, are
equinumerous with an \(m\times m\) square, or equivalently with the
\(4^{n-1}\) order-\((n-1)\) base words.  Finding a natural digitwise
bijection would give a combinatorial explanation of this square count.
\end{remark}

We next count reflection-symmetric triangles in an ordinary finite
triangular point set.  We use the barycentric convention introduced in the
Introduction; Figure~\ref{fig:translated-triangular-point-set} illustrates
the two branches of the count.

\begin{figure}[H]
\centering
\begin{tikzpicture}[font=\small, line cap=round, line join=round,scale=0.72]
  \definecolor{shellblue}{RGB}{42,103,160}
  \definecolor{shellorange}{RGB}{188,95,34}
  \definecolor{gridgray}{RGB}{194,202,210}
  \coordinate (A) at (-3.25,0);
  \coordinate (B) at ( 3.25,0);
  \coordinate (C) at (0,5.62917);
  \foreach \j in {1,...,5}{
    \pgfmathsetmacro{\f}{\j/6}
    \draw[gridgray,very thin] ($(C)!\f!(A)$)--($(C)!\f!(B)$);
    \draw[gridgray,very thin] ($(A)!\f!(B)$)--($(A)!\f!(C)$);
    \draw[gridgray,very thin] ($(B)!\f!(A)$)--($(B)!\f!(C)$);
  }
  \draw[black!70,semithick] (A)--(B)--(C)--cycle;
  \foreach \a in {0,...,6}{
    \foreach \b in {0,...,6}{
      \pgfmathtruncatemacro{\c}{6-\a-\b}
      \ifnum\c>-1
        \pgfmathsetmacro{\xx}{3.25*(\b-\a)/6}
        \pgfmathsetmacro{\yy}{5.62917*\c/6}
        \fill[black!46] (\xx,\yy) circle (1.15pt);
      \fi
    }
  }
  \draw[black!65,densely dashed,thin] (0,-0.15)--(0,5.72);
  \coordinate (P) at (0,1.87639);
  \coordinate (Qp) at (-1.08333,0);
  \coordinate (Rp) at (1.08333,0);
  \fill[shellblue,opacity=.09] (P)--(Qp)--(Rp)--cycle;
  \draw[shellblue,thick] (P)--(Qp)--(Rp)--cycle;
  \coordinate (Qm) at (0.54167,2.814585);
  \coordinate (Rm) at (-0.54167,2.814585);
  \fill[shellorange,opacity=.09] (P)--(Qm)--(Rm)--cycle;
  \draw[shellorange,thick] (P)--(Qm)--(Rm)--cycle;
  \fill[black] (P) circle (2.3pt);
  \foreach \point in {Qp,Rp}{\fill[shellblue] (\point) circle (2.35pt);}
  \foreach \point in {Qm,Rm}{\fill[shellorange] (\point) circle (2.35pt);}
  \node[anchor=west,fill=white,inner sep=1.3pt] at (0.10,1.87) {$P=(2,2,2)$};
  \node[anchor=east,text=shellblue] at (-1.15,-.19) {$Q_{+}$};
  \node[anchor=west,text=shellblue] at (1.15,-.19) {$R_{+}$};
  \node[anchor=east,text=shellorange] at (-.65,3.04) {$R_{-}$};
  \node[anchor=west,text=shellorange] at (.65,3.04) {$Q_{-}$};
  \node[anchor=south] at (0,5.80) {$(0,0,6)$};
  \node[anchor=north east] at (-3.25,-.22) {$(6,0,0)$};
  \node[anchor=north west] at (3.25,-.22) {$(0,6,0)$};
  \node[anchor=west] at (.07,4.48) {$a=b$};
  \node[anchor=west,align=left,font=\footnotesize] at (3.58,4.1) {$m=6,\quad k=7$\\[3pt]
    $P=(t,t,m-2t),\ t=2$\\[7pt]
    \textcolor{shellblue}{$r=+2$: toward $c=0$}\\[3pt]
    \textcolor{shellorange}{$r=-1$: toward $c>2$}};
\end{tikzpicture}
\caption{The barycentric patch $\Omega_6$ ($k=7$), also a model of the
order-three downward centroid grid. Reflection exchanges $a$ and $b$;
$P=(2,2,2)$ is fixed. The two highlighted equilateral triangles show the
$r>0$ and $r<0$ cases, yielding $m-2t$ and $t$ possibilities at a general
fixed point $(t,t,m-2t)$.}
\label{fig:translated-triangular-point-set}
\end{figure}

\begin{lemma}[Axis-symmetric triangles in $T_k$]
Let $f(k)$ be the number of equilateral triangles in the $k$-row triangular
point set $T_k$ that are symmetric about one fixed main axis. Then
\begin{equation}\label{eq:f-axis}
f(k)=
\begin{cases}
\dfrac{k(3k-2)}8,& k\text{ even},\\[5pt]
\dfrac{3(k^2-1)}8,& k\text{ odd}.
\end{cases}
\end{equation}
\end{lemma}
\begin{proof}
Put $m=k-1$ and identify $T_k$ with the barycentric triangular grid
(also used in Section~11). With the chosen main reflection, its fixed
vertices are
\[
\begin{aligned}
\Omega_m&=\{(a,b,c)\in\mathbb Z_{\ge0}^3:a+b+c=m\},\\
\rho(a,b,c)&=(b,a,c),\\
P&=(t,t,m-2t),\qquad 0\le t\le\lfloor m/2\rfloor.
\end{aligned}
\]
Every nondegenerate reflection-invariant equilateral triangle has exactly one
fixed vertex $P$ and two vertices exchanged by $\rho$. Write
\[
\begin{aligned}
Q&=(t+r,t+s,m-2t-r-s),\\
R&=(t+s,t+r,m-2t-r-s).
\end{aligned}
\]
The two grid directions meet at $60^\circ$, with metric
$q_{60}(x,y)=x^2+xy+y^2$. Reflection gives $PQ=PR$, so equilateralness
is equivalent to
\[
q_{60}(r,s)=q_{60}(s-r,r-s),\quad\text{i.e.}\quad
r^2+rs+s^2=(r-s)^2\iff rs=0.
\]
Interchanging $Q$ and $R$ interchanges $r$ and $s$; hence we may take $s=0$,
with $r\ne0$. If $r=d>0$, the condition that $Q,R\in\Omega_m$ gives
$1\le d\le m-2t$. If $r=-d<0$, it gives $1\le d\le t$.
Thus $P$ is a vertex of exactly $(m-2t)+t=m-t$ such triangles. Summing,
\[
f(k)=\sum_{t=0}^{\lfloor m/2\rfloor}(m-t)
=\begin{cases}
\dfrac{h(3h-1)}2,&k=2h,\\[5pt]
\dfrac{3h(h+1)}2,&k=2h+1,
\end{cases}
\]
which is~\eqref{eq:f-axis}. The same formula covers $k=1$.
\end{proof}

\begin{theorem}[Reflection-symmetric counts]\label{thm:reflection-counts}
For any fixed one of the three main axes,
\begin{align}
  |E_n^{\mathrm{sym}}(\tau)|&=3\cdot4^{n-1}-2^n,\label{eq:Aaxis}\\
  |E_n^{\mathrm{sym}}(\tau)\cap L_n|&=3^n-2^n,\label{eq:Aaxis-local}\\
  |E_n^{\mathrm{sym}}(\tau)\cap X_n|&=3\cdot4^{n-1}-3^n,\label{eq:Aaxis-nonlocal}\\
  |E_n^{\mathrm{sym}}(\tau)\cap L_n\cap P_n|&=F_{2n+2}-2^n.\label{eq:Aaxis-pc}
\end{align}
Counting each triangle only once when it is symmetric about more than one
axis gives
\begin{align}
  |E_n^{\mathrm{sym}}|&=9\cdot4^{n-1}-5\cdot2^n+2,\label{eq:Aall}\\
  |E_n^{\mathrm{sym}}\cap L_n|&=3^{n+1}-5\cdot2^n+2,\label{eq:Aall-local}\\
  |E_n^{\mathrm{sym}}\cap X_n|&=9\cdot4^{n-1}-3^{n+1},\label{eq:Aall-nonlocal}\\
  |E_n^{\mathrm{sym}}\cap L_n\cap P_n|
    &=3F_{2n+2}-5\cdot2^n+2.\label{eq:Aall-pc}
\end{align}
\end{theorem}

\begin{proof}
The upward and downward centroid classes are copies of \(T_N\) and
\(T_{N-1}\), respectively, with \(N=2^n\), and their main reflection axes
are aligned with the corresponding digitwise transpositions.  The preceding
lemma gives
\[
  |E_n^{\mathrm{sym}}(\tau)|=f(N)+f(N-1)=3\cdot4^{n-1}-2^n.
\]

Since an involution acting on a three-element set fixes one element, a
\(\tau\)-invariant local triangle has a vertex fixed by \(\tau\).  For
fixed nonempty \(S\), after choosing this fixed vertex every selected digit
is \(f_\tau\), while each common unselected digit lies in
\(\{e,f_\tau\}\).  Thus \(S\) contributes \(2^{n-|S|}\) symmetric local
triangles, and
\[
  \sum_{\varnothing\ne S\subseteq\{1,\ldots,n\}}2^{n-|S|}
   =3^n-2^n.
\]
Subtracting from \eqref{eq:Aaxis} gives \eqref{eq:Aaxis-nonlocal}.

For product-centered symmetric local triangles, encode \(S\) by a binary
word.  At even selected parity an unselected coordinate may be either
\(e\) or \(f_\tau\); at odd parity the criterion
\eqref{eq:pc-parity} forces it to be \(e\).  A selected coordinate changes
parity.  Hence the two-state matrix is
\[
  \begin{pmatrix}2&1\\1&1\end{pmatrix}
  =
  \begin{pmatrix}1&1\\1&0\end{pmatrix}^{\!2}.
\]
Starting from even parity, the total weight of all length-\(n\) patterns is
\(F_{2n+2}\).  The empty selection contributes \(2^n\), proving
\eqref{eq:Aaxis-pc}.

It remains to correct overlaps among the three axes.  A triangle invariant
under two distinct reflections is invariant under the full digitwise
\(S_3\)-action.  Hence the three pairwise intersections of the axis-fixed
families all coincide with their triple intersection.  A nondegenerate
three-point \(S_3\)-invariant set must be a size-\(3\) orbit.  Such an orbit
is obtained by choosing a nonempty support \(S\), taking a word equal to one
fixed noncentral digit on \(S\) and \(e\) outside \(S\), and applying the
global \(S_3\)-action.  Hence there are \(2^n-1\) such triangles.  They are
local and product-centered.  Inclusion-exclusion therefore gives
\[
  3|E_n^{\mathrm{sym}}(\tau)|-3(2^n-1)+(2^n-1)
   =3|E_n^{\mathrm{sym}}(\tau)|-2(2^n-1),
\]
and the same correction applies to the local and product-centered local
subfamilies, yielding \eqref{eq:Aall}--\eqref{eq:Aall-pc}.
\end{proof}

\begin{remark}[Reflection-restricted product-centered counts]\label{rem:reflection-pc-enumeration}
The first four data columns of Table~\ref{tab:reflection} are given by
Theorem~\ref{thm:reflection-counts}.  The final two columns were obtained by
exhaustive exact enumeration through order \(6\).  No closed formula is
presently known for either
\[
  |E_n^{\mathrm{sym}}\cap P_n|
  \qquad\text{or}\qquad
  |E_n^{\mathrm{sym}}\cap X_n\cap P_n|.
\]
Since
\[
  |E_n^{\mathrm{sym}}\cap X_n\cap P_n|
   =|E_n^{\mathrm{sym}}\cap P_n|-|E_n^{\mathrm{sym}}\cap L_n\cap P_n|,
\]
and the local term is known exactly from Theorem~\ref{thm:reflection-counts},
a closed formula for either of the two unknown sequences would immediately
give the other.  Determining such a formula is left for future work.

At \(n=6\), the exact local product-centered count and the enumerated
nonlocal product-centered count coincide:
\[
  |E_6^{\mathrm{sym}}\cap L_6\cap P_6|
   =|E_6^{\mathrm{sym}}\cap X_6\cap P_6|
   =813.
\]
No structural explanation for this equality is presently known.  It may be
an isolated finite-order coincidence, reminiscent of the accidental
order-two symmetry enhancement in the preceding paper
\cite[Subsection~5.3]{Dement2026}.

\end{remark}

\begin{table}[htbp]
\centering
\caption{Triangles symmetric about at least one main axis.}
\label{tab:reflection}
\scriptsize
\setlength{\tabcolsep}{3.5pt}
\begin{tabular}{crrrrrr}
\toprule
\(n\)
&\(|E_n^{\mathrm{sym}}|\)
&\(|E_n^{\mathrm{sym}}\cap L_n|\)
&\(|E_n^{\mathrm{sym}}\cap X_n|\)
&\(|E_n^{\mathrm{sym}}\cap L_n\cap P_n|\)
&\(|E_n^{\mathrm{sym}}\cap P_n|\)
&\(|E_n^{\mathrm{sym}}\cap X_n\cap P_n|\)\\
\midrule
1&1&1&0&1&1&0\\
2&18&9&9&6&6&0\\
3&106&43&63&25&34&9\\
4&498&165&333&87&138&51\\
5&2146&571&1575&274&490&216\\
6&8898&1869&7029&813&1626&813\\
\bottomrule
\end{tabular}
\end{table}

\FloatBarrier
\section{The no-\texorpdfstring{\(e\)}{e} Sierpiński support}\label{sec:sierpinski}

Let
\[
  S_n=\{i,j,k\}^n\subseteq\Delta_n,
\]
and write
\[
  E_n^S=\{T\in E_n:T\subseteq S_n\},
  \qquad
  L_n^S=E_n^S\cap L_n.
\]
Its formal sum is the finite Sierpiński-type support \(S_E(n)\) introduced in
\cite[Example~5.9]{Dement2026}, where it is written in the octal alphabet
\(\{1,2,4\}\).  There it arises because it is fixed by every digit
permutation, hence symmetric with respect to all three triangular axes.
Here the same support appears because it is closed under every local cyclic
action.

\begin{proposition}[Local triangles inside \(S_n\)]
For every \(b\in S_n\) and every nonempty
\(S\subseteq\{1,\ldots,n\}\), the selected set \(S\) is admissible and
\[
  \gamma_S^t(b)\in S_n
  \qquad(t=0,1,2).
\]
The number of local equilateral triangles all of whose vertices lie in
\(S_n\) is
\begin{equation}\label{eq:LS}
  |L_n^S|=(2^n-1)3^{n-1}.
\end{equation}
\end{proposition}

\begin{proof}
There are no \(e\)-digits, so
\(\supp(b)=\{1,\ldots,n\}\), and \(\gamma\) permutes \(i,j,k\).
Thus the restriction of \(\widetilde L_n\) to \(b\in S_n\) consists of all
pairs
\[
  (b,S),\qquad b\in S_n,\quad
  \varnothing\ne S\subseteq\{1,\ldots,n\},
\]
and has cardinality \(3^n(2^n-1)\).  This subset is \(C_3\)-invariant, so
Proposition~\ref{prop:C3-local-representatives} gives
\[
  |L_n^S|
    =\frac{3^n(2^n-1)}3
    =(2^n-1)3^{n-1}.
\]
\end{proof}

\medskip
\noindent\textbf{OEIS cross-reference.}
The values
\[
  1,9,63,405,2511,15309,92583,\ldots
\]
of \(|L_n^S|\) form
\oeislink{https://oeis.org/A016137}{OEIS A016137} with the order index shifted
by one.  After adjoining the order-zero term \(0\), a comment by Paul Barry
on that entry identifies the sequence as the fourth binomial transform of
the Jacobsthal sequence
\oeislink{https://oeis.org/A001045}{A001045} \cite{OEIS}.  Appendix~A gives a
direct floretion realization of this relation using the order-two
recurrence-generator framework of \cite[Appendix~B]{Dement2026}.

\begin{proposition}[Fibonacci count for product-centered local triangles in \(S_n\)]
Let \(F_1=F_2=1\).  The number of product-centered local triangles contained
in \(S_n\) is
\begin{equation}\label{eq:CS}
  C_n^S=(F_{n+2}-1)3^{n-1}.
\end{equation}
\end{proposition}

\begin{proof}
The parity criterion \eqref{eq:pc-parity}, originating in
\cite[Theorem~4.8(4)]{Dement2026}, simplifies sharply on \(S_n\).
Every unselected digit is noncentral, so a local triangle
\(T(b,S)\subseteq S_n\) is product-centered if and only if
\begin{equation}\label{eq:Sn-parity}
  \#\{s\in S:s<r\}\equiv0\pmod2
  \qquad\text{for every }r\notin S.
\end{equation}
In particular, the condition depends only on \(S\), not on the starting
word \(b\).

Encode \(S\) by a binary word of length \(n\), writing \(1\) for a selected
coordinate and \(0\) otherwise.  Condition \eqref{eq:Sn-parity} says that a
\(0\) may occur only after an even number of preceding \(1\)'s.  Thus once
a \(1\) changes the running parity from even to odd, another \(1\) is
forced before a \(0\) can occur again.  Equivalently, the word is assembled
from blocks \(0\) and \(11\), with the possibility of one final unpaired
\(1\).

The number of such length-\(n\) binary words, including the all-zero word,
is \(F_{n+2}\).  Removing the empty selection leaves \(F_{n+2}-1\)
nonempty sets \(S\).

For a fixed nonempty \(S\), there are \(3^n\) choices of starting word
\(b\in S_n\).  The synchronized action
\[
  b\longmapsto\gamma_S(b)
\]
is free, since \(S\ne\varnothing\) and no coordinate of \(b\) is \(e\).
Its orbits therefore have size \(3\), and
Proposition~\ref{prop:C3-local-representatives} identifies each orbit with
one local triangle \(T(b,S)\).  Hence each admissible \(S\) contributes
\[
  \frac{3^n}{3}=3^{n-1}
\]
distinct local triangles.  Multiplying by the \(F_{n+2}-1\) admissible
nonempty selections gives \eqref{eq:CS}.
\end{proof}

The first values of \(C_n^S\) are
\[
  1,6,36,189,972,4860,24057,\ldots,
\]
and among local triangles in \(S_n\) the product-centered fraction is
\[
  \frac{F_{n+2}-1}{2^n-1}.
\]

As an independent computational check, exhaustive exact enumeration with
\eqref{eq:zplus}--\eqref{eq:kappa-lookup} gives
\[
  1,9,63,405,2511,15309,92583
\]
equilateral triples contained entirely in \(S_n\) for \(1\le n\le7\),
matching the local count \eqref{eq:LS}.  The argument below proves that this
agreement holds for every \(n\).

\begin{remark}[Digit-stripping strategy for Theorem~\ref{thm:noe-rigidity}]
\label{rem:noe-rigidity-strategy}
The main goal of this section is to show that locality is exhaustive on
\[
  S_n=\{i,j,k\}^n.
\]
To this end, let \(T=\{x,y,z\}\subseteq S_n\) be a nondegenerate
equilateral centroid triangle.  Choose an ordering of its vertices and
\(\varepsilon\in\{+,-\}\) such that, with
\[
  A=\Lambda_n(x),\qquad
  B=\Lambda_n(y),\qquad
  C=\Lambda_n(z),
\]
one has the oriented completion relation
\begin{equation}\label{eq:noe-oriented-completion}
  C=A+R_\varepsilon(B-A).
\end{equation}
Write the last digits of \(x,y,z\) as \(a,b,c\), and their prefixes as
\(x',y',z'\):
\[
  x=x'a,\qquad y=y'b,\qquad z=z'c,
  \qquad a,b,c\in\{i,j,k\}.
\]
Put
\[
  A'=\Lambda_{n-1}(x'),\qquad
  B'=\Lambda_{n-1}(y'),\qquad
  C'=\Lambda_{n-1}(z').
\]
Because no \(e\)-digit occurs, the sign factors in
\eqref{eq:Lambda} are all \(+1\), and the same last-digit decomposition
used in Lemma~\ref{lem:orientation-residues} becomes
\begin{equation}\label{eq:noe-last-digit}
  A=2A'+\lambda(a),\qquad
  B=2B'+\lambda(b),\qquad
  C=2C'+\lambda(c).
\end{equation}
Substituting \eqref{eq:noe-last-digit} into
\eqref{eq:noe-oriented-completion} and collecting prefix terms gives
\[
  2\bigl[C'-A'-R_\varepsilon(B'-A')\bigr]
  =
  \lambda(a)+R_\varepsilon\bigl(\lambda(b)-\lambda(a)\bigr)-\lambda(c).
\]
This motivates the notation
\begin{equation}\label{eq:Depsilon}
  D_\varepsilon(a,b)
  :=
  \lambda(a)+R_\varepsilon\bigl(\lambda(b)-\lambda(a)\bigr),
\end{equation}
so that
\begin{equation}\label{eq:noe-stripping-identity}
  2\bigl[C'-A'-R_\varepsilon(B'-A')\bigr]
   =D_\varepsilon(a,b)-\lambda(c).
\end{equation}
The left-hand side lies in \(2\mathbb Z^2\).  Hence
\begin{equation}\label{eq:noe-digit-congruence}
  D_\varepsilon(a,b)\equiv\lambda(c)
  \pmod{2\mathbb Z^2}.
\end{equation}
The next lemma shows that, on the no-\(e\) alphabet, considerably more is
true: the congruence forces an exact digit identity and fixes the cyclic
order of every varying coordinate.

The proof of Theorem~\ref{thm:noe-rigidity} will therefore follow the chain
\[
\begin{gathered}
\text{global oriented equilateral relation}\\
\Downarrow\\
\text{mod-\(2\) last-digit constraint}\\
\Downarrow\\
\text{exact digit identity}\\
\Downarrow\\
\text{the same oriented relation for the prefixes}\\
\Downarrow\\
\text{the same \(\varepsilon\) at every coordinate}\\
\Downarrow\\
\text{one common \(\gamma\)-direction}\\
\Downarrow\\
\text{locality by Proposition~\ref{prop:local-recognition}.}
\end{gathered}
\]
Thus the argument iterates an affine relation rather than inducting on
nondegenerate prefix triangles.  We shall see that the prohibition on
\(e\) is not just simplifying the coordinates; it is the actual source of
the rigidity.
\end{remark}

\begin{lemma}[No-\(e\) forcing of digit triples]\label{lem:noe-digit-forcing}
Let \(a,b,c\in\{i,j,k\}\) and \(\varepsilon\in\{+,-\}\).  Suppose
\[
  D_\varepsilon(a,b)\equiv\lambda(c)
  \pmod{2\mathbb Z^2}.
\]
Then exactly one of the following occurs:
\[
\begin{array}{c|cc}
 & b & c\\
\hline
\varepsilon=- & a & a\\
              & \gamma(a) & \gamma^2(a)\\[1mm]
\varepsilon=+ & a & a\\
              & \gamma^2(a) & \gamma(a).
\end{array}
\]
Moreover, in every allowed case the congruence upgrades to the exact
identity
\begin{equation}\label{eq:noe-digit-exact}
  D_\varepsilon(a,b)=\lambda(c)
  \qquad\text{in }\mathbb Z^2.
\end{equation}
\end{lemma}

\begin{proof}
Modulo \(2\), the three vectors
\[
  \lambda(i)=(1,0),\qquad
  \lambda(j)=(0,1),\qquad
  \lambda(k)=(-1,-1)
\]
are exactly the three nonzero elements of \(\mathbb F_2^2\).
Using
\[
  R_+(r,s)=(s,s-r),
  \qquad
  R_-(r,s)=(r-s,r),
\]
direct substitution in \eqref{eq:Depsilon} gives
\[
\begin{array}{c|ccc}
 & b=a & b=\gamma(a) & b=\gamma^2(a)\\
\hline
D_-(a,b) &
 \lambda(a) &
 \lambda(\gamma^2(a)) &
 0 \pmod{2}\\
D_+(a,b) &
 \lambda(a) &
 0 \pmod{2} &
 \lambda(\gamma(a)).
\end{array}
\]
Since \(\lambda(c)\) is never \(0\pmod2\), the zero-residue cases are
impossible.  In each remaining case, direct substitution of
\(\lambda(i),\lambda(j),\lambda(k)\) gives equality in \(\mathbb Z^2\).
For example,
\[
  D_-(i,j)=(-1,-1)=\lambda(k),
  \qquad
  D_+(i,j)=(2,2)\equiv(0,0)\pmod2.
\]
This proves both the stated alternatives and
\eqref{eq:noe-digit-exact}.
\end{proof}

\begin{theorem}[Locality is exhaustive on \(S_n\)]\label{thm:noe-rigidity}
Every nondegenerate equilateral centroid triangle contained in
\(S_n=\{i,j,k\}^n\) is a local \(\gamma\)-cycle.  Equivalently,
\begin{equation}\label{eq:ES=LS}
  E_n^S=L_n^S
  \qquad(n\ge1).
\end{equation}
Consequently,
\begin{equation}\label{eq:ES-count}
  |E_n^S|=(2^n-1)3^{n-1}.
\end{equation}
\end{theorem}

\begin{proof}
The case \(n=1\) is immediate.  Let \(n\ge2\) and take
\(T=\{x,y,z\}\in E_n^S\).  Use the ordering and notation of
Remark~\ref{rem:noe-rigidity-strategy}.  Equation
\eqref{eq:noe-stripping-identity} gives
\[
  2\bigl[C'-A'-R_\varepsilon(B'-A')\bigr]
   =D_\varepsilon(a,b)-\lambda(c),
\]
and reduction modulo \(2\) gives
\[
  D_\varepsilon(a,b)\equiv\lambda(c)\pmod{2\mathbb Z^2}.
\]
Lemma~\ref{lem:noe-digit-forcing} shows that the last coordinate is either
constant or contains \(i,j,k\) in the cyclic direction determined by the
single sign \(\varepsilon\).  More importantly, it upgrades the congruence
to the exact identity
\[
  D_\varepsilon(a,b)=\lambda(c).
\]
The right-hand side of \eqref{eq:noe-stripping-identity} therefore
vanishes, and
\begin{equation}\label{eq:noe-prefix-relation}
  C'=A'+R_\varepsilon(B'-A').
\end{equation}

Thus stripping the last digit preserves the same oriented affine relation
with the same \(\varepsilon\).  No nondegeneracy assumption is needed for
the intermediate prefixes: relation \eqref{eq:noe-prefix-relation} remains
valid whether the three prefix points are distinct or coincide.  Iterating
the relation coordinate by coordinate shows that every coordinate of the
ordered triple \((x,y,z)\) is either constant or contains \(i,j,k\), and
every varying coordinate has one common cyclic direction: the forward
\(\gamma\)-direction when \(\varepsilon=-\), and the reverse direction when
\(\varepsilon=+\).

Since \(T\) is nondegenerate, at least one coordinate varies.
Proposition~\ref{prop:local-recognition} now recognizes \(T\) as a local
\(\gamma\)-cycle.  Hence \eqref{eq:ES=LS} holds, and
\eqref{eq:ES-count} follows from \eqref{eq:LS}.
\end{proof}

\section{Multiplication-generated equilateral triangles}\label{sec:multiplication}

Before the local-cycle construction, examples from the online multiplier
already suggested a connection between multiplication and centroid
geometry.  For some pairs of positive base vectors, the unsigned product
gives a third base vector whose centroid completes an equilateral triangle.

\begin{definition}[Multiplication-generated triangle]
An equilateral triple \(T\in E_n\) is \emph{multiplication-generated} if its
vertices can be written
\[
  T=\{b,c,u(bc)\},
  \qquad b,c\in\Delta_n,
\]
where \(u(\pm d)=d\) denotes the unsigned base vector.  Write
\[
  \mathcal M_n
   :=\{T\in E_n:T\text{ is multiplication-generated}\}.
\]
\end{definition}

\begin{proposition}[Multiplication and the product point]
\label{prop:mult-product-point}
For \(T\in E_n\), the following are equivalent:
\begin{enumerate}[label=\textup{(\roman*)}]
\item \(T\) is multiplication-generated;
\item \(p(T)=e_n\);
\item \(C_T=0\).
\end{enumerate}
Equivalently,
\[
  \mathcal M_n
   =\{T\in E_n:p(T)=e_n\}
   =\{T\in E_n:C_T=0\}.
\]
\end{proposition}

\begin{proof}
Write \(T=\{x,y,z\}\).  If, after relabeling,
\(z=u(xy)\), then \(xy=\pm z\), and hence
\[
  xyz=\pm z^2=\pm e_n.
\]
Thus \(p(T)=e_n\).

Conversely, if \(p(T)=e_n\), then for any ordering of the three vertices
\[
  xyz=\pm e_n.
\]
Therefore
\[
  xy=\pm z^{-1}.
\]
Every basis word satisfies \(z^2=\pm e_n\), so \(z^{-1}=\pm z\).
Consequently \(xy=\pm z\), and therefore \(u(xy)=z\).  Thus \(T\) is
multiplication-generated.

Finally, \(C_T=P(p(T))\), while \(P(e_n)=0\).  Since the centroid map is
injective, \(C_T=0\) holds if and only if \(p(T)=e_n\).
\end{proof}

\begin{corollary}[Coordinatewise form of multiplication-generation]
\label{cor:mult-digit-patterns}
Let \(T=\{x,y,z\}\in E_n\).  Then \(T\) is multiplication-generated if and
only if, at every coordinate \(r\), the unordered digit triple
\[
  \{x_r,y_r,z_r\}
\]
has one of the forms
\[
  \{e,e,e\},
  \qquad
  \{e,a,a\}\quad(a\in\{i,j,k\}),
  \qquad
  \{i,j,k\}.
\]
\end{corollary}

\begin{proof}
By Proposition~\ref{prop:mult-product-point}, multiplication-generation is
equivalent to \(p(T)=e_n\), or equivalently to the product of the three
digits being central in every coordinate.  Among unordered triples from
\(\{i,j,k,e\}\), the displayed patterns are exactly those whose product is
\(\pm e\).
\end{proof}

Use the signed-centralizer notation of \cite{Dement2026}:
\[
\begin{aligned}
  C_{\mathrm{tiles}}(b)
    &=\{c\in\Delta_n:cb=bc\},\\
  C_+(b)
    &=\{c\in\Delta_n:cb=bc\text{ and }cb\in\Delta_n\},\\
  C_-(b)
    &=\{c\in\Delta_n:cb=bc\text{ and }cb\in-\Delta_n\},
\end{aligned}
\]
so that
\[
  C_{\mathrm{tiles}}(b)=C_+(b)\sqcup C_-(b).
\]
The signs in \(C_+\) and \(C_-\) refer to the sign of the common product,
not to commutation versus anticommutation.

\begin{corollary}[Commutation filter for multiplication-generated triangles]
\label{cor:mult-commutation-filter}
Suppose
\[
  T=\{b,c,d\}\in E_n,\qquad d=u(bc).
\]
Then
\[
  b^2=c^2=d^2=s e_n
  \qquad(s\in\{+1,-1\}),
\]
and
\[
  cb=s\,bc.
\]
Hence \(b\) and \(c\) commute when \(s=+1\) and anticommute when
\(s=-1\).  In the commuting case,
\[
  c\in C_{\mathrm{tiles}}(b)=C_+(b)\sqcup C_-(b).
\]
\end{corollary}

\begin{proof}
Theorem~\ref{thm:same-orientation} gives equality of the three squares.
By \cite[Lemma~6.7 and Corollary~6.8]{Dement2026}, positive basis words
commute or anticommute, so write
\[
  cb=\vartheta\,bc,\qquad\vartheta\in\{+1,-1\}.
\]
The sign in \(bc=\pm d\) disappears on squaring, and
\[
  d^2=(bc)^2=b(cb)c=\vartheta b^2c^2=\vartheta e_n.
\]
Comparing with \(d^2=s e_n\) gives \(\vartheta=s\), hence \(cb=s\,bc\).
\end{proof}

\begin{proposition}[Global \(\gamma\)-orbits are multiplication-generated]
\label{prop:global-mult}
For every \(b\in\Delta_n\setminus\{e_n\}\),
\[
  \{b,\gamma(b),\gamma^2(b)\}
\]
is a multiplication-generated local equilateral triangle, because
\begin{equation}\label{eq:gamma-product}
  b\,\gamma(b)=\gamma^2(b).
\end{equation}
The number of distinct triangles in this family is
\begin{equation}\label{eq:Gn}
  G_n=\frac{4^n-1}{3}.
\end{equation}
\end{proposition}

\begin{proof}
Coordinatewise,
\[
  ij=k,\qquad jk=i,\qquad ki=j,\qquad ee=e,
\]
all with positive sign, proving \eqref{eq:gamma-product}.  The orbit is
local with selected set \(S=\supp(b)\).  The unique fixed word of the
global \(\gamma\)-action is \(e_n\); all other \(4^n-1\) words lie in
orbits of size \(3\).
\end{proof}

\begin{corollary}[Local multiplication rigidity]
\label{cor:local-mult-rigidity}
For \(T\in L_n\), the following are equivalent:
\[
  T\in\mathcal M_n
  \qquad\Longleftrightarrow\qquad
  T=\{b,\gamma(b),\gamma^2(b)\}
  \text{ for some }b\ne e_n.
\]
Hence
\[
  |\mathcal M_n\cap L_n|=\frac{4^n-1}{3}.
\]
\end{corollary}

\begin{proof}
Write \(T=T(b,S)\) with the canonical condition
\(\varnothing\ne S\subseteq\supp(b)\).  By
Corollary~\ref{cor:local-product},
\[
  p(T)=c(b,S),
\]
where \(c(b,S)\) is obtained from \(b\) by replacing the selected digits by
\(e\).  Proposition~\ref{prop:mult-product-point} shows that \(T\) is
multiplication-generated exactly when \(c(b,S)=e_n\).  This occurs exactly
when every coordinate outside \(S\) is already \(e\).  Together with
\(S\subseteq\supp(b)\), that is equivalent to
\[
  S=\supp(b).
\]
Thus \(T\) is a global \(\gamma\)-orbit.  The converse follows from
Proposition~\ref{prop:global-mult}, and the count follows from
\eqref{eq:Gn}.
\end{proof}

Writing \(\nu(b)=\#\supp(b)\), the global-\(\gamma\) family splits according
to \(b^2=(-1)^{\nu(b)}e_n\).  The even- and odd-\(\nu\) word counts are
\[
  \sum_{\nu\ {\rm even}}\binom n\nu3^\nu
    =\frac{4^n+(-2)^n}{2},
  \qquad
  \sum_{\nu\ {\rm odd}}\binom n\nu3^\nu
    =\frac{4^n-(-2)^n}{2}.
\]
Removing \(e_n\) from the even count and dividing by the orbit size \(3\)
gives
\[
  G_+(n)=\frac{4^n+(-2)^n-2}{6},
  \qquad
  G_{\mathrm{anti}}(n)=\frac{4^n-(-2)^n}{6}.
\]
There are no \(C_-\)-examples in the global-\(\gamma\) family, since
\eqref{eq:gamma-product} always has positive sign.

Exhaustive exact enumeration through \(n=6\), using
\eqref{eq:zplus}--\eqref{eq:kappa-lookup} together with the unsigned-product
condition, had found
\[
  |\mathcal M_n\cap X_n|=0
  \qquad(1\le n\le6).
\]
These checks remain independent finite-order verification of the
all-order result subsequently proved by Ma and Zhang
\cite{MaZhang2026}.  The corresponding data, including the signed
commutation classification, are collected in Table~\ref{tab:mult}
after the pairwise-anticommuting class is formally defined in
Section~\ref{sec:signed}.

Before that proof, Corollary~\ref{cor:mult-digit-patterns} showed that a
hypothetical nonlocal example would have to contain a coordinate of type
\(\{e,a,a\}\), or else use only \(\{e,e,e\}\) and \(\{i,j,k\}\) while
failing the common cyclic-direction condition of
Proposition~\ref{prop:local-recognition}.

\begin{question}[Former open question: nonlocal multiplication rigidity]
\label{q:nonlocal-mult}
Can a multiplication-generated equilateral triangle be nonlocal?
Equivalently, does
\[
  T\in E_n,\qquad C_T=0
  \quad\Longrightarrow\quad
  T\in L_n
\]
hold for every \(n\)?  By Corollary~\ref{cor:local-mult-rigidity}, a positive
answer is equivalent to saying that every multiplication-generated
equilateral triangle is a nontrivial global \(\gamma\)-orbit.
\end{question}

\medskip
\noindent\textbf{Resolution of Question~\ref{q:nonlocal-mult}
(Ma--Zhang, 2026).}
Haobo Ma and Wenlin Zhang have proved that, for every \(n\ge1\),
\begin{equation}\label{eq:mazhang-rigidity}
  \mathcal M_n
   =\bigl\{\{b,\gamma(b),\gamma^2(b)\}:b\in\Delta_n\setminus\{e_n\}\bigr\}
   \subseteq L_n.
\end{equation}
In particular,
\[
 |\mathcal M_n|=\frac{4^n-1}{3},\qquad
 |\mathcal M_n\cap X_n|=0 \quad(n\ge1).
\]
This classification is \emph{their theorem} \cite[Theorem~1]{MaZhang2026},
not a theorem established by the preceding enumeration.  Their proof
extends the affine digit-stripping method of
Remark~8.3, Lemma~8.4, and Theorem~8.5.  The unsigned multiplication
constraint, combined with the oriented equilateral-completion equation
modulo \(2\), forces the last digits of the three vertices to be either
\((e,e,e)\) or a consistently cyclic ordering of \((i,j,k)\).
The same-orientation condition makes the three prefix signs equal,
so the last digits can be deleted while preserving the exact oriented
completion equation, including the all-\(e\) case.  Induction gives
the global cyclic classification.  Their paper also reports independent
exact distance-based checks through order \(8\).

\begin{remark}[OEIS A002450]
The sequence
\[
  0,1,5,21,85,341,\ldots
\]
with \(G_0=0\) is
\oeislink{https://oeis.org/A002450}{OEIS A002450} \cite{OEIS}.  Its
positive-index terms are the base-\(4\) repunits
\[
  G_n=11\cdots11_4.
\]
The same entry notes that these are generalized octagonal numbers, since
\[
  3G_n+1=4^n=(2^n)^2.
\]
For \(n\ge2\), however, \(G_n\equiv5\pmod8\), so only
\(G_1=1_8=i\) is itself a positive floretion word in the octal alphabet
\(\{1,2,4,7\}\).
\end{remark}

\section[Signed multiplication and pairwise-anticommuting equilateral triangles]{Signed multiplication and\texorpdfstring{\\}{ }pairwise-anticommuting equilateral triangles}\label{sec:signed}

The preceding sections use multiplication primarily through the unsigned map
$u(\pm b)=b$.  This is sufficient for product points and
multiplication-generation because it removes a global central sign.  Signed
commutation data already appears in Section~9 when the
multiplication-generated family is split into commuting and anticommuting
cases.  Here that sign information becomes a primary classifier on all of
$E_n$, rather than only a secondary subdivision of $\mathcal M_n$.

For positive basis words $b,c\in\Delta_n$, define the commutation sign
$\chi(b,c)\in\{\pm1\}$ by
\[
cb=\chi(b,c)\,bc.
\]
Thus $\chi(b,c)=+1$ for a commuting pair and $\chi(b,c)=-1$ for an
anticommuting pair.

\begin{definition}[Anticommuting positions]
For $b,c\in\Delta_n$, let $m(b,c)$ be the number of coordinate positions
at which the corresponding digits of $b$ and $c$ anticommute. Equivalently,
\[
m(b,c)=\#\{\,r:\ b_r,c_r\in\{i,j,k\},\ b_r\ne c_r\,\}.
\]
\end{definition}

\begin{lemma}[Commutation sign]
For all positive basis words $b,c\in\Delta_n$,
\[
bc=(-1)^{m(b,c)}cb,
\qquad\text{equivalently}\qquad
\chi(b,c)=(-1)^{m(b,c)}.
\]
\end{lemma}

\begin{proof}
At one coordinate, $e$ commutes with every digit and equal noncentral
digits commute.  Two distinct digits in $\{i,j,k\}$ anticommute.  The global
sign is therefore the product of one factor $-1$ for each coordinate counted
by $m(b,c)$.
\end{proof}

\begin{definition}[Pairwise-anticommuting class]
For $T=\{b_1,b_2,b_3\}\in E_n$, put
\[
\Theta(T)=b_1+b_2+b_3.
\]
Let
\[
\mathrm{AC}_n
=
\{\,T\in E_n:\ b_1b_2=-b_2b_1,\ b_1b_3=-b_3b_1,\ b_2b_3=-b_3b_2\,\}.
\]
We call these \emph{pairwise-anticommuting equilateral triangles}.  The next
proposition gives the equivalent scalar vertex-sum-square description.
\end{definition}

For comparison with the multiplication-generated family of Section~9,
Table~\ref{tab:mult} records the exact finite-order split after the present
definition has made the notation \(\mathrm{AC}_n\) available.

\begin{table}[H]
\centering
\caption{Exact enumeration of multiplication-generated \mbox{equilateral} triples
through order \(6\).  The fourth data column is the pairwise-anticommuting
subfamily.  The final column counts nonlocal multiplication-generated
triangles.  The all-order vanishing of that column, and the corresponding
all-order signed counts, follow from the rigidity theorem of
Ma and Zhang \cite{MaZhang2026}; the finite data provide independent checks.}
\label{tab:mult}
\begin{tabular}{crrrrr}
\toprule
\(n\)&\(|\mathcal M_n|\)&\(C_+\)&\(C_-\)&
\(|\mathcal M_n\cap\mathrm{AC}_n|\)&
\(|\mathcal M_n\cap X_n|\)\\
\midrule
1&1&0&0&1&0\\
2&5&3&0&2&0\\
3&21&9&0&12&0\\
4&85&45&0&40&0\\
5&341&165&0&176&0\\
6&1365&693&0&672&0\\
\bottomrule
\end{tabular}
\end{table}

\noindent\textit{OEIS cross-reference.}
Within the local multiplication-generated family,
Corollary~\ref{cor:local-mult-rigidity} identifies every triangle with a
nontrivial global \(\gamma\)-orbit
\(\{b,\gamma(b),\gamma^2(b)\}\).  Writing
\(\nu(b)=\#\supp(b)\), all three vertices have square
\((-1)^{\nu(b)}e_n\); hence the commutation filter of
Corollary~\ref{cor:mult-commutation-filter} shows that the orbit is
pairwise anticommuting exactly when \(\nu(b)\) is odd.  There are
\[
  \sum_{\nu\ {\rm odd}}\binom n\nu3^\nu
  =\frac{4^n-(-2)^n}{2}
\]
odd-support words, and every global \(\gamma\)-orbit has three words.
Therefore
\[
  |\mathcal M_n\cap L_n\cap\mathrm{AC}_n|
  =G_{\mathrm{anti}}(n)
  =\frac13\sum_{\nu\ {\rm odd}}\binom n\nu3^\nu
  =\frac{4^n-(-2)^n}{6}
  =2^{n-1}J_n,
\]
where \(J_n\) is the Jacobsthal sequence.  The values
\[
  0,1,2,12,40,176,672,\ldots
\]
form \oeislink{https://oeis.org/A003683}{OEIS A003683} \cite{OEIS}.

The Ma--Zhang rigidity theorem \cite{MaZhang2026} now shows that
this formula counts the \emph{entire} multiplication-generated
pairwise-anticommuting family, not only its local part:
\[
  |\mathcal M_n\cap\mathrm{AC}_n|=G_{\mathrm{anti}}(n)
  =\frac{4^n-(-2)^n}{6}\qquad(n\ge1).
\]
Likewise, the all-order commuting counts are
\(C_+(n)=(4^n+(-2)^n-2)/6\) and \(C_-(n)=0\).
The columns of Table~\ref{tab:mult} give finite checks of these exact results.

\begin{proposition}[Scalar vertex-sum square]\label{prop:scalar-square}
Let $b_1,b_2,b_3$ be three distinct positive basis words and
$T=\{b_1,b_2,b_3\}$.  Then
\[
\Theta(T)^2\in\mathbb{R}e_n
\quad\Longleftrightarrow\quad
b_1b_2=-b_2b_1,\quad
b_1b_3=-b_3b_1,\quad
b_2b_3=-b_3b_2.
\]
If in addition $T\in E_n$, then the same-orientation theorem gives a common
sign
\[
\varepsilon(T)=(-1)^{\nu(b_1)}
              =(-1)^{\nu(b_2)}
              =(-1)^{\nu(b_3)},
\]
and hence
\[
T\in\mathrm{AC}_n
\quad\Longleftrightarrow\quad
\Theta(T)^2=3\varepsilon(T)e_n.
\]
\end{proposition}

\begin{proof}
Expand
\[
\Theta(T)^2
=b_1^2+b_2^2+b_3^2
 +(b_1b_2+b_2b_1)
 +(b_1b_3+b_3b_1)
 +(b_2b_3+b_3b_2).
\]
The three self-squares are scalar multiples of $e_n$ by (2.1).  Each
anticommutator is either zero or twice a signed basis word.

It remains to exclude cancellation among distinct nonzero anticommutators.
The unsigned pair-product words
$u(b_1b_2)$, $u(b_1b_3)$, and $u(b_2b_3)$ are pairwise distinct.
For example, if $u(b_1b_2)=u(b_1b_3)$, then
$b_1b_2=\pm b_1b_3$; left multiplication by $b_1^{-1}$ gives
$b_2=\pm b_3$, and positivity of the \mbox{basis words} forces $b_2=b_3$,
contrary to distinctness.  The other pairs are identical.
Moreover, $u(b_i b_j)\ne e_n$ for $i\ne j$: coordinatewise, the
unsigned product is $e$ exactly when the two corresponding digits agree, so
$u(b_i b_j)=e_n$ would force $b_i=b_j$.  Thus the nonzero anticommutators
are supported on distinct nonidentity basis words and are linearly
independent both from one another and from the scalar line $\mathbb{R}e_n$.
Hence the square is scalar exactly when all three anticommutators vanish.

For $T\in E_n$, Theorem~\ref{thm:same-orientation} and \eqref{eq:square-orientation} imply
$b_1^2=b_2^2=b_3^2=\varepsilon(T)e_n$, giving the final formula.
\end{proof}

\begin{corollary}[Anticommutation on the main axes]\label{cor:axis-ac}
Let
\[
  T=\{b_I,b_J,b_K\}\in E_n^{\mathrm{ax}},
\]
where \(b_I\), \(b_J\), and \(b_K\) lie on the \(I\)-, \(J\)-, and
\(K\)-axes, respectively.  Write their supports as \(A,B,C\), so that
\(b_I\) has digit \(i\) on \(A\) and \(e\) elsewhere, and cyclically for
\(b_J,b_K\).  Then
\[
  T\in\mathrm{AC}_n
  \quad\Longleftrightarrow\quad
  |A\cap B|,\ |B\cap C|,\ |C\cap A|
  \text{ are all odd}.
\]
If these equivalent conditions hold, then
\[
  \Theta(T)^2=3\varepsilon(T)e_n,
  \qquad
  \varepsilon(T)=(-1)^{|A|}
                =(-1)^{|B|}
                =(-1)^{|C|}.
\]
\end{corollary}

\begin{proof}
At each position in \(A\cap B\), the corresponding digits of \(b_I\)
and \(b_J\) are \(i\) and \(j\), which anticommute. At every other
position, at least one digit is \(e\), so the digits commute. Hence
\[
  m(b_I,b_J)=|A\cap B|.
\]
Cyclically,
\[
  m(b_J,b_K)=|B\cap C|,
  \qquad
  m(b_K,b_I)=|C\cap A|.
\]
The commutation-sign lemma gives the parity criterion.  Since
\(T\in E_n\), the same-orientation theorem gives the common square sign
shown above, and Proposition~\ref{prop:scalar-square} gives the scalar-square
formula.
\end{proof}

\subsection*{The local parity criterion}

\begin{theorem}[Local pairwise-anticommutation criterion]\label{thm:local-ac}
Let
\[
T(b,S)=\{b,\gamma_S(b),\gamma_S^2(b)\}
\]
be a local triangle, with
$\varnothing\ne S\subseteq\supp(b)$.  Then
\[
T(b,S)\in\mathrm{AC}_n
\quad\Longleftrightarrow\quad
|S|\ \text{is odd}.
\]
\end{theorem}

\begin{proof}
Outside $S$, every pair of vertices has equal digits.  At each selected
coordinate, every pair has two distinct noncentral digits.  Hence
\[
m\bigl(\gamma_S^r(b),\gamma_S^s(b)\bigr)=|S|
\qquad(r\ne s).
\]
The result follows from the commutation-sign lemma.
\end{proof}

\begin{corollary}[Number of local pairwise-anticommuting triangles]\label{cor:local-ac-count}
For $n\ge1$,
\[
\begin{aligned}
|\mathrm{AC}_n\cap L_n|
&=
\sum_{\substack{1\le k\le n\\ k\ {\rm odd}}}
\binom nk\,3^{k-1}4^{n-k}  \\
&=\frac{7^n-1}{6}.
\end{aligned}
\]
The first values are
\[
1,\ 8,\ 57,\ 400,\ 2801,\ 19608,\ 137257,\ 960800,\ldots.
\]
\end{corollary}

\begin{proof}
For a fixed selected set $S$ of size $k$, Proposition~\ref{prop:fixed-S-count} gives
$3^{k-1}4^{n-k}$ local triangles.  The theorem retains exactly the odd
values of $k$.  Therefore
\[
|\mathrm{AC}_n\cap L_n|
=
\frac13\cdot\frac{(4+3)^n-(4-3)^n}{2}
=
\frac{7^n-1}{6}.
\]
\end{proof}

\subsection*{Even selected sets and a secondary global orbit}

For a local representative $(b,S)$, let $d_S(b)\in\Delta_n$ be the word
obtained by retaining the digits of $b$ on $S$ and replacing every
unselected digit by $e$.

\begin{proposition}[Square image of a local cycle]\label{prop:local-square-image}
Let $T=T(b,S)$ and put $\varepsilon=(-1)^{\nu(b)}$.  Then
\[
\Theta(T)^2=
\begin{cases}
3\varepsilon e_n,
& |S|\text{ odd},\\[1mm]
\varepsilon\!\left[
3e_n+
2\bigl(d_S(b)+\gamma(d_S(b))+\gamma^2(d_S(b))\bigr)
\right],
& |S|\text{ even}.
\end{cases}
\]
In the even case, the non-scalar part is supported on a nontrivial global
$\gamma$-orbit, hence on the vertices of another equilateral triangle.
\end{proposition}

\begin{proof}
Write
\[
x=b,\qquad y=\gamma_S(b),\qquad z=\gamma_S^2(b).
\]
If $|S|$ is odd, Theorem~\ref{thm:local-ac} shows that $x,y,z$ pairwise
anticommute.  Hence all three anticommutators vanish in
\[
(x+y+z)^2
=x^2+y^2+z^2+(xy+yx)+(yz+zy)+(zx+xz).
\]
Since $x^2=y^2=z^2=\varepsilon e_n$, this gives
\(\Theta(T)^2=3\varepsilon e_n\).

Now suppose $|S|$ is even.  Any two vertices agree outside $S$, while at
each coordinate in $S$ their digits are distinct members of
$\{i,j,k\}$.  Thus each pair has exactly
\(|S|\) anticommuting positions. By the commutation-sign lemma,
all three pairs therefore commute.
At an unselected noncentral coordinate the square of the common digit
contributes $-e$, while at a selected coordinate
$a\gamma(a)=\gamma^2(a)$, cyclically.  Since every selected coordinate is
noncentral,
\[
xy=\varepsilon\,\gamma^2(d_S(b)),\qquad
yz=\varepsilon\,d_S(b),\qquad
zx=\varepsilon\,\gamma(d_S(b)).
\]
Together with $x^2=y^2=z^2=\varepsilon e_n$, expansion of
$(x+y+z)^2$ gives the formula.  Because $S\ne\varnothing$, the global
$\gamma$-orbit of $d_S(b)$ is nontrivial.
\end{proof}

\medskip
\noindent\textit{Geometric interpretation.}
The even case is more rigid than a mere failure of scalarity: the entire
non-scalar support of $\Theta(T)^2$ is the global $\gamma$-orbit of
$d_S(b)$, hence is itself the vertex set of an equilateral triangle.  Thus
the local parity dichotomy sends odd $|S|$ to a scalar square, while even
$|S|$ produces a scalar term together with a uniformly weighted secondary
equilateral support.
\medskip

\begin{example}[The contrast between B.2 and B.3 from Appendix B]
For the local but non-product-centered triangle
\[
T=\{\sel{i}ij,\sel{j}ij,\sel{k}ij\}=T(iij,\{1\}),
\]
the selected-set size is odd, so
\[
(\sel{i}ij+\sel{j}ij+\sel{k}ij)^2=-3eee.
\]
By contrast,
\[
T'=\{j\sel{i}\sel{i},j\sel{j}\sel{j},j\sel{k}\sel{k}\}
   =T(jii,\{2,3\})
\]
is product-centered but has an even selected set.  Here
$d_{\{2,3\}}(jii)=eii$ and
\[
(j\sel{i}\sel{i}+j\sel{j}\sel{j}+j\sel{k}\sel{k})^2
=-3eee-2(eii+ejj+ekk).
\]
Thus product-centeredness and pairwise anticommutation are distinct
predicates.
\end{example}

\subsection*{Nonlocal examples and exact finite-order enumeration}

The local parity theorem does not exhaust the pairwise-anticommuting class.
The first nonlocal examples occur already at order $3$.

\begin{example}[A nonlocal pairwise-anticommuting triangle]
Let
\[
T=\{iji,\ jij,\ eek\}.
\]
Its scaled centroid coordinates are
\[
\Lambda_3(iji)=(5,2),\qquad
\Lambda_3(jij)=(2,5),\qquad
\Lambda_3(eek)=(-1,-1),
\]
which form an equilateral triangle.  The three commutation parities are
\[
m(iji,jij)=3,\qquad
m(iji,eek)=1,\qquad
m(jij,eek)=1,
\]
so $T\in\mathrm{AC}_3$.  It is nonlocal because the first coordinate has
digit set $\{i,j,e\}$, which is neither constant nor $\{i,j,k\}$ and hence
fails Proposition~\ref{prop:local-recognition}.
\end{example}

Exhaustive exact enumeration using the completion formulas
\eqref{eq:zplus}--\eqref{eq:kappa-lookup} and the commutation-parity test gives the following split.
No closed formula for the nonlocal column is asserted here.

\begin{table}[h]
\centering
\caption{Pairwise-anticommuting equilateral triangles for small orders.}
\begin{tabular}{r r r r}
\toprule
$n$ & $|\mathrm{AC}_n|$ & $|\mathrm{AC}_n\cap L_n|$
& $|\mathrm{AC}_n\cap X_n|$\\
\midrule
1 & 1       & 1       & 0\\
2 & 8       & 8       & 0\\
3 & 96      & 57      & 39\\
4 & 928     & 400     & 528\\
5 & 13\,472 & 2\,801  & 10\,671\\
6 & 183\,936& 19\,608 & 164\,328\\
\bottomrule
\end{tabular}
\end{table}

The rapid emergence of the nonlocal column shows that the signed
commutation structure cuts across locality rather than merely refining it.

\begin{remark}[Order-two recurrence contrast]\label{rem:ac-order2}
The exact local count
\[
  |\mathrm{AC}_n\cap L_n|=\frac{7^n-1}{6}
\]
satisfies the homogeneous second-order recurrence
\(a_n=8a_{n-1}-7a_{n-2}\) and therefore fits the order-two
recurrence-generator framework of Appendix~A.  By contrast, neither of the
computed sequences
\[
  |\mathrm{AC}_n|
  =1,8,96,928,13472,183936
\]
nor
\[
  |\mathrm{AC}_n\cap X_n|
  =0,0,39,528,10671,164328
\]
satisfies a homogeneous second-order linear recurrence with constant
coefficients across the displayed orders.  Thus these total and nonlocal
counts do not fit the present order-two recurrence scheme.  This observation
does not preclude other floretion constructions for the same sequences.
\end{remark}

\begin{remark}[Unsigned product data versus commutation-sign data]
Product points and multiplication-generation are largely governed by
unsigned products such as $u(b_1b_2b_3)$.  Pairwise anticommutation instead
depends on the central signs discarded by $u$.  Thus the two viewpoints
extract complementary information from the same multiplication law:
unsigned product words locate algebraic products geometrically, whereas
commutation signs record how the signed central extension behaves on pairs.
\end{remark}

\section{Radial incidence and circular shells}\label{sec:radial-incidence}

The question of whether the triangle-incidence function
\[
 \eta_n(b)=|\{T\in E_n:b\in T\}|
\]
depends only on \(q_{120}(\Lambda_n(b))\) arose from exact enumeration
through order six.  We answer that question affirmatively for all orders.
The argument counts all equilateral centroid triangles, without a
locality restriction.

For the finite triangular patch we use barycentric triples
\((a,b,c)\) and their derived \(60^\circ\) grid coordinates \((a,b)\),
as introduced in Section~1.  The auxiliary \((r,s)\) variables in the
counting diagram for Lemma~11.1 are instead plotted on perpendicular
\emph{drawing} axes: they are integer indices, not a third geometric
coordinate system or distance metric.

Write \(N=2^n\).  The side parameter
\(M\) of an \((M+1)\)-row triangular lattice patch counts its unit
intervals along a side.

\begin{lemma}[Vertex incidence in a triangular grid]\label{lem:grid-incidence}
Let
\[
  \Omega_M=\{(a,b,c)\in\mathbb Z_{\geq0}^3:a+b+c=M\}
\]
be the triangular lattice patch of side parameter \(M\).  For a point
\((a,b,c)\), the number of distinct nondegenerate equilateral triangles in
\(\Omega_M\) with that point as a vertex is
\begin{equation}\label{eq:grid-incidence}
  D_M(a,b,c)=M+ab+bc+ca.
\end{equation}
\end{lemma}

\begin{proof}
We use the \(60^\circ\) lattice convention stated in the Introduction.
Since \(a+b+c=M\), a barycentric point \((a,b,c)\) is represented
by \((a,b)\), with \(c=M-a-b\).  Thus the triangular patch is represented by
\[
  \{(x,y)\in\mathbb Z^2:x\geq0,\ y\geq0,\ x+y\leq M\}.
\]
Fix \(A=(a,b,c)\in\Omega_M\).  We count equilateral triangles containing
\(A\) by choosing a second vertex \(B\) and constructing the third
vertex \(C\) by counterclockwise rotation through \(60^\circ\) about
\(A\).

\smallskip
\noindent\emph{First additional vertex.}
Let \((u,v)\in\mathbb Z^2\) be the displacement from \(A\) to \(B\).
In the \(60^\circ\) lattice basis, \(B\) has coordinates \((a+u,b+v)\)
and hence barycentric coordinates
\[
  B=(a+u,\ b+v,\ c-u-v).
\]
For \(B\) to belong to \(\Omega_M\), all three coordinates must be
nonnegative:
\[
\begin{aligned}
  a+u\geq0 &\iff u\geq-a,\\
  b+v\geq0 &\iff v\geq-b,\\
  c-u-v\geq0 &\iff u+v\leq c.
\end{aligned}
\]

\smallskip
\noindent\emph{Second additional vertex.}
In this lattice basis, counterclockwise rotation by \(60^\circ\) acts
on displacements by
\[
  R_{60^\circ}(u,v)=(-v,u+v).
\]
The third vertex consequently has lattice coordinates
\((a-v,b+u+v)\).  Its third barycentric coordinate is
\(M-(a-v)-(b+u+v)=c-u\), so
\[
  C=(a-v,\ b+u+v,\ c-u).
\]
Again, all three coordinates must be nonnegative:
\[
\begin{aligned}
  a-v\geq0 &\iff v\leq a,\\
  b+u+v\geq0 &\iff u+v\geq-b,\\
  c-u\geq0 &\iff u\leq c.
\end{aligned}
\]

\smallskip
\noindent\emph{Combining the conditions.}
Both additional vertices belong to the patch precisely when
\[
  \boxed{-a\leq u\leq c,\qquad
  -b\leq v\leq a,\qquad
  -b\leq u+v\leq c.}
\]
The first two inequalities determine a rectangle containing
\((a+c+1)(a+b+1)\) integer points.  The third inequality removes two
opposite triangular corners.  To count the corner where \(u+v>c\),
write
\[
  u=c-i,\qquad v=a-j,\qquad i,j\geq0.
\]
Then \(u+v>c\) is equivalent to \(i+j<a\).  The excluded integer
pairs number
\[
  1+2+\cdots+a=\frac{a(a+1)}{2}.
\]
For the opposite corner, where \(u+v<-b\), write
\[
  u=-a+i,\qquad v=-b+j,\qquad i,j\geq0.
\]
This again gives \(i+j<a\), so the second corner has the same count.

For visualization, set \(r=v+b\) and \(s=u+a\), and plot the integer
pairs \((r,s)\) on \emph{orthogonal} axes.  The admissible set becomes
\[
  0\leq r\leq a+b,\qquad
  0\leq s\leq a+c,\qquad
  a\leq r+s\leq M.
\]
Figure~\ref{fig:lemma11-counting} shows this rectangular index diagram.
Its perpendicular axes are not the geometric axes of the \(60^\circ\)
grid; no Euclidean distance is computed from this plot.

\begin{figure}[H]
\centering
\includegraphics[width=.68\linewidth]{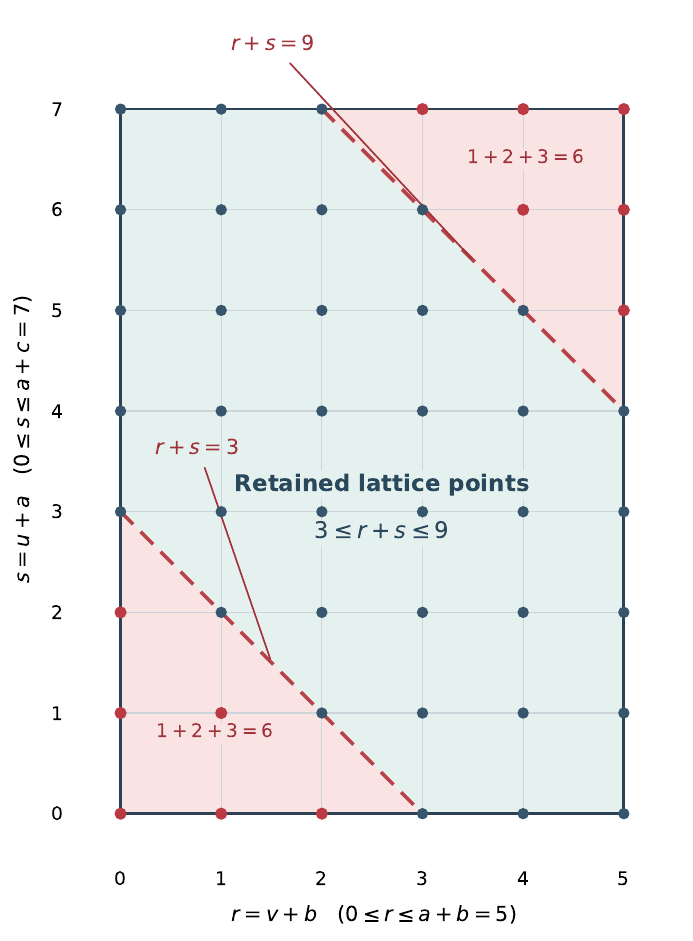}
\caption{Counting region for Lemma~\ref{lem:grid-incidence}, shown
using auxiliary indices \(r,s\) on perpendicular drawing axes, not
the geometric \(60^\circ\) grid basis. For \(a=3\), \(b=2\), \(c=4\)
(\(M=9\)), the two shaded opposite corners each contain six excluded
integer pairs; points on the dashed boundary lines are retained.}
\label{fig:lemma11-counting}
\end{figure}

Finally, exclude the zero displacement \((u,v)=(0,0)\), which
produces a degenerate triangle.  The number of admissible nonzero
displacements is therefore
\begin{align*}
  D_M(a,b,c)
    &=(a+c+1)(a+b+1)-a(a+1)-1\\
    &=a+b+c+ab+bc+ca\\
    &=M+ab+bc+ca.
\end{align*}
Each nondegenerate equilateral triangle containing \(A\) is counted
exactly once: of its two other vertices, exactly one can serve as
\(B\) so that the other is obtained by counterclockwise rotation
through \(60^\circ\) about \(A\).
\end{proof}

\begin{theorem}[Exact radial-incidence law]\label{thm:radial-incidence}
For every \(n\geq1\) and \(b\in\Delta_n\), put
\(N=2^n\) and \(R=q_{120}(\Lambda_n(b))\).  Then
\begin{equation}\label{eq:radial-incidence}
  \eta_n(b)=
  \begin{cases}
  \dfrac{(N-1)(N+2)-R}{3},
    &\text{if \(b\) labels an upward tile},\\[6pt]
  \dfrac{(N-2)(N+1)-R}{3},
    &\text{if \(b\) labels a downward tile}.
  \end{cases}
\end{equation}
In particular, for any \(b,c\in\Delta_n\),
\[
  q_{120}(\Lambda_n(b))=q_{120}(\Lambda_n(c))
  \quad\Longrightarrow\quad \eta_n(b)=\eta_n(c).
\]
Thus the radial function \(\Phi_n\) exists in every order and is given
explicitly by \eqref{eq:radial-incidence} on the realized radii.
\end{theorem}

\begin{proof}
We first express Lemma~\ref{lem:grid-incidence} radially, without
assuming that \(M\) is divisible by \(3\).  The geometric center of
\(\Omega_M\) is \(o=(M/3,M/3,M/3)\), which need not be a lattice point.
For \(x=(a,b,c)\in\Omega_M\), we measure squared distance
from \(o\) in units of nearest-neighbor grid spacing, using
\(q_{60}\) in the barycentric coordinate pair.  Put
\(\alpha=a-M/3\), \(\beta=b-M/3\), and \(\delta=c-M/3\).
Since \(\alpha+\beta+\delta=0\), the barycentric metric gives
\[
 \begin{aligned}
 q_{60}(\alpha,\beta)
  &=\alpha^2+\alpha\beta+\beta^2\\
  &=\frac{\alpha^2+\beta^2+\delta^2}{2}
   =\frac{M^2}{3}-(ab+bc+ca).
 \end{aligned}
\]
Consequently, \eqref{eq:grid-incidence} becomes
\[
 D_M(a,b,c)=\frac{M(M+3)}3-q_{60}(\alpha,\beta).
\]
For the point \((3,2,4)\) illustrated in
Figure~\ref{fig:lemma11-counting}, we have \(M=9\) and
\(q_{60}(0,-1)=1\).  The direct and radial counts agree:
\[
 D_9(3,2,4)=9+6+8+12=35=\frac{9(9+3)}3-1.
\]
Thus incidence is constant on circles about a patch's geometric center
and decreases linearly with squared radius.

To apply this to the two orientation patches, let \(h_n\) be the side
length of an order-\(n\) tile, equal to the nearest-neighbor spacing of
either centroid grid.  At order one, \(P(i)=d_1u\) and
\(P(j)=d_1w\), with \(\|u\|=\|w\|=1\) and
\(u\cdot w=\cos120^\circ=-1/2\).  Their separation is \(h_1\), so
\[
 h_1^2=\|P(i)-P(j)\|^2
  =d_1^2(\|u\|^2+\|w\|^2-2u\cdot w)=3d_1^2.
\]
Each midpoint subdivision halves lengths, hence
\(h_n^2=3d_1^2/4^{n-1}\).  Writing \(\Lambda_n(b)=(X,Y)\),
Lemma~2.4 and the \(120^\circ\) metric (2.5) give
\[
 \|P(b)\|^2=\frac{d_1^2}{4^{n-1}}
 (X^2-XY+Y^2)=\frac{d_1^2}{4^{n-1}}q_{120}(\Lambda_n(b)).
\]
Dividing by \(h_n^2\) yields
\begin{equation}\label{eq:radius-scaling}
 \frac{\|P(b)\|^2}{h_n^2}=\frac{q_{120}(\Lambda_n(b))}{3}.
\end{equation}
Both centroid grids have their geometric center at the origin:
digitwise \(\gamma\)-equivariance preserves each orientation patch
and rotates it by \(120^\circ\) about \(P(e_n)=0\).
Thus for the barycentric point corresponding to \(b\), the two
squared-radius forms are related by the explicit identity
\begin{equation}\label{eq:centroid-grid-radius-bridge}
 q_{120}(\Lambda_n(b))=3q_{60}(\alpha,\beta).
\end{equation}
The factor of three reflects the different lattice scales, not a change in
the underlying Euclidean radius.  Since \(R=q_{120}(\Lambda_n(b))\), the
radial grid formula becomes
\begin{equation}\label{eq:grid-radial-incidence}
 D_M(b)=\frac{M(M+3)-R}3.
\end{equation}

The upward centroid patch has \(N\) rows and hence side parameter
\(M=N-1\); the downward patch has \(N-1\) rows and side parameter
\(M=N-2\), by (2.2).  Theorem~3.5 ensures that each equilateral
triangle through a centroid lies entirely in its orientation patch.
Substitution yields \eqref{eq:radial-incidence}.

Theorem~3.5 separates triangles by orientation but does not itself
exclude centroids of different orientations lying at the same radius.
To establish radial homogeneity across both patches, we also rule out
a shared squared radius.  The quadratic form satisfies
\[
 q_{120}(x,y)=(x+y)^2-3xy\equiv(x+y)^2\pmod3.
\]
By Lemma~2.6, upward centroids have \(x+y\equiv N-1\pmod3\),
and downward centroids have \(x+y\equiv N-2\pmod3\).
Since \(N=2^n\equiv\pm1\pmod3\), exactly one of
\(N-1,N-2\) is divisible by \(3\).  Hence the squared radius
\(q_{120}(\Lambda_n(b))\) is \(0\pmod3\) on one patch and
\(1\pmod3\) on the other.  No squared radius is
shared, and the desired implication follows.
\end{proof}

\noindent\textit{An alternative argument.}
Suppose an upward and a downward centroid shared the same scaled
squared radius \(R\).  Since both incidence counts are integers, their
difference would also be an integer.  But \eqref{eq:radial-incidence}
gives
\[
  \eta_n(\text{up})-\eta_n(\text{down})
  =\frac{(N-1)(N+2)-(N-2)(N+1)}3
  =\frac{2N}{3},
\]
which is not an integer for \(N=2^n\).  This independently excludes a
shared radius, without using residues modulo \(3\).

For completeness, the function can be read directly from the squared
radius, without requiring a separate orientation test.  An attained
value \(R=q_{120}(\Lambda_n(b))\) always satisfies \(R\equiv0\) or \(1\pmod3\),
since \(q(x,y)\equiv(x+y)^2\pmod3\).  Because \(N=2^n\), the
upward case is characterized by \(R\bmod3=n\bmod2\) as an element
of \(\{0,1\}\): when \(n\) is even, upward tiles have \(R\equiv0\pmod3\),
and when \(n\) is odd they have \(R\equiv1\pmod3\).  More explicitly,
for realized \(R\) set
\[
  \Phi_n(R)=
  \begin{cases}
  \bigl((N-1)(N+2)-R\bigr)/3,
      &R\bmod3=n\bmod2,\\
  \bigl((N-2)(N+1)-R\bigr)/3,
      &R\bmod3\ne n\bmod2.
  \end{cases}
\]
Then \(\eta_n(b)=\Phi_n(q_{120}(\Lambda_n(b)))\) for every \(b\in\Delta_n\).

\begin{corollary}[Algebraic form of radial incidence]\label{cor:radial-algebraic}
For a word \(b\), put
\(\tau(b)=(-1)^{\#\{r:b_r=e\}}\), and write
\(b^2=\varepsilon_b e_n\) with \(\varepsilon_b\in\{-1,1\}\).
Then Theorem~\ref{thm:radial-incidence} is equivalently
\[
 \boxed{\begin{aligned}
 3\eta_n(b)+q_{120}(\Lambda_n(b))&=4^n-2+2^n\tau(b)\\
                   &=4^n-2+(-2)^n\varepsilon_b.
 \end{aligned}}
\]
\end{corollary}
\begin{proof}
The two branches of \eqref{eq:radial-incidence} have constants
\(N^2+N-2\) and \(N^2-N-2\), respectively.  The sign
\(\tau(b)\) is positive precisely for upward tiles.  Moreover,
\(\varepsilon_b=(-1)^{\nu(b)}\) by (2.1), so
\(\tau(b)=(-1)^n\varepsilon_b\).  Substituting \(N=2^n\)
proves both equalities.
\end{proof}

\begin{corollary}[Local, nonlocal, and local anticommuting vertex incidence]\label{cor:local-nonlocal-incidence}
Define \(\eta_n^{\mathrm{local}}(b)\) and
\(\eta_n^{\mathrm{nonlocal}}(b)\) to count the triangles of
\(L_n\) and \(X_n=E_n\setminus L_n\), respectively, that contain \(b\).
For every \(b\in\Delta_n\),
\[
 \boxed{\begin{aligned}
 \eta_n^{\mathrm{local}}(b)&=2^{\nu(b)}-1,\\
 \eta_n^{\mathrm{nonlocal}}(b)
   &=\frac{4^n+2^n\tau(b)-2-q_{120}(\Lambda_n(b))}{3}
     -2^{\nu(b)}+1.
 \end{aligned}}
\]
If \(\eta_n^{\mathrm{local,AC}}(b)\) counts local pairwise-anticommuting
triangles through \(b\), then
\[
 \boxed{\eta_n^{\mathrm{local,AC}}(b)=
 \begin{cases}
 2^{\nu(b)-1}, &\nu(b)>0,\\
 0, &\nu(b)=0.
 \end{cases}}
\]
\end{corollary}
\begin{proof}
Each nonempty selected set \(S\subseteq\supp(b)\)
produces one local triangle \(T(b,S)\) containing \(b\).
Conversely, a local triangle containing \(b\) has a unique
selected-position set \(S\subseteq\supp(b)\), by
Proposition~\ref{prop:local-recognition}, and equals \(T(b,S)\)
regardless of which of its three vertices was used as a representative.
Thus there are \(2^{\nu(b)}-1\) local triangles through \(b\).
Subtracting from the algebraic form of the total incidence yields
the nonlocal formula.  A local triangle is pairwise anticommuting
exactly when \(|S|\) is odd, by Theorem~\ref{thm:local-ac}.
For nonempty support, precisely \(2^{\nu(b)-1}\) subsets have
odd cardinality; for empty support there are none.
\end{proof}
Summing the final formula over all \(b\in\Delta_n\) and
dividing by three recovers
\(|\mathrm{AC}_n\cap L_n|=(7^n-1)/6\) from
Corollary~\ref{cor:local-ac-count}.

For example, at order four the words \(eiej\) and \(ijkk\) have the
same squared centroid radius \(21\) and total incidence \(83\),
but their local incidences are respectively \(3\) and \(15\).
Thus their nonlocal incidences are \(80\) and \(68\).

\begin{corollary}[Recovery of the central-vertex count]
For the central word \(e_n\), Theorem~\ref{thm:radial-incidence} yields
\[
  \eta_n(e_n)=\frac{4^n+(-2)^n-2}{3},
\]
recovering Proposition~\ref{prop:central-vertex}.
\end{corollary}

\begin{proof}
Here \(\Lambda_n(e_n)=0\), so its squared centroid radius is zero.
The central tile is upward for even \(n\)
and downward for odd \(n\) by Definition~2.2.
The appropriate branch of \eqref{eq:radial-incidence} gives the
stated count. Equivalently, the orientation grid containing \(e_n\)
has side parameter \(M=N-1\) for even \(n\) and \(M=N-2\) for odd
\(n\).  In either case \(3\mid M\) and \(e_n\) is the central
lattice point. Lemma~\ref{lem:grid-incidence} then gives
\(\eta_n(e_n)=3m(m+1)=6\sum_{j=1}^m j\) with \(M=3m\),
recovering also the hexagonal-shell count in the proof of
Proposition~\ref{prop:central-vertex}.
\end{proof}

\begin{remark}[Limit of radial homogeneity]
The theorem concerns the entire class \(E_n\).  It does not imply
radial homogeneity of the restricted classes \(L_n\), \(X_n\), or
\(\mathrm{AC}_n\); the counterexamples discussed in the final
directions for further work remain unaffected.
\end{remark}

\medskip
\noindent\textbf{Circular shells and triangle incidence.}
For each attained \(R\geq0\), define the complete shell
\[
  K_n(R)=\{b\in\Delta_n:q_{120}(\Lambda_n(b))=R\}.
\]
For any \(K\subseteq\Delta_n\), put
\[
  N_j(K)=|\{T\in E_n:|T\cap K|=j\}|,\qquad j=0,1,2,3,
\]
and define the pair-incidence sum
\begin{equation}\label{eq:shell-pair-sum}
  B(K)=\sum_{\{b,c\}\subseteq K}\kappa_n(b,c),
\end{equation}
where \(\kappa_n\) is Definition~3.1.  This \(B(K)\) is distinct
from the all-pair counts \(B_j(n)\) in Corollary~4.2.

\begin{proposition}[Incidence decomposition]\label{prop:shell-decomposition}
For any subset \(K\subseteq\Delta_n\),
\begin{align}
 \sum_{b\in K}\eta_n(b)
   &=N_1(K)+2N_2(K)+3N_3(K),\label{eq:shell-inc-count}\\
 B(K)&=N_2(K)+3N_3(K).\label{eq:shell-pair-count}
\end{align}
Consequently
\begin{align}
 N_2(K)&=B(K)-3N_3(K),\label{eq:shell-N2}\\
 N_1(K)&=\sum_{b\in K}\eta_n(b)-2B(K)+3N_3(K).
     \label{eq:shell-N1}
\end{align}
\end{proposition}

\begin{proof}
Count incidences \((b,T)\) with \(b\in K\cap T\) to get
\eqref{eq:shell-inc-count}.  Count incidences \((\{b,c\},T)\)
with \(\{b,c\}\subseteq K\cap T\) to get
\eqref{eq:shell-pair-count}: triangles with two vertices in
\(K\) contribute once and those with three contribute three
times.  Eliminating \(N_2\) gives the remaining formulas.
\end{proof}

\begin{proposition}[Triangles lying on a circular shell]\label{prop:shell-triples}
Let \(R>0\) be an attained squared radius, and put
\(K=K_n(R)\) and \(k=|K|\).  Then
\begin{equation}\label{eq:shell-N3}
  N_3(K)=\frac{k}{3}.
\end{equation}
These triangles are exactly the global \(\gamma\)-orbits
\(\{b,\gamma(b),\gamma^2(b)\}\) contained in \(K\).
Writing \(D=\Phi_n(R)\) and \(B=B(K)\),
\begin{equation}\label{eq:shell-three-counts}
  N_1=k(D+1)-2B,\qquad N_2=B-k,\qquad N_3=k/3.
\end{equation}
In particular, the number of triangles meeting the shell is
\(kD-B+k/3\).
\end{proposition}

\begin{proof}
The digitwise action \(\gamma\) rotates centroid points by
\(120^\circ\) about the origin and acts freely on the noncentral
words, so \(K\) is partitioned into three-element global
\(\gamma\)-orbits.  Each orbit is an equilateral triangle.
Conversely, if all three vertices of an equilateral triangle
lie on a circle centered at the origin, the circle is its
circumcircle; uniqueness of the circumcenter forces the three
vertices to be separated by \(120^\circ\).  Hence they form one
such orbit.  Thus \(N_3=k/3\), and
Proposition~\ref{prop:shell-decomposition} and
Theorem~\ref{thm:radial-incidence} yield
\eqref{eq:shell-three-counts}.
\end{proof}

\begin{corollary}[Product point and Euclidean center]\label{cor:zero-centers}
For every nondegenerate equilateral centroid triangle \(T\in E_n\),
\[
 \boxed{\begin{gathered}
 C_T=0\quad\Longleftrightarrow\quad Q_T=0\\
 \Longleftrightarrow\quad
 T=\{b,\gamma(b),\gamma^2(b)\}
 \quad\text{for some }b\in\Delta_n\setminus\{e_n\}.
 \end{gathered}}
\]
\end{corollary}
\begin{proof}
The multiplication/product-point characterization of
Proposition~\ref{prop:mult-product-point} gives
\(C_T=0\Longleftrightarrow T\in\mathcal M_n\).
By Ma and Zhang's rigidity theorem \cite{MaZhang2026}, these are
precisely the nontrivial global \(\gamma\)-orbits.
Every such orbit has Euclidean center \(0\), since
\(P(\gamma b)\) and \(P(\gamma^2b)\) are rotations of
\(P(b)\) by \(120^\circ\) and \(240^\circ\) about the origin.
Conversely, if \(Q_T=0\), then \(0\) is the circumcenter of the
nondegenerate equilateral triangle.  Its three vertices are
therefore related by the same \(120^\circ\) rotations, and
centroid injectivity and \(\gamma\)-equivariance identify their
words as one global orbit.
\end{proof}

\begin{remark}[Exact finite pair procedure]
The pair sum \(B(K)\) in \eqref{eq:shell-pair-sum}
is computable without enumerating triples.  For each unordered
pair \(\{b,c\}\subseteq K\), compute \(Z_+(b,c)\) and
\(Z_-(b,c)\) from (3.2)--(3.3) and count how many belong to
\(\Lambda_n(\Delta_n)\), as prescribed by (3.4).
The resulting sum, together with
\eqref{eq:shell-three-counts}, supplies exact values of
\(N_1,N_2,N_3\) for any complete shell.
\end{remark}

\medskip
\noindent\textbf{The midpoint, or incircle, shell.}
There is a geometrically distinguished shell meeting each of
the three sides of the downward-centroid triangular grid at
exactly one lattice point.  Its three boundary points are the
side midpoints.  Put \(M=N-2\), which is even.  In barycentric
coordinates the midpoint points are the permutations of
\((0,M/2,M/2)\).  We use the term \emph{midpoint shell}
for the entire radial shell through these points, including
any interior centroids at the same radius.

\begin{corollary}[Incidence on the midpoint shell]\label{cor:midpoint-shell}
For \(n\geq2\) define
\[
  K_n^{\mathrm{mid}}
   =K_n\!\left((2^{n-1}-1)^2\right).
\]
It contains exactly one centroid on each side of the
downward-centroid grid, and every \(b\in K_n^{\mathrm{mid}}\)
satisfies
\begin{equation}\label{eq:midpoint-incidence}
  \eta_n(b)=4^{n-1}-1.
\end{equation}
If \(k_n=|K_n^{\mathrm{mid}}|\) and
\(B_n^{\mathrm{mid}}=B(K_n^{\mathrm{mid}})\), the numbers
of triangles meeting the midpoint shell in precisely one,
two, or three vertices are
\[
  N_1=k_n4^{n-1}-2B_n^{\mathrm{mid}},\qquad
  N_2=B_n^{\mathrm{mid}}-k_n,\qquad N_3=k_n/3.
\]
\end{corollary}

\begin{proof}
At the midpoint \((0,M/2,M/2)\),
\[
 ab+bc+ca=\frac{M^2}{4},\qquad
 \frac{\|P(b)\|^2}{h_n^2}=\frac{M^2}{12}.
\]
By \eqref{eq:radius-scaling}, its scaled squared radius is
\(q_{120}(\Lambda_n(b))=M^2/4=(2^{n-1}-1)^2\).  The congruence separation
in Theorem~\ref{thm:radial-incidence} excludes upward
centroids from this shell, so all its points lie in the
downward grid.

On a side of that grid, say \(a=0\), one has \(b+c=M\) and
\[
  \frac{\|P\|^2}{h_n^2}=\frac{M^2}{3}-bc.
\]
This is minimized uniquely at \(b=c=M/2\).  Thus exactly
one boundary point on each side belongs to the shell.
Lemma~\ref{lem:grid-incidence} now gives
\[
  \eta_n(b)=M+\frac{M^2}{4}
    =\frac{(M+2)^2}{4}-1=4^{n-1}-1
\]
at every point of the shell, by
Theorem~\ref{thm:radial-incidence}.
Apply Proposition~\ref{prop:shell-triples} for the
three-count formulas.
\end{proof}

For reference, exact pair enumeration using Proposition~3.2 gives
the following small-order midpoint-shell data.  The columns
\(N_j\) count triangles with exactly \(j\) vertices on the shell.
\begin{center}
\begin{tabular}{rrrrrrrr}
\toprule
\(n\)&\(Q\)&\(k_n\)&\(\eta_n\)&\(B_n^{\mathrm{mid}}\)&\(N_1\)&\(N_2\)&\(N_3\)\\
\midrule
2&1&3&3&6&0&3&1\\
3&9&6&15&15&66&9&2\\
4&49&9&63&36&504&27&3\\
5&225&6&255&15&1506&9&2\\
6&961&9&1023&36&9144&27&3\\
\bottomrule
\end{tabular}
\end{center}

\begin{figure}[H]
\centering
\includegraphics[width=.75\linewidth]{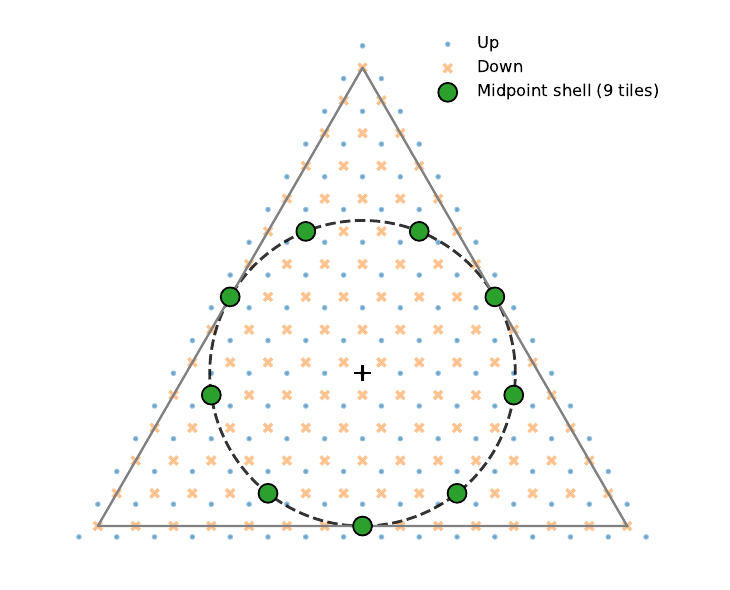}
\caption{The complete midpoint shell at order \(n=4\),
shown as a dashed circle through the three side midpoints of
the downward-centroid grid. The nine highlighted centroids
lie at the same squared radius and each belongs to
\(4^3-1=63\) equilateral triangles. The finite triangular
patch does not possess continuous rotational symmetry.}
\end{figure}

\begin{remark}[Midpoint shell versus most populated shell]
The midpoint shell is distinguished geometrically: its
circle is the incircle of the downward centroid grid.  It is
not, however, generally a shell of maximum cardinality,
and a maximally populated shell need not be unique.  For
example, in order \(4\) the midpoint shell has \(9\) points,
while four other shells each have \(12\); in order \(5\) it
has \(6\) points, while two other shells each have \(18\).
Thus ``unique largest ring'' is appropriate in the sense
of the incircle, not in the sense of the greatest number
of centroids.
\end{remark}

\begin{remark}[Scope]
Theorem~\ref{thm:radial-incidence} resolves the previously
computational radial-incidence question.  The new ring
formulas reduce finer incidence counts to the finite pair
sum \(B(K)\).  They do not assert a closed formula for the
number of centroids on each shell, nor radial homogeneity
for local, nonlocal, or pairwise-anticommuting subclasses.
\end{remark}

\newpage
\begin{remark}[Different radial behavior of triangle subclasses]\label{rem:radial-subclasses-example}
The two order-four words \(eiej\) and \(ijkk\) lie on the same
squared-radius shell:
\[
 q_{120}(\Lambda_4(eiej))=q_{120}(\Lambda_4(ijkk))=21.
\]
Their incidences illustrate an important distinction among the triangle
classes:
\begin{center}
\small
\begin{tabular}{lrr}
\toprule
Number of triangles through the indicated word
  & \(eiej\) & \(ijkk\) \\
\midrule
All equilateral triangles & 83 & 83 \\
Local triangles & 3 & 15 \\
Product-centered triangles (\(C_T=Q_T\)) & 4 & 13 \\
Triangles centered at the origin (\(Q_T=0\)) & 1 & 1 \\
\bottomrule
\end{tabular}
\end{center}
The total and local counts follow from Theorem~\ref{thm:radial-incidence}
and Corollary~\ref{cor:local-nonlocal-incidence}; in particular, the
corresponding nonlocal counts are \(80\) and \(68\).  The product-centered
entries are exact finite counts obtained by enumerating the order-four
completions of Proposition~\ref{prop:two-completions} and testing
\(3C_T=P(b_1)+P(b_2)+P(b_3)\) for each triangle
\(T=\{b_1,b_2,b_3\}\).  In contrast, the last row holds for \emph{every}
noncentral word, not merely this pair: by
Proposition~\ref{prop:shell-triples}, each such word belongs to exactly
one origin-centered equilateral triangle, its global \(\gamma\)-orbit.
Thus total and origin-centered incidence are radially homogeneous,
whereas local and product-centered incidence need not be.
\end{remark}

\clearpage
\section{Alternative enumerations and consequences of radial incidence}
\label{sec:alternative-enumerations}

\begin{proposition}[Squared-radius moment identity]\label{prop:radius-moment}
For every \(n\geq1\), with \(q(x,y)=x^2-xy+y^2\) and
\(\Lambda_n:\Delta_n\to\mathbb Z^2\),
\begin{equation}\label{eq:radius-moment}
 \sum_{b\in\Delta_n}q(\Lambda_n(b))
   =\frac{4^n(4^n-1)}{4}=3|E_n|.
\end{equation}
\end{proposition}

\begin{proof}
Write \(b^2=\varepsilon_b e_n\), where \(\varepsilon_b=\pm1\).
Since \(i^2=j^2=k^2=-e\) and \(e^2=e\), summing over the
independent digit positions gives
\[
 \sum_{b\in\Delta_n}b^2
  =(i^2+j^2+k^2+e^2)^{\otimes n}
  =(-2)^n e_n,
 \qquad
 \sum_{b\in\Delta_n}\varepsilon_b=(-2)^n.
\]
Summing Corollary~\ref{cor:radial-algebraic} over all \(4^n\)
basis words therefore yields
\begin{align*}
 3\sum_{b\in\Delta_n}\eta_n(b)
 +\sum_{b\in\Delta_n}q(\Lambda_n(b))
 &=4^n(4^n-2)+(-2)^n\sum_{b\in\Delta_n}\varepsilon_b\\
 &=4^n(4^n-2)+4^n
  =4^n(4^n-1).
\end{align*}
Every equilateral triangle contributes exactly three vertex
incidences, so \(\sum_b \eta_n(b)=3|E_n|\).  Using the already
established total \(|E_n|=4^n(4^n-1)/12\), we conclude
\[
 \sum_{b\in\Delta_n}q(\Lambda_n(b))
   =4^n(4^n-1)-9|E_n|
   =\frac{4^n(4^n-1)}4=3|E_n|.
\]
\end{proof}

The moment identity states that the sum of squared
scaled centroid radii equals the total number of
vertex--triangle incidences:
\[
 \sum_{b\in\Delta_n}q(\Lambda_n(b))
      =\sum_{b\in\Delta_n}\eta_n(b).
\]
This is an equality of sums, not a pointwise equality
\(q(\Lambda_n(b))=\eta_n(b)\).

\medskip
\noindent\textbf{Why zero and two completions balance.}
Put \(N=2^n\).  The numbers of vertices in the two orientation classes are
\[
 U=\frac{N(N+1)}{2},\qquad
 V=\frac{N(N-1)}{2}.
\]
Hence the number \(S\) of unordered same-orientation pairs is
\[
 S=\binom U2+\binom V2
   =\frac{N^2(N^2-1)}4=3|E_n|.
\]
Let \(A_j(n)\), as in Theorem~4.1, count same-orientation
pairs with exactly \(j\) possible equilateral completions.
Counting all such pairs and then counting triangle sides yields
\[
 A_0+A_1+A_2=S,\qquad
 A_1+2A_2=3|E_n|=S.
\]
Subtracting gives
\begin{equation}\label{eq:zero-two-balance}
 A_0(n)=A_2(n).
\end{equation}
Thus the average number of completions per same-orientation
pair is exactly one.  This short argument recovers the equality
in (4.3), but not the individual formulas for \(A_0,A_1,A_2\).

\medskip
\noindent\textbf{Global cyclic and central-vertex counts revisited.}
By Proposition~\ref{prop:shell-triples}, the equilateral
triangles lying wholly on a circle centered at the origin
are precisely the nontrivial global \(\gamma\)-orbits.
Summing \(|K_n(R)|/3\) over nonzero shells gives
\((4^n-1)/3\), recovering the count in Proposition~9.5
geometrically, but not its product identity.

Likewise the central count of Proposition~4.6 follows
by setting \(q(\Lambda_n(e_n))=0\) in
Theorem~\ref{thm:radial-incidence}; the overlapping
hexagonal-patch proof in Section~4 provides the finer
geometry.  A related identity is worth recording.
The commuting global \(\gamma\)-orbits of Section~9
number
\[
 G_+(n)=\frac{4^n+(-2)^n-2}{6},
\]
because the even-\(\nu\) words (including \(e_n\))
number \((4^n+(-2)^n)/2\), and the other words split
into free three-element orbits.  Thus
\begin{equation}\label{eq:central-commuting-orbits}
 C_n=2G_+(n).
\end{equation}
The equality of counts does not assert a canonical
two-to-one correspondence.

\medskip
\noindent\textbf{Axis and boundary incidence sums.}
The radial law can also be summed over lower-dimensional parts of the
centroid configuration.  Unlike Theorem~\ref{thm:axis-supported}, which
counts triangles supported entirely on the three main axes, the first
identity below counts incidences of arbitrary equilateral triangles at
vertices on one fixed axis.  The following counts are \emph{vertex--triangle
incidences}, so a triangle with two vertices in the specified set is
counted twice.

\begin{corollary}[Main-axis and boundary incidences]
\label{prop:axis-boundary-incidences}
Let \(N=2^n\), and let \(I_n\) be the set of words with centroids on
one fixed main axis.  Then
\begin{equation}\label{eq:axis-incidence-147590}
 \sum_{b\in I_n}\eta_n(b)
  =\sum_{b\in I_n}q_{120}(\Lambda_n(b))
  =\frac{N(N^2-2)}4
  =2^{3n-2}-2^{n-1}.
\end{equation}
This is \oeislink{https://oeis.org/A147590}{OEIS A147590}, starting at
\(n=1\).

For a triangular point patch \(\Omega_M\), let
\(\partial\Omega_M\) be the points with at least one zero
barycentric coordinate.  The sum of vertex incidences on one specified
side and on the complete boundary, respectively, is
\begin{align}\label{eq:boundary-incidence-grid}
 \sum_{(0,b,c)\in\Omega_M}D_M(0,b,c)
   &=\frac{M(M+1)(M+5)}6,\notag\\
 \sum_{v\in\partial\Omega_M}D_M(v)
   &=\frac{M(M^2+6M-1)}2.
\end{align}
In particular, the outer boundary of the upward-centroid patch
(\(M=N-1\)) has incidence sum
\begin{equation}\label{eq:upward-outer-incidence}
 \frac{(N-1)(N^2+4N-6)}2,
\end{equation}
while the boundary of the separate downward-centroid patch
(\(M=N-2\)) has sum
\begin{equation}\label{eq:downward-boundary-incidence}
 \frac{(N-2)(N^2+2N-9)}2.
\end{equation}
For \(M=0\), the boundary-incidence sum is understood as zero.
\end{corollary}

\begin{proof}
By cyclic symmetry, take \(I_n=\{i,e\}^n\).  For \(b\in I_n\), write
\(x(b)=X_n(b)\), so \(\Lambda_n(b)=(x(b),0)\), and put
\(s(b)=(-1)^{\#e(b)}\).  The digit recursion (2.3) gives
\[
\begin{array}{c|cc}
\text{extension}&x&s\\ \hline
bi&2x(b)+s(b)&s(b)\\
be&2x(b)&-s(b)
\end{array}
\]
Set \(S_n=\sum_{b\in I_n}x(b)^2\) and
\(T_n=\sum_{b\in I_n}x(b)s(b)\).  Pairing the two extensions of
 each word and writing \(x=x(b)\), \(s=s(b)\) in the summands gives
\begin{align*}
T_{n+1}
 &=\sum_{b\in I_n}\bigl((2x+s)s+(2x)(-s)\bigr)
   =\sum_{b\in I_n}1=2^n,\\
S_{n+1}
 &=\sum_{b\in I_n}\bigl((2x+s)^2+(2x)^2\bigr)
   =8S_n+4T_n+2^n.
\end{align*}
Since \(T_1=1\), it follows that \(T_n=2^{n-1}\) for \(n\ge1\),
and therefore \(S_{n+1}=8S_n+3\cdot2^n\).
A quick check gives \(S_1=1\); solving the recurrence yields
\[
 S_n=2^{3n-2}-2^{n-1}=\frac{N(N^2-2)}4.
\]
Finally, summing Corollary~\ref{cor:radial-algebraic} over the axis,
the orientation-sign terms cancel since
\(\sum_{b\in I_n}s(b)=0\).  Thus
\[
3\sum_{b\in I_n}\eta_n(b)+S_n=N(N^2-2)=4S_n,
\]
which proves \eqref{eq:axis-incidence-147590}.

On a side of \(\Omega_M\), Lemma~\ref{lem:grid-incidence} gives
\(D_M(0,b,M-b)=M+b(M-b)\), whose sum over \(0\le b\le M\)
is the first formula in \eqref{eq:boundary-incidence-grid}.
Summing the three sides and subtracting the three corners, each
counted twice and each of incidence \(M\), gives the second formula.
The two specializations follow by taking \(M=N-1\) and \(M=N-2\).
\end{proof}

\begin{samepage}
For comparison, subtracting the triangles supported entirely on the
interior copy of \(\Omega_{M-3}\) from the classical triangular-grid count
\(\binom{M+3}{4}\) (Remark~4.5) gives the number of \emph{distinct}
triangles meeting the boundary:
\begin{equation}\label{eq:triangles-touching-boundary}
 \begin{aligned}
 &\#\{T\text{ equilateral in }\Omega_M:
       T\cap\partial\Omega_M\ne\varnothing\}\\
 &\hspace{2em}=\binom{M+3}{4}-\binom M4
 =\frac{M(M^2+1)}2,
 \end{aligned}
\end{equation}
where \(\binom M4=0\) for \(M<4\).  Unlike an incidence sum,
this counts every boundary-touching triangle once.
\end{samepage}

\section*{Directions for further work}

The objects studied in this paper are best regarded as overlapping
structural predicates on \(E_n\), not as a single nested hierarchy.
Locality, product-centeredness, reflection symmetry,
multiplication-generation, and pairwise anticommutation define distinct
intersecting families.  The exact results in Sections~5--12 leave
several specific classification questions.

\medskip\noindent\textbf{Resolved multiplication rigidity.}
Question~\ref{q:nonlocal-mult} has been answered affirmatively by
Ma and Zhang \cite{MaZhang2026}: there are no nonlocal
multiplication-generated equilateral triangles at any order.
Equivalently, \(\mathcal M_n\cap X_n=\varnothing\),
and the global cyclic family exhausts \(\mathcal M_n\).
Consequently, the \emph{negative commuting column of the
multiplication-generated equilateral family} vanishes in every order
(i.e., \(C_-(n)=0\) in the signed-count notation of
Table~\ref{tab:mult}); this does not
assert that the signed centralizer sets \(C_-(b)\) of the
preceding paper are empty.  The multiplication-generated
pairwise-anticommuting count is \((4^n-(-2)^n)/6\).  The broader classification of
\emph{nonlocal product-centered} triangles and the full
pairwise-anticommuting class remains open.

\medskip\noindent\textbf{Product points and reflection.}
The local product-centered class is counted by
Theorem~\ref{thm:pc-local-count}.  For the nonlocal family the known
initial values
\[
  |P_n\cap X_n|=0,0,15,225,2280,20295,\ldots
\]
do not yet yield a structural characterization or closed formula.
The related counts \(|E_n^{\mathrm{sym}}\cap P_n|\) and
\(|E_n^{\mathrm{sym}}\cap X_n\cap P_n|\) remain open in this paper;
the order-six equality of the local and nonlocal reflection-restricted
product-centered counts (both \(813\)) deserves further investigation.

\medskip\noindent\textbf{Commutation data.}
The local pairwise-anticommuting class has the exact count
\((7^n-1)/6\); an all-order count of \(|\mathrm{AC}_n|\)
or \(|\mathrm{AC}_n\cap X_n|\) is not established.  In particular,
one may investigate \(\mathrm{AC}_n\cap P_n\) and
\(\mathrm{AC}_n\cap E_n^{\mathrm{sym}}\), or seek a joint finite-state description
for signed commutation and unsigned product data beyond the
now-classified multiplication-generated family.

\medskip\noindent\textbf{Radial refinements.}
The exact radial-incidence theorem in
Section~\ref{sec:radial-incidence} resolves the previously open
observation that total triangle incidence depends only on
\(q(\Lambda_n(b))\).  The natural next questions concern the
cardinalities of individual circular shells, closed formulas for
the pair-incidence sum \(B(K)\), and the distribution of local,
nonlocal, and pairwise-anticommuting triangles among the shells.
The local and nonlocal \emph{vertex-incidence functions} are now
known exactly from the support-sensitive formulas following
Theorem~\ref{thm:radial-incidence}; unlike total incidence, they
need not be radially homogeneous.  The exact order-four examples
there exhibit this failure.  The pairwise-anticommuting incidence
function also fails radial homogeneity already in order \(3\).
The remaining shell-distribution questions concern, for example,
triangles having two or three vertices in a given restricted subset.
The midpoint-shell formula \(\eta_n(b)=4^{n-1}-1\) and the shell
three-vertex count \(N_3(K)=|K|/3\) supply exact starting points.

\medskip\noindent\textbf{Centralizer-tile supports.}
A further restricted-triangle problem comes from the base paper's
Definition~6.4.  Fix \(b\in\Delta_n\setminus\{e_n\}\) and consider the
positive centralizer tile set \(C_{\mathrm{tiles}}(b)\), its
signed-product components \(C_+(b)\) and \(C_-(b)\), and the cyclic
images \(\gamma(C_-(b))\), \(\gamma^2(C_-(b))\).
For any such subset \(K\), how many members of \(E_n\) have all three
vertices in \(K\), and how do the local/nonlocal splits vary with \(b\)?
The last three counts are equal, since \(\gamma\) acts by a Euclidean
rotation.  The base paper already guarantees
\((4^{n-1}-1)/3\) global \(\gamma\)-triangles inside \(C_+(b)\),
but additional equilateral triangles may exist; their complete
enumeration is a distinct question.  Digit actions, orbit structures,
and stabilizers could help organize it.

\medskip\noindent\textbf{Digit-restricted supports and recurrence guidance.}
The no-\(e\) rigidity theorem in Section~8 suggests extending
the analysis to supports formed by forbidding another digit.
Such supports and their algebraic symmetries belong to a separate
project and are not used in the present proofs.  Appendix~A
illustrates how coefficient-recurrence families can guide the
search for restricted triangle classes: the translated Jacobsthal
\(r=3\) case led to the axis-fixed product-point family of
Proposition~\ref{prop:Y-axis}.  Other recurrence sequences in
the order-two generator family may suggest additional,
independently testable geometric restrictions.

\medskip\noindent\textbf{Isosceles triangles and regular hexagons.}
Mixed forward and backward digitwise cyclic actions provide natural
nonequilateral triangles.  For example, starting from \(i^{2m}\) and
cycling the letters forward at odd positions and backward at even
positions gives a three-word orbit whose centroid triangle is
isosceles with side-length ratio \(\sqrt7:\sqrt7:1\), independent
of \(m\ge1\).  More generally, the two base vertices of any strictly
isosceles centroid triangle must have the same tile orientation,
by the squared-distance residue obstruction modulo \(3\); its apex
may have the opposite orientation.  A classification by orientation
and apex angle, and the enumeration of regular centroid hexagons,
are natural questions for future work.

\appendix
\section{Order-two recurrence generators for the enumerative sequences}

Appendix~B of \cite{Dement2026} records an order-two floretion family whose
successive powers generate second-order linear recurrences under coefficient
extraction.  Several sequences arising independently from the geometric
enumeration in Sections~5--10 fit naturally into the same construction.  The
comparison can also be used in reverse: the missing \(r=3\) translate below
led to the geometric class counted in Proposition~\ref{prop:Y-axis}.  None of the
algebraic realizations in this appendix is needed for Sections~1--10.

For an order-two floretion \(Y\) and a basis word \(u\in\Delta_2\), write
\([u]Y\) for the coefficient of \(u\) in \(Y\).  Put
\begin{equation}\label{eq:Eprime}
  E'=\frac14(ie+ei+ii+jj+kk+jk+kj+ee),
\end{equation}
and, for scalars \(A,B,C\), define
\begin{equation}\label{eq:ZABC}
  Z(A,B,C)=E'(Aei+Bej+Cek).
\end{equation}
As shown in the recurrence-generator construction of
\cite[Appendix~B]{Dement2026}, direct multiplication gives
\begin{equation}\label{eq:Zcubic}
  Z^3+AZ^2+BC\,Z=0.
\end{equation}
Thus, with \(p=-A\) and \(q=-BC\), every fixed coefficient sequence
\(a_m=[u]Z^m\) satisfies
\[
  a_m=p\,a_{m-1}+q\,a_{m-2}\qquad(m\ge3).
\]
The following proposition constructs linear coefficient functionals for
which the recurrence holds already at \(m=2\), with arbitrarily prescribed
initial conditions.

\begin{proposition}[Matching the initial conditions]
Let \(a_0,a_1\) be prescribed and
\[
  a_m=p\,a_{m-1}+q\,a_{m-2}
  \qquad(m\ge2).
\]
Choose \(B,C\) with \(BC=-q\) and \(B\ne C\), put
\(Z=Z(-p,B,C)\), and define
\begin{equation}\label{eq:ell}
  \ell(Y)
   =a_0\bigl([ee]Y+[kk]Y\bigr)
    +\frac{4a_1-2pa_0}{B-C}[ij]Y.
\end{equation}
Then
\[
  \ell(Z^m)=a_m
  \qquad(m\ge0),
\]
where \(Z^0=e_2\).
\end{proposition}

\begin{proof}
Direct multiplication from \eqref{eq:Eprime}--\eqref{eq:ZABC} gives
\[
  [ee]Z=[kk]Z=\frac p4,\qquad
  [ij]Z=\frac{B-C}{4},
\]
and
\[
  [ee]Z^2=[kk]Z^2=\frac{p^2+2q}{4},
  \qquad
  [ij]Z^2=\frac{p(B-C)}4.
\]
Hence \(\ell(Z^0)=a_0\), \(\ell(Z)=a_1\), and
\(\ell(Z^2)=pa_1+qa_0\).  Equation \eqref{eq:Zcubic} propagates the same
recurrence for all higher powers.
\end{proof}

Thus the common factor \(E'\) is not tied to one named sequence: it realizes
arbitrary second-order recurrence coefficients and, through a suitable
coefficient functional, arbitrary initial conditions.

\subsection{Fibonacci, Jacobsthal, and the sequence \texorpdfstring{\(H_n\)}{Hn}}

The Fibonacci generator is already given in
\cite[Appendix~B]{Dement2026}.  With
\[
  Z_F=Z(-1,1,-1)=E'(-ei+ej-ek),
\]
one has
\begin{equation}\label{eq:ZF}
  2[ij]Z_F^m=F_m
  \qquad(m\ge0),
\end{equation}
where \(F_0=0,F_1=1\).  Consequently
\[
  2[ij](Z_F^2)^m=F_{2m},
\]
so the even-index Fibonacci sequence in Section~\ref{sec:reflection} is generated by the
square of the same order-two floretion.

For the Jacobsthal numbers \(J_0=0,J_1=1\),
\[
  J_m=J_{m-1}+2J_{m-2},
\]
take
\[
  Z_J=Z(-1,1,-2)=E'(-ei+ej-2ek).
\]
Then
\begin{equation}\label{eq:ZJ}
  -4[ii]Z_J^m=J_m
  \qquad(m\ge0).
\end{equation}

The recurrence of Theorem~\ref{thm:pc-local-count},
\[
  H_0=1,\qquad H_1=7,\qquad
  H_m=5H_{m-1}+5H_{m-2},
\]
is obtained from
\[
  Z_H=Z(-5,1,-5)=E'(-5ei+ej-5ek)
\]
by
\begin{equation}\label{eq:ZH}
  H_m=[ee]Z_H^m+[kk]Z_H^m+3[ij]Z_H^m.
\end{equation}
Thus the same order-two family that contains the Fibonacci example from the
preceding paper also contains the recurrence governing all
product-centered local triangles here.

\subsection{Translation by the identity and binomial transforms}

The order-two identity word \(e_2=ee\) commutes with every order-two
floretion.  Therefore,
for any scalar \(r\), linear coefficient functional \(\ell\), and \(m\ge0\),
\begin{equation}\label{eq:binomial}
  \ell\bigl((X+r e_2)^m\bigr)
   =\sum_{k=0}^m\binom mk r^{m-k}\ell(X^k).
\end{equation}
For a nonnegative integer \(r\), the right side is the \(r\)-fold iterate
of the ordinary binomial transform.

From \eqref{eq:ZJ} and
\[
  J_m=\frac{2^m-(-1)^m}{3},
\]
we obtain
\begin{equation}\label{eq:ZJtranslate}
  -4[ii](Z_J+r e_2)^m
   =\frac{(r+2)^m-(r-1)^m}{3}.
\end{equation}

For a subset \(D\subseteq\{i,j,k,e\}\), let \(L_n(D)\) be the class of
local triangles \(T(b,S)\) for which every common unselected digit
\(b_r\), \(r\notin S\), belongs to \(D\).

\begin{proposition}[Allowed unselected alphabets]
If \(|D|=t\), then
\begin{equation}\label{eq:LnD}
  |L_n(D)|=\frac{(t+3)^n-t^n}{3}.
\end{equation}
Consequently, with \(r=t+1\),
\begin{equation}\label{eq:LnD-ZJ}
  |L_n(D)|=-4[ii](Z_J+r e_2)^n.
\end{equation}
\end{proposition}

\begin{proof}
For fixed nonempty \(S\), selected coordinates have three possible starting
digits and unselected coordinates \(t\) choices.  Dividing
\(3^{|S|}t^{n-|S|}\) by the three cyclic representatives and summing over
\(S\ne\varnothing\) gives
\[
  \frac13\sum_{k=1}^n\binom nk3^kt^{n-k}
  =\frac{(t+3)^n-t^n}{3}.
\]
Equation \eqref{eq:LnD-ZJ} is \eqref{eq:ZJtranslate} with \(r=t+1\).
\end{proof}

The four natural choices used here are
\[
\begin{array}{c@{\qquad}c@{\qquad}l}
\toprule
t&D&L_n(D)\\
\midrule
1&\{e\}&\text{global \(\gamma\)-orbits}\\
2&\{e,f_\tau\}&L_n^{\mathrm{pp\text{-}ax}}(\tau)\\
3&\{i,j,k\}&L_n^S\\
4&\{i,j,k,e\}&L_n\\
\bottomrule
\end{array}
\]
and therefore
\begin{align}
 -4[ii](Z_J+2e_2)^m
   &=\frac{4^m-1}{3}=G_m,\label{eq:J2}\\
 -4[ii](Z_J+3e_2)^m
   &=\frac{5^m-2^m}{3}=|L_m^{\mathrm{pp\text{-}ax}}(\tau)|,\label{eq:J3}\\
 -4[ii](Z_J+4e_2)^m
   &=\frac{6^m-3^m}{3}
     =(2^m-1)3^{m-1}=|L_m^S|,\label{eq:J4}\\
 -4[ii](Z_J+5e_2)^m
   &=\frac{7^m-4^m}{3}=|L_m|.\label{eq:J5}
\end{align}
Here the geometric sequences are extended by their natural order-zero value
\(0\).  Equation \eqref{eq:J4} is precisely the fourth-binomial-transform
relation between A001045 and A016137 recorded at
\oeislink{https://oeis.org/A016137}{OEIS A016137} \cite{OEIS}.

The total equilateral-triangle count also has a compact realization:
\begin{equation}\label{eq:E-generator}
  -[ii](4Z_J+8e_2)^m
   =\frac{16^m-4^m}{12}=|E_m|.
\end{equation}
Similarly, the fixed-axis local count \eqref{eq:Aaxis-local} is generated
inside the common \(E'\)-family by
\[
  Z_A=Z(-5,2,3)=E'(-5ei+2ej+3ek),
\]
for which
\begin{equation}\label{eq:A-generator}
  -4[ij]Z_A^m=3^m-2^m=|E_m^{\mathrm{sym}}(\tau)\cap L_m|.
\end{equation}

\subsection{The central-vertex count and the same \(E'\) generator}
\label{app:central-vertex-generator}

The central-vertex count in Proposition~\ref{prop:central-vertex}
also admits a representation through the translated Jacobsthal
coefficient identity, as a difference of two translations.
For each fixed \(n\), both sides of \eqref{eq:ZJtranslate}
are polynomials in \(r\), so that identity extends from
nonnegative integers to \(r=-1\).  Using \(r=2\) and \(r=-1\)
gives, for \(n\ge1\),
\begin{align}
  C_n
    &=\frac{4^n-1}{3}
       -\frac{1-(-2)^n}{3}\notag\\
    &=4[ii]\bigl((Z_J-e_2)^n-(Z_J+2e_2)^n\bigr),
       \label{eq:cyclops-generator}
\end{align}
where \(Z_J=E'(-ei+ej-2ek)\) is the order-two element in
\eqref{eq:ZJ}.

The closed formula
\(C_n=\bigl(4^n+(-2)^n-2\cdot1^n\bigr)/3\)
has characteristic polynomial
\[
 (t-4)(t+2)(t-1)=t^3-3t^2-6t+8.
\]
Consequently it obeys the third-order recurrence
\begin{equation}\label{eq:cyclops-recurrence}
  C_n=3C_{n-1}+6C_{n-2}-8C_{n-3}
  \quad(n\ge3),\qquad C_0=0,\ C_1=0,\ C_2=6.
\end{equation}
The initial terms agree with OEIS A279260 as described in
Remark~\ref{rem:cyclops}.

\subsection{The two axis-supported branches}\label{app:axis-recurrence}

The geometric split of Theorem~\ref{thm:axis-supported} also separates the
axis-supported count into two elementary second-order sequences.  For the
equal-coordinate branch, extend the sequence naturally by \(a_0=0\); for
\(m\ge1\),
\[
  a_m=|E_m^{\mathrm{ax},=}|=2^m-1,
  \qquad
  a_m=3a_{m-1}-2a_{m-2}\quad(m\ge2).
\]
Taking \(p=3\), \(q=-2\), \(B=2\), and \(C=1\) in Proposition~A.1 gives
\[
  Z_{\mathrm{same}}
   =Z(-3,2,1)
   =E'(-3ei+2ej+ek),
\]
with
\begin{equation}\label{eq:Zsame}
  4[ij]Z_{\mathrm{same}}^m=2^m-1.
\end{equation}

For the zero-sum branch, use the naturally shifted sequence
\[
  b_m=|E_{m+1}^{\mathrm{ax},0}|=4^m-1,\qquad
  b_m=5b_{m-1}-4b_{m-2},
  \qquad b_0=0,\ b_1=3.
\]
Taking \(p=5\), \(q=-4\), \(B=4\), and \(C=1\) gives
\[
  Z_{\mathrm{zero}}
   =Z(-5,4,1)
   =E'(-5ei+4ej+ek),
\]
and
\begin{equation}\label{eq:Zzero}
  4[ij]Z_{\mathrm{zero}}^m=4^m-1.
\end{equation}
Consequently, for \(n\ge1\),
\[
  |E_n^{\mathrm{ax}}|
  =
  4[ij]Z_{\mathrm{same}}^n
  +
  4[ij]Z_{\mathrm{zero}}^{\,n-1}.
\]
The geometric factorization and the stars-and-bars enumeration in
Subsection~\ref{subsec:axis-supported} are the substantive ingredients here; the
order-two realizations simply record that the two resulting branch counts
lie in the common \(E'\)-family.

\subsection{Pairwise-anticommuting local cycles}\label{app:ac-recurrence}

The local pairwise anticommuting count of
Corollary~\ref{cor:local-ac-count} supplies another direct link between a
geometric class and the order-two recurrence family.  Put
\[
  a_m=|\mathrm{AC}_m\cap L_m|=\frac{7^m-1}{6},
  \qquad a_0=0,\quad a_1=1.
\]
Then
\[
  a_m=8a_{m-1}-7a_{m-2}\qquad(m\ge2).
\]
In Proposition~A.1 take \(p=8\), \(q=-7\), \(B=7\), and \(C=1\).  Thus
\[
  Z_{\mathrm{AC}}
   =Z(-8,7,1)
   =E'(-8ei+7ej+ek),
\]
and the initial-condition functional reduces to
\[
  \ell(Y)=\frac23[ij]Y.
\]
Therefore
\[
  \frac23[ij]Z_{\mathrm{AC}}^m
   =|\mathrm{AC}_m\cap L_m|
   =\frac{7^m-1}{6}.
\]

The new pairwise-anticommuting family supplies the complementary direction
to the inverse recurrence search discussed above: the geometric odd-\(|S|\)
criterion yields \((7^n-1)/6\) first, after which the second-order recurrence
identifies the natural generator \(Z_{\mathrm{AC}}\).

For reference:
\[
\begin{array}{ll}
\toprule
\text{order-two power family}&\text{coefficient extraction}\\
\midrule
Z_F^m&2[ij]Z_F^m=F_m\\
Z_J^m&-4[ii]Z_J^m=J_m\\
Z_H^m&([ee]+[kk]+3[ij])Z_H^m=H_m\\
(Z_J+2e_2)^m&-4[ii](Z_J+2e_2)^m=G_m\\
\bigl((Z_J-e_2)^m-(Z_J+2e_2)^m\bigr)&4[ii]\bigl((Z_J-e_2)^m-(Z_J+2e_2)^m\bigr)=C_m\\
(Z_J+3e_2)^m&-4[ii](Z_J+3e_2)^m=|L_m^{\mathrm{pp\text{-}ax}}(\tau)|\\
(Z_J+4e_2)^m&-4[ii](Z_J+4e_2)^m=|L_m^S|\\
(Z_J+5e_2)^m&-4[ii](Z_J+5e_2)^m=|L_m|\\
(4Z_J+8e_2)^m&-[ii](4Z_J+8e_2)^m=|E_m|\\
Z_A^m&-4[ij]Z_A^m=|E_m^{\mathrm{sym}}(\tau)\cap L_m|\\
Z_{\mathrm{same}}^m&4[ij]Z_{\mathrm{same}}^m=2^m-1\\
Z_{\mathrm{zero}}^m&4[ij]Z_{\mathrm{zero}}^m=4^m-1\\
Z_{\mathrm{AC}}^m&\frac23[ij]Z_{\mathrm{AC}}^m=\dfrac{7^m-1}{6}\\
\bottomrule
\end{array}
\]

The identities above do not assert that the geometric counts of Sections~5--9 are
caused by the order-two power sequences.  They do show that the comparison
can be used prospectively.  In the translated Jacobsthal family, the
already identified cases \(r=2,4,5\) left the intermediate sequence
\[
  \frac{5^n-2^n}{3}
\]
at \(r=3\).  Reading the construction backwards suggested a class with two
allowed unselected digits; the reflection-fixed alphabet
\(\{e,f_\tau\}\) then led to \(L_n^{\mathrm{pp\text{-}ax}}(\tau)\), whose count was proved
independently in Proposition~\ref{prop:Y-axis}.

This suggests a possible inverse direction for future work: recurrence
sequences generated algebraically by \(Z(A,B,C)^m\) but not presently
visible in Sections~5--9 may correspond to geometrically natural subclasses
defined by support, symmetry, completion behavior, or product-point
constraints.  Conversely, new geometric counting problems may single out
further members of the recurrence family.  No general correspondence is
claimed.

\clearpage
\subsection{An order-two generator for the main-axis incidence sequence}
\label{app:axis-incidence-generator}

The single-axis incidence identity \eqref{eq:axis-incidence-147590}
identifies a geometric realization of
\oeislink{https://oeis.org/A147590}{OEIS A147590}:
\[
 \begin{gathered}
 a_n=2^{3n-2}-2^{n-1}\quad(n\ge1),\qquad a_1=1,\ a_2=14,\\
 a_n=10a_{n-1}-16a_{n-2}\quad(n\ge3).
 \end{gathered}
\]
This recurrence has characteristic roots \(8\) and \(2\), so the
general construction of Proposition~A.1 applies.  Define
\begin{align}\label{eq:Z-axis-incidence}
 Z_{\mathrm{inc}}&=Z(-10,8,2)
   =E'(-10ei+8ej+2ek),\notag\\
 \ell_{\mathrm{inc}}(Y)&=[ee]Y+[kk]Y+6[ij]Y.
\end{align}
Taking \(p=10\), \(q=-16\), \(B=8\), \(C=2\), and the
shifted initial values \(a_1=1,a_2=14\) in
\eqref{eq:ell} gives precisely \(\ell_{\mathrm{inc}}\).  Hence
\begin{equation}\label{eq:A147590-generator}
 \boxed{\quad
 \ell_{\mathrm{inc}}\bigl(Z_{\mathrm{inc}}^{\,n-1}\bigr)
 =2^{3n-2}-2^{n-1}=\mathrm{A147590}(n),\qquad n\ge1.
 \quad}
\end{equation}
In particular, this is an order-two \emph{floretion coefficient}
realization of the axis \emph{incidence} total, rather than a new
formula for the number of distinct triangles meeting that axis.

\section{Examples}\label{app:examples}

Readers looking for a quick feel for the triangle classes may start here and
return to the definitions as needed.  Most examples below are of order
three.  The online multiplier \cite{Calculator} operates on algebra elements,
so an unordered triangle
\[
  T=\{b_1,b_2,b_3\}
\]
is entered there as \(b_1+b_2+b_3\).  The sum is only a convenient input
format for displaying the three constituent tiles; the triangle itself
remains an unordered set.  For local cycles, selected coordinates are shown
in bold blue.

\begin{example}[A global orbit with several symmetries]
Take
\[
  T=\{\sel{i}\sel{i}\sel{i},
       \sel{j}\sel{j}\sel{j},
       \sel{k}\sel{k}\sel{k}\}
    =T(iii,\{1,2,3\}).
\]
This is a nontrivial global \(\gamma\)-orbit.  It lies in the no-\(e\)
support \(S_3\), is product-centered and multiplication-generated, and is
invariant under all three main reflections.  It is also an axis-supported
triangle in the equal-coordinate branch of Theorem~\ref{thm:axis-supported},
with axis parameters \(x=y=z=7\).  Its product point is the origin, since
\(p(T)=eee\).  It is also pairwise anticommuting, and
\[
  (iii+jjj+kkk)^2=-3eee.
\]
\noindent\emph{Calculator input:} \(iii+jjj+kkk\).
\end{example}

\begin{example}[Local but not product-centered]
Let
\[
  T=\{\sel{i}ij,\sel{j}ij,\sel{k}ij\}
    =T(iij,\{1\}).
\]
This is local by construction.  The unselected coordinate \(r=2\) is
noncentral and has one selected coordinate before it, so the product-center
parity criterion fails.  Hence \(T\notin P_3\).  Nevertheless \(|S|=1\) is
odd, so \(T\in\mathrm{AC}_3\) and
\[
  (iij+jij+kij)^2=-3eee.
\]
\noindent\emph{Calculator input:} \(iij+jij+kij\).
\end{example}

\begin{example}[Local, product-centered, and reflection-symmetric]
Take
\[
  T=\{j\sel{i}\sel{i},
       j\sel{j}\sel{j},
       j\sel{k}\sel{k}\}
    =T(jii,\{2,3\}).
\]
The selected coordinates form one terminal block, so the local parity
criterion is satisfied.  Thus
\[
  C_T=Q_T=P(jee).
\]
The transposition fixing \(j\) and exchanging \(i\) and \(k\) fixes the
middle vertex and exchanges the other two, so \(T\) is symmetric about the
\(j\)-axis.  This is the upper order-three example in Figure~1(b).
Since \(|S|=2\) is even, it is not pairwise anticommuting.
Proposition~\ref{prop:local-square-image} gives
\[
  (jii+jjj+jkk)^2=-3eee-2(eii+ejj+ekk).
\]
The non-scalar part is supported on the global \(\gamma\)-orbit of \(eii\).
\noindent\emph{Calculator input:} \(jii+jjj+jkk\).
\end{example}

\begin{example}[A nonlocal equilateral triangle with \(e_2\) as a vertex]
\label{ex:central-vertex}
Take
\[
  T=\{ee,ij,ik\}.
\]
In the scaled integer coordinates of Section~2,
\[
  \Lambda_2(ee)=(0,0),\qquad
  \Lambda_2(ij)=(2,1),\qquad
  \Lambda_2(ik)=(1,-1).
\]
The squared triangular-lattice lengths
\(q(r,s)=r^2-rs+s^2\) of the three sides are all \(3\).
Thus \(T\in\mathcal V_2\) and is one of the \(C_2=6\) triangles
counted by Proposition~\ref{prop:central-vertex}.
It is necessarily nonlocal; directly, the first-position digit
triple \(e,i,i\) is neither constant nor a cyclic \(i,j,k\) triple.
\noindent\emph{Calculator input:} \(ee+ij+ik\).
\end{example}

\begin{example}[Nonlocal and not product-centered]
The lower order-three example in Figure~1(b) is
\[
  T=\{iee,iij,iji\}.
\]
It is nonlocal: in the second coordinate the digit triple is
\(\{e,i,j\}\), which is neither constant nor \(\{i,j,k\}\), so it fails the
digitwise recognition criterion of Proposition~\ref{prop:local-recognition}.
Its product word is \(p(T)=ikk\), and \(P(ikk)\) is not the Euclidean center.
\noindent\emph{Calculator input:} \(iee+iij+iji\).
\end{example}

\begin{example}[A nonlocal product-centered triangle]
Consider
\[
  T=\{iie,ejj,ekk\}.
\]
This triangle is nonlocal: already in the first coordinate the digit triple
is \(\{i,e,e\}\), which is neither constant nor \(\{i,j,k\}\).  Nevertheless
\[
  p(T)=iei,\qquad C_T=Q_T=P(iei),
\]
so \(T\in P_3\cap X_3\).  It is also symmetric about the \(i\)-axis: the
reflection fixing \(i\) fixes \(iie\) and exchanges \(ejj\) with \(ekk\).
Moreover, it is an axis-supported triangle in the zero-sum branch of
Theorem~\ref{thm:axis-supported}: its three axis parameters are
\(6,-3,-3\).  Thus this example also gives a concrete representative of
\(E_3^{\mathrm{ax},0}\).
\noindent\emph{Calculator input:} \(iie+ejj+ekk\).
\end{example}

\begin{example}[A nonlocal pairwise-anticommuting triangle]
The triangle
\[
  T=\{iji,jij,eek\}
\]
is the nonlocal example from Section~\ref{sec:signed}.  Its pairwise counts of
anticommuting positions are \(3,1,1\), so all three pairs anticommute,
while its first coordinate \(\{i,j,e\}\) rules out locality.
\noindent\emph{Calculator input:} \(iji+jij+eek\).
\end{example}

\begin{remark}[Calculator square versus exact power]
For the coefficient-\(1\) sums in Examples B.1--B.3, the Single-tab
\texttt{power 2} preview computes the exact algebraic square.  The separate
\texttt{square} transform intentionally normalizes the current input before
squaring; consequently it displays coefficients four times as large for
these examples.  Thus B.1 and B.2 appear there as \(-12eee\), while B.3
appears as \(-12eee-8eii-8ejj-8ekk\).  The mathematical identities above
are the unnormalized coefficient-\(1\) identities.
\end{remark}

\section{Sequence index}\label{app:sequence-index}

This appendix collects the named or explicitly enumerated integer sequences
that play a structural role in the paper.  It is intended as a quick index:
the proofs and qualifications remain in the \mbox{corresponding sections} of the
paper.  Unless stated otherwise, the displayed initial terms use
\(n=1,2,\ldots\).  A dash in the OEIS column means only that no OEIS
identification is asserted in this paper; it is not a claim that no such
entry exists.  OEIS identifications are understood with the index shifts
specified in the \mbox{corresponding sections}.  Likewise, ``enumerated through
\(6\)'' records finite-order data rather than a conjectured closed formula.

Two sequences previously left without an identification are now listed under
OEIS A400512 (product-centered local triangles) and A400515
(nonlocal triangles).  The new central-vertex count is A279260
with an index shift; see Proposition~\ref{prop:central-vertex}.
The main-axis and patch-boundary incidence sequences follow from
Proposition~\ref{prop:axis-boundary-incidences}.

\subsection*{Completion, locality, product points, and reflection}

\begingroup
\scriptsize
\setlength{\tabcolsep}{2.7pt}
\renewcommand{\arraystretch}{1.14}
\rowcolors{2}{black!2}{white}
\begin{longtable}{
  >{\raggedright\arraybackslash}p{0.14\textwidth}
  >{\raggedright\arraybackslash}p{0.22\textwidth}
  >{\raggedright\arraybackslash}p{0.18\textwidth}
  >{\raggedright\arraybackslash}p{0.27\textwidth}
  >{\centering\arraybackslash}p{0.11\textwidth}}
\caption{Completion, locality, product-point, and reflection sequences.}
\label{tab:sequence-index-main}\\
\rowcolor{orange!10}
\textbf{Name} & \textbf{Description} & \textbf{First terms} &
\textbf{Formula / status} & \textbf{OEIS}\\
\toprule
\endfirsthead
\rowcolor{orange!10}
\multicolumn{5}{c}{\tablename\ \thetable\ (continued)}\\
\rowcolor{orange!10}
\textbf{Name} & \textbf{Description} & \textbf{First terms} &
\textbf{Formula / status} & \textbf{OEIS}\\
\toprule
\endhead
\bottomrule
\endfoot

\(A_0=A_2=B_2\) &
same-orientation pairs with \(0\) or \(2\) completions; equivalently all
pairs with \(2\) completions &
\(0,\allowbreak12,\allowbreak240,\allowbreak4032,\allowbreak65280,\allowbreak1047552\) &
Exact: \(N^2(N^2-4)/16,\ N=2^n\). &
-- \\

\(A_1=B_1\) &
same-orientation pairs with one completion; equivalently all pairs with one
completion &
\(3,\allowbreak36,\allowbreak528,\allowbreak8256,\allowbreak131328,\allowbreak2098176\) &
Exact: \(N^2(N^2+2)/8\). &
\oeislink{https://oeis.org/A233196}{A233196} \\

\(B_0\) &
all unordered pairs with no completion &
\(3,\allowbreak72,\allowbreak1248,\allowbreak20352,\allowbreak327168,\allowbreak5240832\) &
Exact: \(N^2(5N^2-8)/16\). &
-- \\

\(|E_n|\) &
all equilateral centroid triangles &
\(1,\allowbreak20,\allowbreak336,\allowbreak5440,\allowbreak87296,\allowbreak1397760\) &
Exact: \(4^n(4^n-1)/12\). &
\oeislink{https://oeis.org/A166984}{A166984} \\

\(|L_n|\) &
local synchronized \(\gamma\)-cycles &
\(1,\allowbreak11,\allowbreak93,\allowbreak715,\allowbreak5261,\allowbreak37851\) &
Exact: \((7^n-4^n)/3\). &
\oeislink{https://oeis.org/A016150}{A016150} \\

\(|X_n|\) &
nonlocal equilateral triangles &
\(0,\allowbreak9,\allowbreak243,\allowbreak4725,\allowbreak82035,\allowbreak1359909\) &
Exact: \((16^n+3\cdot4^n-4\cdot7^n)/12\). &
\oeislink{https://oeis.org/A400515}{A400515} \\

\(C_n=|\mathcal V_n|\) &
equilateral triangles having the central word \(e_n\) as a vertex
(all nonlocal) &
\(0,\allowbreak6,\allowbreak18,\allowbreak90,\allowbreak330,\allowbreak1386\) &
Exact: \((4^n+(-2)^n-2)/3\);
Proposition~\ref{prop:central-vertex} and Appendix~A. &
\oeislink{https://oeis.org/A279260}{A279260} (shift) \\

\(\sum_{b\in I_n}\eta_n(b)\) &
vertex--triangle incidences on one fixed main axis &
\(1,\allowbreak14,\allowbreak124,\allowbreak1016,\allowbreak8176,\allowbreak65504\) &
Exact: \(2^{3n-2}-2^{n-1}\); Appendix~A generator. &
\oeislink{https://oeis.org/A147590}{A147590} \\

\(B_n^{\mathrm{up}}\) &
vertex--triangle incidences along the entire outer upward-patch boundary &
\(3,\allowbreak39,\allowbreak315,\allowbreak2355,\allowbreak17763,\allowbreak136899\) &
Exact: \((N-1)(N^2+4N-6)/2\), \(N=2^n\). &
-- \\

\(B_n^{\mathrm{down}}\) &
vertex--triangle incidences along the downward-patch boundary &
\(0,\allowbreak15,\allowbreak213,\allowbreak1953,\allowbreak16185,\allowbreak130665\) &
Exact: \((N-2)(N^2+2N-9)/2\), \(N=2^n\). &
-- \\

\(T_n^{\mathrm{up},\partial}\) &
distinct triangles with at least one vertex on the outer upward-patch boundary &
\(1,\allowbreak15,\allowbreak175,\allowbreak1695,\allowbreak14911,\allowbreak125055\) &
Exact: \((N-1)((N-1)^2+1)/2\). &
-- \\

\(T_n^{\mathrm{down},\partial}\) &
distinct triangles with at least one vertex on the downward-patch boundary &
\(0,\allowbreak5,\allowbreak111,\allowbreak1379,\allowbreak13515,\allowbreak119195\) &
Exact: \((N-2)((N-2)^2+1)/2\). &
-- \\

\(|E_n^{\mathrm{ax}}|\) &
equilateral triangles supported on the union of the three main axes &
\(1,\allowbreak6,\allowbreak22,\allowbreak78,\allowbreak286,\allowbreak1086\) &
Exact: \(4^{n-1}+2^n-2\). &
-- \\

\(|E_n^{\mathrm{ax},=}|\) &
equal-coordinate branch \(x=y=z\) of the axis-supported class &
\(1,\allowbreak3,\allowbreak7,\allowbreak15,\allowbreak31,\allowbreak63\) &
Exact: \(2^n-1\). &
\oeislink{https://oeis.org/A000225}{A000225} \\

\(|E_n^{\mathrm{ax},0}|\) &
zero-sum branch \(x+y+z=0\) of the axis-supported class &
\(0,\allowbreak3,\allowbreak15,\allowbreak63,\allowbreak255,\allowbreak1023\) &
Exact: \(4^{n-1}-1\). &
\oeislink{https://oeis.org/A024036}{A024036} (shift) \\

\(|P_n\cap L_n|\) &
product-centered local triangles &
\(1,\allowbreak8,\allowbreak57,\allowbreak373,\allowbreak2342,\allowbreak14343\) &
Exact: \((H_n-4^n)/3\), with \(H_n=5H_{n-1}+5H_{n-2}\). &
\oeislink{https://oeis.org/A400512}{A400512} \\

\(|P_n|\) &
all product-centered equilateral triangles &
\(1,\allowbreak8,\allowbreak72,\allowbreak598,\allowbreak4622,\allowbreak34638\) &
Enumerated through \(n=6\); no closed formula asserted. &
-- \\

\(|P_n\cap X_n|\) &
nonlocal product-centered triangles &
\(0,\allowbreak0,\allowbreak15,\allowbreak225,\allowbreak2280,\allowbreak20295\) &
Enumerated through \(n=6\); no closed formula asserted. &
-- \\

\(|L_n^{\mathrm{pp\text{-}ax}}(\tau)|\) &
local triangles whose product point lies on one fixed main axis &
\(1,\allowbreak7,\allowbreak39,\allowbreak203,\allowbreak1031,\allowbreak5187\) &
Exact: \((5^n-2^n)/3\). &
\oeislink{https://oeis.org/A016127}{A016127} \\

\(|E_n^{\mathrm{sym}}(\tau)|\) &
triangles symmetric about one fixed main axis &
\(1,\allowbreak8,\allowbreak40,\allowbreak176,\allowbreak736,\allowbreak3008\) &
Exact: \(3\cdot4^{n-1}-2^n\). &
\oeislink{https://oeis.org/A165665}{A165665} \\

\(|E_n^{\mathrm{sym}}(\tau)\cap L_n|\) &
local triangles symmetric about one fixed main axis &
\(1,\allowbreak5,\allowbreak19,\allowbreak65,\allowbreak211,\allowbreak665\) &
Exact: \(3^n-2^n\). &
\oeislink{https://oeis.org/A001047}{A001047} \\

\(|E_n^{\mathrm{sym}}(\tau)\cap X_n|\) &
nonlocal triangles symmetric about one fixed main axis &
\(0,\allowbreak3,\allowbreak21,\allowbreak111,\allowbreak525,\allowbreak2343\) &
Exact: \(3\cdot4^{n-1}-3^n\). &
-- \\

\(|E_n^{\mathrm{sym}}(\tau)\cap L_n\cap P_n|\) &
product-centered local triangles symmetric about one fixed main axis &
\(1,\allowbreak4,\allowbreak13,\allowbreak39,\allowbreak112,\allowbreak313\) &
Exact: \(F_{2n+2}-2^n\). &
\oeislink{https://oeis.org/A105693}{A105693} \\

\(|E_n^{\mathrm{sym}}|\) &
triangles symmetric about at least one main axis &
\(1,\allowbreak18,\allowbreak106,\allowbreak498,\allowbreak2146,\allowbreak8898\) &
Exact: \(9\cdot4^{n-1}-5\cdot2^n+2\). &
-- \\

\(|E_n^{\mathrm{sym}}\cap L_n|\) &
local triangles symmetric about at least one main axis &
\(1,\allowbreak9,\allowbreak43,\allowbreak165,\allowbreak571,\allowbreak1869\) &
Exact: \(3^{n+1}-5\cdot2^n+2\). &
\oeislink{https://oeis.org/A281773}{A281773} \\

\(|E_n^{\mathrm{sym}}\cap X_n|\) &
nonlocal triangles symmetric about at least one main axis &
\(0,\allowbreak9,\allowbreak63,\allowbreak333,\allowbreak1575,\allowbreak7029\) &
Exact: \(9\cdot4^{n-1}-3^{n+1}\). &
-- \\

\(|E_n^{\mathrm{sym}}\cap L_n\cap P_n|\) &
product-centered local triangles with a main-axis symmetry &
\(1,\allowbreak6,\allowbreak25,\allowbreak87,\allowbreak274,\allowbreak813\) &
Exact: \(3F_{2n+2}-5\cdot2^n+2\). &
-- \\

\(|E_n^{\mathrm{sym}}\cap P_n|\) &
product-centered triangles with a main-axis symmetry &
\(1,\allowbreak6,\allowbreak34,\allowbreak138,\allowbreak490,\allowbreak1626\) &
Enumerated through \(n=6\); no closed formula asserted. &
-- \\

\(|E_n^{\mathrm{sym}}\cap X_n\cap P_n|\) &
nonlocal product-centered triangles with a main-axis symmetry &
\(0,\allowbreak0,\allowbreak9,\allowbreak51,\allowbreak216,\allowbreak813\) &
Enumerated through \(n=6\); no closed formula asserted. &
-- \\

\end{longtable}
\endgroup

\subsection*{No-\texorpdfstring{\(e\)}{e} support, multiplication generation, and anticommutation}

\begingroup
\scriptsize
\setlength{\tabcolsep}{2.7pt}
\renewcommand{\arraystretch}{1.14}
\rowcolors{2}{black!2}{white}
\begin{longtable}{
  >{\raggedright\arraybackslash}p{0.16\textwidth}
  >{\raggedright\arraybackslash}p{0.21\textwidth}
  >{\raggedright\arraybackslash}p{0.18\textwidth}
  >{\raggedright\arraybackslash}p{0.26\textwidth}
  >{\centering\arraybackslash}p{0.11\textwidth}}
\caption{No-\texorpdfstring{\(e\)}{e}, multiplication-generated, and anticommutation sequences.}
\label{tab:sequence-index-signed}\\
\rowcolor{orange!10}
\textbf{Name} & \textbf{Description} & \textbf{First terms} &
\textbf{Formula / status} & \textbf{OEIS}\\
\toprule
\endfirsthead
\rowcolor{orange!10}
\multicolumn{5}{c}{\tablename\ \thetable\ (continued)}\\
\rowcolor{orange!10}
\textbf{Name} & \textbf{Description} & \textbf{First terms} &
\textbf{Formula / status} & \textbf{OEIS}\\
\toprule
\endhead
\bottomrule
\endfoot

\(|L_n^S|=|E_n^S|\) &
equilateral triangles on the no-\(e\) support \(S_n\) &
\(1,\allowbreak9,\allowbreak63,\allowbreak405,\allowbreak2511,\allowbreak15309\) &
Exact: \((2^n-1)3^{n-1}\). &
\oeislink{https://oeis.org/A016137}{A016137} \\

\(C_n^S\) &
product-centered local triangles in \(S_n\) &
\(1,\allowbreak6,\allowbreak36,\allowbreak189,\allowbreak972,\allowbreak4860\) &
Exact: \((F_{n+2}-1)3^{n-1}\). &
\oeislink{https://oeis.org/A397105}{A397105} \\

\(G_n=|\mathcal M_n\cap L_n|\) &
local multiplication-generated triangles; nontrivial global \(\gamma\)-orbits &
\(1,\allowbreak5,\allowbreak21,\allowbreak85,\allowbreak341,\allowbreak1365\) &
Exact: \((4^n-1)/3\). &
\oeislink{https://oeis.org/A002450}{A002450} \\

\(G_+(n)\) &
even-\(\nu\) commuting part of the global-\(\gamma\) family &
\(0,\allowbreak3,\allowbreak9,\allowbreak45,\allowbreak165,\allowbreak693\) &
Exact: \((4^n+(-2)^n-2)/6\). &
-- \\

\(G_{\mathrm{anti}}(n)
 =|\mathcal M_n\cap L_n\cap\mathrm{AC}_n|\) &
pairwise-anticommuting part of the local multiplication-generated family &
\(1,\allowbreak2,\allowbreak12,\allowbreak40,\allowbreak176,\allowbreak672\) &
Exact: \((4^n-(-2)^n)/6=2^{n-1}J_n\). &
\oeislink{https://oeis.org/A003683}{A003683} \\

\(|\mathcal M_n|\) &
all multiplication-generated equilateral triangles &
\(1,\allowbreak5,\allowbreak21,\allowbreak85,\allowbreak341,\allowbreak1365\) &
Exact for every \(n\): \((4^n-1)/3\), by Ma--Zhang \cite{MaZhang2026}. &
-- \\

\(C_+\) in Table~\ref{tab:mult} &
positive commuting multiplication-generated column &
\(0,\allowbreak3,\allowbreak9,\allowbreak45,\allowbreak165,\allowbreak693\) &
Exact: \((4^n+(-2)^n-2)/6\), by Ma--Zhang. &
-- \\

\(C_-\) in Table~\ref{tab:mult} &
negative commuting multiplication-generated column &
\(0,\allowbreak0,\allowbreak0,\allowbreak0,\allowbreak0,\allowbreak0\) &
Exact: zero for every \(n\), by Ma--Zhang. &
-- \\

\(|\mathcal M_n\cap\mathrm{AC}_n|\) &
pairwise-anticommuting multiplication-generated triangles &
\(1,\allowbreak2,\allowbreak12,\allowbreak40,\allowbreak176,\allowbreak672\) &
Exact for every \(n\): \((4^n-(-2)^n)/6\), by Ma--Zhang. &
-- \\

\(|\mathcal M_n\cap X_n|\) &
nonlocal multiplication-generated triangles &
\(0,\allowbreak0,\allowbreak0,\allowbreak0,\allowbreak0,\allowbreak0\) &
Exactly zero for every \(n\), by Ma--Zhang \cite{MaZhang2026}. &
-- \\

\(|\mathrm{AC}_n|\) &
all pairwise-anticommuting equilateral triangles &
\(1,\allowbreak8,\allowbreak96,\allowbreak928,\allowbreak13472,\allowbreak183936\) &
Enumerated through \(n=6\); no closed formula asserted; see
Remark~\ref{rem:ac-order2}. &
-- \\

\(|\mathrm{AC}_n\cap L_n|\) &
local pairwise-anticommuting triangles &
\(1,\allowbreak8,\allowbreak57,\allowbreak400,\allowbreak2801,\allowbreak19608\) &
Exact: \((7^n-1)/6\). &
\oeislink{https://oeis.org/A023000}{A023000} \\

\(|\mathrm{AC}_n\cap X_n|\) &
nonlocal pairwise-anticommuting triangles &
\(0,\allowbreak0,\allowbreak39,\allowbreak528,\allowbreak10671,\allowbreak164328\) &
Enumerated through \(n=6\); no closed formula asserted; see
Remark~\ref{rem:ac-order2}. &
-- \\

\end{longtable}
\endgroup

\subsection*{Auxiliary and classical sequences}

\begingroup
\scriptsize
\setlength{\tabcolsep}{3pt}
\renewcommand{\arraystretch}{1.14}
\rowcolors{2}{black!2}{white}
\begin{longtable}{
  >{\raggedright\arraybackslash}p{0.13\textwidth}
  >{\raggedright\arraybackslash}p{0.23\textwidth}
  >{\raggedright\arraybackslash}p{0.18\textwidth}
  >{\raggedright\arraybackslash}p{0.25\textwidth}
  >{\centering\arraybackslash}p{0.11\textwidth}}
\caption{Auxiliary and classical sequences used in the paper.}
\label{tab:sequence-index-aux}\\
\rowcolor{orange!10}
\textbf{Name} & \textbf{Role} & \textbf{First terms} &
\textbf{Definition} & \textbf{OEIS}\\
\toprule
\endfirsthead
\rowcolor{orange!10}
\multicolumn{5}{c}{\tablename\ \thetable\ (continued)}\\
\rowcolor{orange!10}
\textbf{Name} & \textbf{Role} & \textbf{First terms} &
\textbf{Definition} & \textbf{OEIS}\\
\toprule
\endhead
\bottomrule
\endfoot

\(t_k=\binom{k+2}{4}\) &
all equilateral triangles in the \(k\)-row triangular point set \(T_k\) &
\(0,\allowbreak1,\allowbreak5,\allowbreak15,\allowbreak35,\allowbreak70,\ldots\)
for \(k=1,2,\ldots\) &
Classical triangular-grid count used in Remark~4.5. &
\oeislink{https://oeis.org/A000332}{A000332} (shift) \\

\(F_n\) &
Fibonacci counts in Sections~6--8 and reflection symmetry &
\(0,\allowbreak1,\allowbreak1,\allowbreak2,\allowbreak3,\allowbreak5,\allowbreak8,\ldots\) &
\(F_0=0,\ F_1=1,\ F_n=F_{n-1}+F_{n-2}\). &
\oeislink{https://oeis.org/A000045}{A000045} \\

\(J_n\) &
Jacobsthal sequence in Appendix~A and \(G_{\mathrm{anti}}(n)=2^{n-1}J_n\) &
\(0,\allowbreak1,\allowbreak1,\allowbreak3,\allowbreak5,\allowbreak11,\allowbreak21,\ldots\) &
\(J_0=0,\ J_1=1,\ J_n=J_{n-1}+2J_{n-2}\). &
\oeislink{https://oeis.org/A001045}{A001045} \\

\(H_n\) &
two-state recurrence controlling \(|P_n\cap L_n|\) &
\(1,\allowbreak7,\allowbreak40,\allowbreak235,\allowbreak1375,\allowbreak8050,\allowbreak47125,\ldots\) &
\(H_0=1,\ H_1=7,\ H_n=5H_{n-1}+5H_{n-2}\). &
-- \\

\end{longtable}
\endgroup

\section*{Statements and Declarations}

\noindent\textbf{Acknowledgment of subsequent work.}
The author thanks Haobo Ma and Wenlin Zhang for communicating their
all-order proof of the multiplication-rigidity result posed as
Question~\ref{q:nonlocal-mult}.  The author also thanks Michael de Vlieger
for his help linking results of this paper to OEIS entries.

\medskip
\noindent\textbf{Funding.}
The author received no financial support for the research, authorship, or
publication of this article.

\medskip
\noindent\textbf{Conflict of interest.}
The author declares that he has no conflict of interest.

\medskip
\noindent\textbf{Use of generative AI.}
During the preparation and revision of this manuscript, the author used
generative AI tools for language editing, structural suggestions, and
exploratory mathematical discussion, including searching for relevant
literature.  All suggestions incorporated into the present manuscript were
critically assessed and mathematically worked through by the author.  The
author takes full responsibility for the present content of the manuscript.

\enlargethispage{7\baselineskip}

\end{document}